\documentclass[11pt,a4paper]{amsart}
\usepackage[english]{babel}
\usepackage{amsmath,amsfonts,amssymb,relsize,amsthm}
\usepackage{dsfont}  
\usepackage{esint}
\usepackage{graphicx}
\usepackage{subcaption}
\usepackage{subcaption}

  \usepackage{hyperref} %%Warning: when you first run your tex
 \usepackage[active]{srcltx} %%allows you to jump from your editor
 \usepackage{color,soul} %%allows you to use the \st{...} command, which
\usepackage{cleveref}
\usepackage{mathtools}
\usepackage[left=2.5cm,right=2.5cm,top=2.5cm,bottom=2.5cm]{geometry}
\makeatletter
\newtheorem*{rep@theorem}{\rep@title}
\newcommand{\newreptheorem}[2]{%
\newenvironment{rep#1}[1]{%
 \def\rep@title{#2 \ref{##1}}%
 \begin{rep@theorem}}%
 {\end{rep@theorem}}}
\makeatother

\newtheorem{theorem}{Theorem}[section]

\newtheorem{lemma}[theorem]{Lemma}
\newtheorem{proposition}[theorem]{Proposition}
\newtheorem{hyp}{Hypothesis}[section]
\newtheorem*{ocp}{Optimal Control Problem}
\newtheorem*{rcp}{Reachability Problem}
\theoremstyle{definition}

\newtheorem{remark}{Remark}
\newtheorem{propri}{Property}

\numberwithin{equation}{section}

\newreptheorem{theorem}{Theorem}
\newreptheorem{proposition}{Proposition}
\newreptheorem{lemma}{Lemma}
\newreptheorem{corollary}{Corollary}

\def\O{{\Omega}}
\def\o{{\omega}}
\def\eps{{\varepsilon}}

\def\h{{\mathcal{H}}}

\def\k{{\mathcal{K}}}
\def\l{{\mathcal{L}}}

\def\r{{\mathcal{R}}}
\def\s{{\mathcal{S}}}

\def\M{{\mathcal{M}}}
\def\D{{\mathcal{D}}}
\def\U{{\mathcal{U}}}
\def\F{{\mathcal{F}}}
\def\T{{\mathcal{T}}}

\def\R{{\mathbb{R}}}

\def\N{{\mathbb{N}}}
\def\Q{{\mathbb{Q}}}

\newcommand{\opk}[1]{\k[{#1}]}

\newcommand{\opr}[1]{\r[{#1}]}
\newcommand{\ops}[2]{\s_{_{#2}}[{#1}]}

\DeclareMathOperator{\supp}{supp}
\newcommand{\vphi}{\varphi}
\newtheorem{claim}[theorem]{Claim}

\usepackage[breakable,skins,theorems]{tcolorbox}

\definecolor{b}{rgb}{0.00000,0.44700,0.74100}%
\definecolor{o}{rgb}{0.85000,0.32500,0.09800}%
\definecolor{y}{rgb}{0.92900,0.69400,0.12500}%
\definecolor{g}{rgb}{0.46600,0.67400,0.18800}%
\definecolor{p}{rgb}{0.49400,0.18400,0.55600}%
\definecolor{inrae}{RGB}{0,163,166}
\definecolor{k}{rgb}{0,0,0}%
\definecolor{gray}{rgb}{0.75,0.75,0.75}%
\definecolor{pink}{rgb}{1,0.7529,0.7961}%
\definecolor{r}{rgb}{0.6350,0.0780,0.1840}%
\definecolor{cyan}{rgb}{0.3010,0.7450,0.9330}%
\hypersetup{colorlinks = true, linkcolor=b, citecolor=o}

\newtcolorbox[]{blockboxbis}[2][]{%
	colback=r!5,%
	colframe=r!75!k,%orange!75!red!80!black,%
	colbacktitle=white,%
	coltitle=r!75!k,%orange!75!red, 
	top=7pt,
	fonttitle=\bfseries,%
	title={#2},%
	enhanced,%
}
\newenvironment{formula}[1]{\begin{equation}\label{eq:#1}}
                       {\end{equation}\noindent}

\def\Fi#1{\begin{formula}{#1}}
\def\Ff{\end{formula}\noindent}

\setstcolor{red}

\def\ds{\displaystyle}

\author{Claudia Alvarez-Latuz}
\address{Avignon Universit\'e, LMA, UPR 2151, Avignon, France}
\address{UR 546 Biostatistique et Processus  Spatiaux, INRAE, Domaine St Paul Site Agroparc, F-84000 Avignon, France,}
\email{claudia.alvarez-latuz@univ-avignon.fr}
\author{T\'erence Bayen} 
\address{Avignon Universit\'e, LMA, UPR 2151, Avignon, France}
\email{terence.bayen@univ-avignon.fr}
\author{J\'er\^ome Coville}
\address{UR 546 Biostatistique et Processus  Spatiaux, INRAE, Domaine St Paul Site Agroparc, F-84000 Avignon, France}
\address{ICJ UMR 5208 - Universite Claude Bernard Lyon 1, Campus de la Doua, 69622 Villeurbanne, France}
\email{jerome.coville@inrae.fr} 

\title[Target controllability for a  minimum time in chemostat model]{ Target controllability for a minimum time problem in a trait-structured chemostat model}

\begin{document}
\begin{abstract}
In this paper, we consider a minimum time control problem governed by a trait-structured chemostat model including mutation and one limiting substrate. 
Our first main result proves the well-posedness of the control-to-state mapping.  
We subsequently analyze the class of {\it{auxostat-type controls}}, feedback laws designed to regulate substrate concentration, and prove that the corresponding solutions converge to a stationary state of the system. 
These convergence results are used to show the reachability of a target set corresponding to the selection of a population with a low weighted averaged half-saturation constant. Finally, we show the existence of an optimal control for the minimum time problem associated with reaching the target set. These theoretical findings are completed by numerical simulations. 
 \end{abstract}
 
\maketitle

\section{Introduction}{\label{sec-intro}}

\subsection{General context}
Continuous culture can be used to select through competition specific microorganisms of interest from a pool of species initially present or appearing by mutation. Since its introduction in the fifties by Monod \cite{Monod1942,Monod1950} and Novik and Szilard \cite{Novick1950}, the chemostat system has found widespread applications over the past 75 years in a variety of contexts \cite{Balagadde2005,dykhuizen1993,Harmand2017,Keevil2001,Smith1995}. In particular, it has provided valuable insights in microbiology into the adaptive mechanisms involved in responses to controlled environments (see, {\it{e.g.}},  \cite{Cox1974,dykhuizen1993,Ferenci1999,Ferenci2006,Gresham2015,Hoskisson2005}). 
Although chemostat systems are for most of them designed to study a finite number of population, continuous versions of it with infinitely many type of population have also been introduced, see  \cite{Champagnat2011,Diekmann2005,Leman2015,Mirrahimi2012a,Perthame2007}.
In this spirit we consider here the growth of a population structured by a phenotypical trait $z$ living in a limited substrate environment,  all traits competing with each other through the access of the resource. 
A key aspect of the model is that 
phenotypical changes in the population are taken into account through the inclusion of mutation. 
%the description of possible mutation. 
Following the modeling approach proposed by \cite{Champagnat2014a,Champagnat2011,Coville2012a,Diekmann2005,Leman2015,Mirrahimi2012a,Perthame2007}, the  dynamics of this population can be represented by the following set of integro-differential equations: 
\begin{equation}\label{main}
\begin{dcases}
\partial_t f(t,z)=(\mu(s(t),z)- u(t))f(t,z) +\alpha \Delta_z f(t,z) \quad &\text{ for   }\quad t>0, z\in\O,\\
\dot s(t)=-\int_{\O}\mu(s(t),z')f(t,z')\,dz' +u(t)(s_{in}-s(t))\quad &\text{ for   }\quad t>0, \\
\partial_{\vec{n}} f(t,z)=0 \quad &\text{ for   }\quad t>0, z\in\partial\O,\\
s(0)= s_{0}>0,\quad f(0,z)=f_0(z)\ge 0 \quad &\text{ for   }\quad z\in\O,
\end{dcases}
\end{equation}
where the function $s$ is the substrate concentration and the function $f$ represents the density of the population under consideration growing on the substrate. Here, the constant $s_{in}>0$ is the input substrate concentration, $s_0$ is the initial substrate concentration, $f_0$, the initial density of the population, $\mu$ is the growth function (or kinetics) and  the function $u\ge 0$ is the dilution rate which serves as the control input for the system.  The distinction between the different species is made through the function $\mu$ which depends on the trait $z\in\O$, where $\Omega$ is a bounded 
domain of $\R^d$ with smooth boundary and unit normal vector $\vec{n}$.  
The mutation process, which can be interpreted as movement in the trait space, is here considered to be driven by diffusion and is therefore modeled by the Laplace operator, where  $\alpha\ge 0$ represents the mutation rate.  Additionally, the Neumann boundary condition reflects that the population does not convert into or from traits beyond the limits of the domain.  
This work is devoted to the analysis of~\eqref{main}, including existence and uniqueness of solutions, stabilization properties, reachability of a given target set, and the associated minimum-time control problem.

\subsection{Related models} 
In the literature, numerous results have been established within the framework of adaptive dynamics, considering both deterministic and stochastic settings. These include findings on the convergence of solutions to stationary states, as well as descriptions of the asymptotic limits as the mutation parameter $\alpha\downarrow 0$. Given the extensive scope of existing work, we highlight here a selection of relevant contributions. First, the article \cite{Champagnat2010} investigates an extension of the competitive Lotka-Volterra system and proves convergence to a unique stable equilibrium point. Note also that concentration phenomena in the limit of vanishing mutation rates are examined in \cite{Champagnat2014a,Champagnat2011}, in both deterministic and stochastic contexts. Another relevant contribution is the analysis of a non-local Lotka-Volterra mutation-competition model in \cite{Lorz2011}. 
In \cite{Diekmann2005}, a selection-mutation model with asymptotically small mutation parameter is studied based on the HJB equation.
Extinction phenomena in a stochastic chemostat model are addressed in \cite{Collet2013}, while \cite{Mirrahimi2012a} studies a trait-structured chemostat-type model and provides sufficient conditions for population survival. The relationship between deterministic and stochastic formulations of the chemostat is examined in \cite{campillo2011stochastic}. 
Finally, it is worth mentioning \cite{djidjou2025evolutionary} which studies a related model in which the mutation term differs : whereas we consider a Laplacian operator, mutation in \cite{djidjou2025evolutionary} is modeled through a kernel function $K(z,z')$ representing the probability density of the phenotype $z'$ to be a mutant of (initial) phenotype $z$ (which is a non-local operator). In addition, our work addresses optimal control questions. 

\subsection{A minimum time control problem}
While existence and uniqueness of solutions in our setting follows from rather classical arguments, much less is known about controllability and optimal control, particularly for model \eqref{main}.  To study this, let us define an objective function. The question is whether one can  design an optimal control strategy, {\it{i.e.}}, an admissible control function $u$, optimizing a desired objective function. 
In the context of~\eqref{main}, this objective raises significant challenges. The main aim of this article is therefore to clarify whether the system~\eqref{main} can be controlled, with the following minimum-time control problem in mind:
\begin{ocp} 
    Let $\mathcal{F}$ be a set of initial condition for \eqref{main} together with a 
    target set $\T_0$. Given $(s_0,f_0)\in \mathcal{F}$, the goal is to find an optimal control $u$ such that the solution to \eqref{main} 
     reaches the target set  $\T_0$ in minimal time and stays in it for all subsequent times,  {\it{i.e.}}, 
    \begin{equation}{\label{OCP0}}\inf_{u\in\U} T_u \quad {\rm{s.t.}} \quad f(t,\cdot)\in \T_0 \quad\text{ {\rm{for all}} }\, t\ge T_u,\end{equation}
where $\mathcal{U}$ is the set of admissible controls and $f$ the solution to \eqref{main} for the control $u$. 
\end{ocp}

Optimization of the chemostat system, particularly through minimum-time problems, has many practical applications in bioprocesses, such as wastewater treatment and the selection of species of interest from a mixed population (see, {\it{e.g.}},  \cite{bayen2013optimal,Bayen2017,djema2025optimal}). Other criteria, such as productivity, can also be considered, as in~\cite{Bayen2022}. In the present work, we extend this line of inquiry to model \eqref{main}, which incorporates a trait.  
In particular, we choose the target set $\mathcal{T}_0$ following the approach of \cite{Bayen2017}, with the specific biological aim of selecting populations of interest (which is a common objective in the control of bioprocesses). 
Minimum time problems in infinite-dimensional settings (see, {\it{e.g.}}, \cite{fattorini1999infinite,loheac2017minimal,wang2012equivalence,wang2021minimal}) are generally more challenging than their finite-dimensional counterparts. In the latter, the application of the Pontryagin Maximum Principle and related techniques is standard, whereas their extension to infinite dimensions faces substantial obstacles. 
One important aspect of problem \eqref{OCP0} is the nonlinearity of the state equation, in contrast with related works.
Another distinctive feature of problem \eqref{OCP0} is the requirement that the state, upon first reaching the target, must stay within it (see also \cite{wang2021minimal}). This “reach-and-stay" condition is more demanding than the classical reachability requirement, yet we will show that it can be achieved under suitable assumptions.

 \subsection{Notation and assumptions} 
The following notations will be used throughout the paper.
\begin{itemize}
    \item The notations $\R_+$ and $\R_+^*$ refer to the intervals $[0,+\infty)$ and $(0,+\infty)$ respectively.
    \item The set $\O$ is a bounded domain in $\R^d$ ($d\geq 1$) with a smooth boundary (of class $C^1$) ; $\partial \O$ and $\text{int}(\Omega)$ denote respectively its boundary and its interior. 
    \item Given a Borel set in $\R^k$, for $k\geq 1$, $|A|$ and $\mathds{1}_A$ stand for 
    the Lebesgue measure of $A$ and the indicator function of $A$, respectively. 
    \item $B_R$ and $B_R(x)$ denote the open balls of radius $R$ centered  at $0$ and $x$ respectively. 
    \item $\M(\O)$ and $\M^+(\O)$ denote the spaces of finite signed and positive Radon measures on $\O$, respectively. 
    \item For brevity, we will often write $L^p$ and $L^{\infty}$ for the Lebesgue spaces $L^{p}(\Omega)$ and $L^{\infty}(\Omega)$.
    \item The $L^p$-norm of a function $g$ in $L^p$ ($1 \leq p <+\infty$) is written
    $\|g\|_p :=\left(\int_{\O} |g(z)|^p\,dz\right)^{\frac{1}{p}} $.
    \item $C(I)$ is the Banach space of continuous real-valued functions on $I$. For an interval $J\subset \R$, $C(I,J)$ denotes the subset of $C(I)$ with values in $J$. 
    \item For a Banach space $E$, the set $C(I,E)$ denotes continuous functions with values in $E$. 
    We alert the reader that we use both the notation $C(I,J)$ (when $J$ is an interval) and $C(I,E)$ when $E$ is a Banach space. 
    \item For $(k,\beta) \in \mathbb{N} \times  (0,1)$, the set $C^{k,\beta}(I,E)$, resp.~$C^{k,\beta}_{loc}(I,E)$, denotes functions that are  H\"older continuous, resp.~locally H\"older continuous of order $(k,\beta)$ with values in $E$. 
    \item For $(k,\beta) \in \mathbb{N} \times  (0,1)$, $C^{k,\beta}(\Omega)$, resp.~$C^{k,\beta}_{loc}(\Omega)$, denotes real-valued functions that are  H\"older continuous, resp.~locally H\"older continuous of order $(k,\beta)$.
    \item For $f, \tilde f\in C(\R,L^1)$, $m(t)$ and $\tilde m(t)$ stand for $\|f\|_1(t)$ and $\|\tilde f\|_{1}(t)$, respectively.
    \item The admissible control set is 
    $\U:=\{u:\R_+\rightarrow [0,u_{max}] \; | \; u  \textrm{ meas.} \}$ with $u_{max}>0$.
    Admissible controls are denoted by $u$.
    \item The time derivative of a function $F:\R\rightarrow \R$ is written 
        $\dot{F}$ or $\frac{d}{dt}F(t)$. 
        
\end{itemize}
Returning to the modeling assumptions (particularly that the growth function is always null in absence of substrate in a chemostat system), we adopt the following hypotheses for the kinetics.
\begin{hyp}\label{hyp0}
    $\mu\in C^{1}_{loc}(\R_+,C^{1}(\Omega))$ ; $\mu(s,\cdot)\in L^{\infty}(\O)\cap C^{0,1}(\O)$ for all $s>0$ ; $\mu(s,z)>0$ for all $(s,z)\in \R_+^*\times \bar \O$ and $\mu(0,z)=0$ for all $z\in \O$.
\end{hyp}
Additionally, we impose the following boundedness condition on the kinetics and a corresponding Lipschitz assumption.
\begin{hyp}\label{hyp1}
    There exists $\hat \mu>0$ such that for all $s\ge 0$ 
$$
\sup_{z \in \O} \mu(s,z)\le \hat \mu s,
$$
and for all $S\geq 0$ there is $b_S>0$ such that for all $s,s' \in [0,S]$ and for all  $z\in\O$, 
$$ 
  |\mu(s,z)-\mu(s',z)|\le b_S |s-s'|. 
$$
\end{hyp}

\subsection{Overview of the contribution and organization} 
This section presents the paper's main results in detail. 
 \subsubsection{Existence results for \eqref{main}}% (section \ref{sec1})}
To properly handle the control problem, it is important to study the input-output map, which associates a solution of the coupled system to a given control, and to show that it is well-defined.
That is why, we will first establish in Section~\ref{sec1} the existence 
and uniqueness of positive (mild and classical) solutions to \eqref{main}, given some admissible control.  
To understand the difficulty of obtaining such results, assume for the moment that the kinetics $\mu$ is of Monod type, that is, 
\begin{equation}{\label{monod-eq}}
\mu(s,z):= \frac{\bar \mu s}{r(z)+s},
\end{equation}
where $r\in C(\O)$, $r>0$ over $\Omega$, and $\bar \mu>0$. 
The  problem \eqref{main} then rewrites:
$$
\begin{dcases}
\partial_t f(t,z)=\left[\frac{\bar\mu s(t)}{r(z)+s(t)}- u(t)\right]f(t,z) +\alpha \Delta_{z} f(t,z) \quad &\text{ for }\quad t>0, z\in\O,\\
\dot s= -\bar\mu\int_{\O}\frac{s(t)}{r(z')+s(t)}f(t,z')\,dz' + u(t)(s_{in}-s(t))\quad &\text{ for }\quad t>0,\\
\partial_{\vec{n}} f(t,z)=0 \quad &\text{ for }\quad t>0, z\in\partial\O,\\
s(0)=s_0>0,\qquad f(0,z)=f_0(z) \quad &\text{ for }\quad z\in\O.\\
\end{dcases}
$$
When the dilution rate $u$ is smooth, it is reasonable to think that the above nonlinear problem is well posed and that the solution will be smooth thanks to the general parabolic theory. However, when $u$ is only measurable (for instance, whenever $u$ is of bang-bang type), this fact is not clear anymore. 
Our first result is to define properly a solution to \eqref{main} and to show that for any admissible control function, the system \eqref{main} admits a unique, well-defined solution.  
In this direction, our main result is as follows. The result is stated for arbitrary kinetics (not restricted to Monod kinetics), provided suitable assumptions are satisfied.

\begin{theorem}
\label{th-existence}
Suppose that $\alpha>0$ and that $\mu$ satisfies Hypotheses \ref{hyp0}-\ref{hyp1}. 
Then, for every admissible control $u$, and for every initial data $s_0\in (0,s_{in})$ and $f_0\in L^1(\O)$ such that $f_0\ge 0$, there is a unique pair of positive functions $(s,f)$ with  $s\in C^{1}(\R^*_{+},[0,s_{in}])\cap C(\R_{+},[0,s_{in}])$ and $f\in C(\R_+,L^{1}(\O))\cap C^{0,1}_{loc}(\R_+^*,C^{2,\beta}(\O)\cap C^{0,\beta}(\bar \O))$ such that $(s,f)$ is a mild solution to %the system 
\eqref{main}. Moreover, if $u\in C(\R)$, then $f$ is a strong solution and $f\in C^1(\R^*_+,C^{2,\beta}(\O)\cap C^{0,\beta}(\bar \O))$. 
\end{theorem}
The proof of this theorem is detailed in Section~\ref{sec1}. As we can see, there is a direct relationship between the smoothness of the control and the smoothness of the resulting solution. In particular, whenever $u$ is piecewise continuous, then, the solution is a piecewise $C^1$ function of the time.
When $\alpha =0$ the system \eqref{main} exhibits less regularity. Nevertheless, the existence of a solution can still be established, as stated below. 
\begin{theorem}%[Existence Thm]
\label{th:existence-alpha0}
Suppose that $\alpha=0$ and that $\mu$ satisfies Hypotheses \ref{hyp0}-\ref{hyp1}. 
Then, for every admissible control $u$, and for every initial data $s_0\in (0,s_{in})$ and $f_0\in L^1(\O)$ such that $f_0\ge 0$, there is a unique pair of non negative functions $(s,f)$ with  $s\in C^{1}(\R^*_{+},[0,s_{in}])\cap C(\R_{+},[0,s_{in}])$ and $f\in C(\R_+,L^{1}(\O))\cap C^1(\R_+^*,L^1(\O))$ such that $(s,f)$ is a mild solution to \eqref{main}. In addition,  $\supp(f(t,\cdot)) \subset \supp(f_0)$ for all $t\ge 0$. 
\end{theorem}
\begin{remark}{\label{rem0}} In contrast with Theorem \ref{th-existence}, the proof of the previous statement follows from classical results in the literature using  Cauchy-Lipschitz's theory in Banach spaces  \cite{Brezis2010,Desvillettes2008,Rudin1987}, therefore, we do not provide it here.
\end{remark}

\subsubsection{Stabilization results}% (section \ref{sec2})}  
Control theory frequently utilizes various input classes distinct from open-loop controls, in particular those involving a measure of the  state of the system. Therefore, in this paper, we shall  take advantage of 
{\it{feedback controls}} (closed-loop control functions) that depend on the state of the system. Such controls may be used in several contexts such as to adjust the solution of a control system under deviations, to reduce sensitivity of initial conditions, or to stabilize a dynamical system. 
In this framework, at each time instant, a feedback control is determined by the current state of the system. In practical terms, the control may be a function of time, space, and the state function itself. In Section~\ref{sec2}, we  examine convergence of solutions to \eqref{main} whenever the input is an  {\it{auxostat-type control}}, a particular type of feedback control. In bioprocesses, auxostats are designed to regulate the substrate concentration at a desired set-point (see, {\it{e.g.}}, \cite{Bayen2017} in the context of the chemostat system and Remark \ref{rm-auxo}).
We shall see that, under such a control function, the system is driven towards a desired stationary state of \eqref{main}. Thanks to these properties, auxostat-type controls will constitute a key tool in Section~\ref{secTarget} for proving the ``reach and stay'' property of the desired  given target set in the minimum-time problem \eqref{OCP0}, for Monod kinetics. 
Let us then consider control functions that can be expressed as a functional of the state, {\it{i.e.}},
$$
u(t):=\opr{f(t,\cdot), s(t)}, \quad t \geq 0,
$$ 
where $\r$  a continuous functional over the space $L^{1}(\O)\times \R_+$ with values in $\R_+$. 
For such a control function, we have the following result. 
\begin{theorem}%[Existence Thm]
\label{th:existence-closeloop}
Suppose that $\alpha>0$ and that $\mu$ satisfies Hypotheses  \ref{hyp0}-\ref{hyp1}. Then, for every control map $t\mapsto u(t):=\opr{f(t,\cdot), s(t)}$ and for every initial data $s_0\in (0,s_{in})$ and $f_0\in L^1(\O)$ such that $f_0\ge 0$, there is a unique pair of positive functions $(s,f)$ with  $s\in C^{1}(\R^*_{+},[0,s_{in}])\cap C(\R_{+},[0,s_{in}])$ and $f\in C(\R_+,L^{1}(\O))\cap C^{0,1}_{loc}(\R_+^*,C^{2,\beta}(\O)\cap C^{0,\beta}(\bar \O))$ such that $(s,f)$ is a mild solution to 
\eqref{main}. 
Moreover, if $u\in C(\R)$, then $f$ is a strong solution and $f\in C^1(\R_+^*,C^{2,\beta}(\O)\cap C^{0,\beta}(\bar \O))$. 
\end{theorem}
\begin{remark}{\label{rem1}} In contrast with Theorem \ref{th-existence}, we shall only sketch the proof of the previous statement, which follows from classical results in the literature since $t\mapsto u(t)$ is no longer arbitrary, but a feedback of $f$ and $s$ \cite{Coville2025,Perthame2007}. It can also be noticed that without a uniform bound on the $L^1$-norm of the solution $f$ and the precise description of the behavior of $s$, such a control function may exceed any prescribed upper bound $u_{max}$, making it potentially not suitable to be used for solving \eqref{OCP0}.
\end{remark}

Building on this result, Section~\ref{sec2} examines how auxostat-type controls $u$ of the following form affect the behavior of the system:
\begin{align*}
   & \text{(i)}\quad u(t)= \frac{1}{s_{in} -s(t)}\int_{\O}\mu(s(t),z)f(t,z)\,dz, 
    &&\text{(ii)} \quad u(t)= \frac{1}{s_{in} -\sigma}\int_{\O}\mu(s(t),z)f(t,z)\,dz,\\
    &\text{(iii)}\quad u(t)= \frac{1}{s_{in} -s(t)}\int_{\O}\mu(\sigma,z)f(t,z)\,dz,  
    &&\text{(iv)}\quad u(t)= \frac{1}{s_{in} -\sigma}\int_{\O}\mu(\sigma,z)f(t,z)\,dz,
    \end{align*}
where $\sigma\in (0,s_{in})$ is a given key parameter. 
To simplify notation, the dependence of $u$ on $\sigma$ is omitted. 
\begin{remark}{\label{rm-auxo}} Controls (i), (iii), and (iv) can be viewed as variants of control (ii), which represents the continuous biogas output in a continuously stirred chemostat. This quantity is readily measurable in practice 
(see, {\it{e.g.}}, \cite{Bayen2017}).  
The terminology {\it{auxostat-type controls}} stems from the fact that, when control (iii) is substituted into \eqref{main}, the system satisfies the following  property\footnote{This is straightforward from Cauchy-Lipschitz's Theorem.}:
$$
s(t_0)=\sigma \; \Rightarrow \; \forall t\geq t_0, \; \dot{s}(t)=0,  
$$
where $t_0$ is arbitrary.  
Thus, this control enables the regulation of the substrate concentration at the desired value $\sigma$. 
As far as we know, studying the effect of controls (i)-(ii)-(iii)-(iv) into \eqref{main}  has not been addressed previously. 
\end{remark}
For such controls, Theorem \ref{th:existence-closeloop} ensures the existence of a unique solution solution $(s,f)$ to \eqref{main},  which is well defined over $\R_+ \times \Omega$.
In Section~\ref{sec2}, we establish the existence of a stationary solution of \eqref{main}, parametrized by $\sigma$, around which the system is stabilized by any of the three controls (ii), (iii), or (iv), for all nonnegative initial data.
For the control given by $(i)$, the behavior of the system slightly differs, and we have the following convergence result. 

\begin{theorem}\label{th:stabi-i}
    Suppose that $\mu$ satisfies Hypotheses  \ref{hyp0}-\ref{hyp1}, that $\alpha>0$, and that the control $u(\cdot)$ is given by $(i)$. Let then $(s,f)$ be the unique solution to \eqref{main}  associated with the initial data $(s_0,u_0)$ where $s_0>0$ and $f_0 \geq 0$ ($f_0$ being non-null). Denote by $\varphi_{s_0}$  a positive eigenfunction associated with the principal eigenvalue $\lambda_1$ of the following spectral problem:
    \begin{equation*}
        \begin{dcases}
            \alpha \Delta\varphi (z) +\mu(s_0,z)\varphi(z) =-\lambda \varphi(z) \quad &\rm{for} \quad z\in\O,\\
            \partial_{\vec{n}} \varphi(z)=0 \quad &\rm{for} \quad z\in\partial\O.
        \end{dcases}
    \end{equation*}
  Then  for all $t\geq 0$, $s(t)=s_0$ and  for all $1\le p<+\infty$, one has:
 $$\lim_{t\to +\infty}\Big\|f(t,\cdot)- \frac{s_{in}-s_0}{\|\varphi_{s_0}\|_{1}}\varphi_{s_0}(\cdot)\Big\|_{p}=0.$$
\end{theorem}
The proof of this theorem is detailed in Section~\ref{sec2-1}. As a result, the feedback given by (i) drives the solution of the system to a stationary solution parametrized by the initial condition. From an applicative standpoint, controls of type (ii), (iii), or (iv) are preferable as the limit solution is independent of the initial condition. 
For each one of the three previous cases, the stabilization result can be stated as follows, noting that it requires slightly stronger assumptions on the growth rate function. 
\begin{theorem}\label{th:stabi-other}
    Suppose that $\mu$ satisfies 
    Hypotheses \ref{hyp0}-\ref{hyp1}, that $\alpha>0$, and let 
$\sigma\in(0,s_{in})$. Denote by $\varphi_\sigma$  a positive eigenfunction associated with the principal eigenvalue $\lambda_1$ of the following spectral problem:
    \begin{equation*}
        \begin{dcases}
            \alpha \Delta\varphi (z) +\mu(\sigma,z)\varphi(z) =-\lambda \varphi(z) \quad &\rm{for} \quad z\in\O,\\
            \partial_{\vec{n}} \varphi(z)=0 \quad &\rm{for} \quad z\in\partial\O.
        \end{dcases}
    \end{equation*}
    Assume further that at least one of the following assumptions holds:
 \begin{itemize}
     \item the control $u$ is given by (ii) and for all $s>0$, there is $c_0(s)>0$ such that
     $$\inf_{z\in\O}\mu(s,z)\ge c_0(s), $$ 
     \item the control $ u $ is given by (iii) and  for all $s\geq 0$, there is $c_1(s)>0$ such that 
     $$\inf_{z\in \O}\partial_s \mu(s,z)\ge c_1(s),$$ 
     \item the control $ u $ is given by (iv), for all $z\in\O,$ $\mu(\cdot,z) \in C^{2}(\R_+)$ and $\mu$ is strictly concave w.r.t.~$s$, {\it{i.e.}}, 
     $$\partial_{ss}\mu(s,z)<0 \quad \text{for all} \quad (s,z)\in [0,s_{in}]\times \O.$$
 \end{itemize}
 Then, the pair $(\bar s,\bar f):=\left(\sigma, (s_{in}-\sigma)\frac{\varphi_{\sigma}}{\|\varphi_\sigma\|_{1}} \right)$ is the unique stationary solution to \eqref{main} in which $u$ is given by any of the three controls that satisfy their corresponding aforementioned assumptions. Moreover, for any initial data $(s_0,f_0)$ where $s_0>0$ and $f_0\geq 0$ ($f_0$ being non-null), the unique solution  $(s,f)$ to \eqref{main} 
associated with the corresponding control
 converges to $(\sigma,\bar f)$ in the following sense: for $1 \leq p<+\infty$, one has 
 $$\lim_{t\to +\infty}\|f(t,\cdot)-\bar f(\cdot)\|_{p}=0.$$
\end{theorem}
The proof of this theorem is detailed in Section~\ref{sec2-2}. Theorem \ref{th:stabi-other} is illustrated in Section~\ref{sec-num-auxostat} in the case (iv). Additionally, the three control functions may not be always admissible, nevertheless, in Section~\ref{secTarget}, we show  how to construct admissible controls involving auxostat controls of the form (iv).
\subsubsection{Target controllability results}% (section \ref{secTarget})} 
We now turn to describing the results obtained in this paper concerning the reachability of a given target set by solutions of \eqref{main}. This analysis relies in particular on the auxostat-type controls introduced in the previous section. 
For this purpose, we assume throughout this part 
that $\mu$ is of Monod type, {\it{i.e.}}, it satisfies \eqref{monod-eq} where 
$r\in C(\O,\R_+^*)$ is the so-called {\it{half-saturation function}} and $\bar \mu>0$ is the maximum growth rate. 
For a given trait $z$, the half-saturation function $r(z)$ is the substrate concentration at which the specific growth rate reaches $\bar \mu/2$. 
It represents a measure of growth efficiency of the species on the substrate. For instance, a low half-saturation means that the species can grow well even if the substrate concentration is low. The function $r$ is the main component of the definition of the target set to be considered hereafter. 
To introduce it precisely, 
let us consider the functional 
$$
\opk{\nu}:=\frac{\int_{\O}r(z)\,d\nu(z)}{\int_{\O}\,d\nu(z)},  
$$
over the set $\M^+(\O)$ (here, $\nu \in \M^+(\O)$).
Next, we identify the positive cone of $L^1(\O)$ as a subset of $\M^+(\O)$, and  for $f\in L^{1}(\O)$ in this positive cone, the functional  $\k$ rewrites
$$
\opk{f}=\frac{\int_{\O}r(z)f(z)\,dz}{\int_{\O}f(z)\,dz},
$$
so that 
\begin{equation}{\label{target-property}}
r_0:=\min_{z\in\O} r(z)\le \opk{f}\le \max_{z\in \O} r(z)=:r_1.
\end{equation}
The functional $\k$ represents the mean half-saturation for all species, weighted by the abundance of each trait. From a modeling viewpoint, we will be interested in maintaining this value under a given threshold to ensure the selection of population having a low half-saturation, in order to have a more efficient use of the substrate (see also \cite{Bayen2017} for further explanations on such a selection process). The idea behind the definition of the target set is then to impose an upper bound on $\opk{f}$ for biological consideration. This leads to the following definition:
$$
\T_0:=\big\{f\in L^1 \; | \; \opk{f}\le k_0\big\},
$$
where $k_0\in (r_0,r_1)$.
If $k_0\ge r_1$, the problem is immediate since for any initial data and any control, the corresponding solution $f$ belongs to $\T_0$ for all time $t>0$, hence, $T_u=0$. Thus, 
any control $u$ is optimal, that is why we assume  $k_0\in (r_0,r_1)$. 
We now specify initial conditions for the minimum time problem. The set of admissible initial conditions is defined as 
$$
\F:= (0,s_{in}]\times \left\{f\in L^1(\Omega) \; | \; 
\exists (\kappa,\eta)\in \R_+^*\times \R_+^*, \; \forall z\in \Omega, \; 
f(z)\geq \kappa\mathds{1}_{B_{\eta}(\bar x)\cap\O}(z)\right\},
$$
where $\bar x \in 
\mathrm{arg\,min}_{z\in \bar \Omega} (r(z))$. Note that for biological reasons $s_0>0$ and $f_0$ is nonnegative over $\Omega$ ($f_0$ being non-zero).
The set $\mathcal{F}$ consists of functions having a positive measure in a neighborhood of the ``leader trait'', {\it{i.e.}}, the trait corresponding to the most efficient growth with respect to substrate consumption. For the specific class of Monod kinetics, this is equivalent to minimizing $r$. This requirement implies that the initial biomass distribution must take positive values in a neighborhood of the leader trait in order to guarantee the existence of an equilibrium when $\alpha = 0$, as predicted by the competitive exclusion principle \cite{Perthame2007}. 
To address the minimum time control problem, we begin by investigating 
the following target reachability problem. 
\begin{rcp}
    Given $u_{max}>0$ and $k_0\in (r_0,r_1)$,
  find a control $u\in \mathcal{U}$
  and a finite time $T\in(0,+\infty)$ such that $f(t,\cdot)\in \T_0$ for all $t\geq T$, that is, 
$$
\forall t\geq T, \quad \k[f(t,\cdot)]\leq k_0. 
$$
\end{rcp}
The main result of this section asserts that the target can always be reached by an admissible control provided that the mutation parameter is small enough. For future reference, we set 
\begin{equation}{\label{def-cst}}
\upsilon:=\sup_{(s,z)\in [0,s_{in}]\times \O }\mu(s,z)
\quad \mathrm{and} \quad
\bar u:=\max(\upsilon,4\upsilon s_{in}).   
\end{equation}

\begin{proposition}\label{prop:controlability}
 Let  $u_{max}\ge \bar u$.  Given $k_0\in (r_0,r_1)$,  there is $\alpha_0$ such that for all $\alpha\in [0,\alpha_0)$ and for all initial data in $\mathcal{F}$, there exists a control $u\in \mathcal{U}$ such that the corresponding solution to \eqref{main} enters the target set $\T_0$ in finite time and remains in 
 $\T_0$ for all subsequent times. 
 \end{proposition}
 The proof of this proposition is detailed in Section~\ref{secTarget-1}.
\subsubsection{Minimum time results}% (section \ref{sec-existence})}
Our last contribution is devoted to the study of \eqref{OCP0} when 
$\alpha$ is small enough (in accordance to Proposition~\ref{prop:controlability}). Based on the previous results, we establish the existence of an optimal control.    

\begin{theorem}{\label{thm-bigBoss}}
Let  $u_{max}\ge \bar u$.  Given $k_0\in (r_0,r_1)$,  there is $\alpha_0$ such that for all $\alpha\in[0,\alpha_0)$ and for all $(s_0,f_0)\in\mathcal F$, the minimum-time problem \eqref{OCP0} admits an optimal control 
$u^*\in \U$ such that the associated state $(f^*,s^*)$ reaches the target set $\mathcal{T}_0$ in a time $T^*$. 
\end{theorem}
The proof of this theorem is detailed in Section~\ref{sec-existence-1}. Finally, we complement these theoretical results with numerical simulations illustrating  convergence of solutions under auxostat-type controls in Section~\ref{sec-num-auxostat}, and sub-optimal solutions to the minimum time control problem in Section~\ref{sec-existence-2}.
\smallskip
\section{Existence of solutions to (\ref{main})}\label{sec1}
The goal of this section is to prove \Cref{th-existence}. 
\subsection{Reformulation of the state equation}
Before going any further, we make some preliminary remarks on system \eqref{main}. 
First, we extend the function $\mu(\cdot,z)$
by continuity to $s\leq 0$ by setting it equal to $0$, and we keep the notation $\mu$ for this extension. Next, we rewrite system \eqref{main} in a more convenient form. For this purpose, for $t\geq 0$, let us set 
$$\ds{M(t):=s(t)-s_{in}+m(t)},$$ 
where $m$ represents the total population $$m(t)=\int_\Omega f(t,z) \, dz.$$ A quick computation shows that $M$ satisfies the following Cauchy problem:
\begin{equation}\label{eq:edo-M}
\frac{d}{dt} M(t)=-u(t)M(t), \quad  M(0)=s(0)-s_{in}+m(0),
\end{equation}
and thus, for all time $t\geq 0$, one has 
$$
\ds M(t)=M(0) e^{-\int_{0}^t u(\tau)\,d\tau}.
$$
As a remark, the total biomass concentration can be written
\begin{equation}{\label{totalBiomass}}
 m(t)=s_{in}-s(t)+(s(0)-s_{in}+m(0))e^{-\int_{0}^t u(\tau)\,d\tau}, \quad t\geq 0. 
\end{equation}
It follows that \eqref{main} is equivalent to the following system (in which the equation for $f$ is time dependent)
\begin{equation}\label{eq-red}
\begin{dcases}
\partial_t f(t,z)=\left[\mu\left(M(t)+s_{in}-m(t),z\right)- u(t)\right]f(t,z) +\alpha \Delta f(t,z) \; &\text{ for   }\; z\in\O,\\
\partial_{\vec{n}} f(t,z)=0 \quad &\text{ for   } z\in\partial\O,\\
f(0,z)=f_0(z)\ge 0 \quad &\text{ for   } z\in\O,
\end{dcases}
\end{equation}
where $t>0$. 
It is readily seen that $f$ is a solution to \eqref{eq-red} if and only if the function 
$$
\tilde{f}(t,x):= e^{\int_{0}^t u(\tau) \,d\tau}f(t,x)
$$
is a solution to the system 
\begin{equation}\label{eq-red:bis}
\begin{dcases}
\partial_t \tilde f(t,z)=\left[\mu\left(M(t)+s_{in}-\frac{M(t)}{M(0)}\tilde m(t),z\right)\right]\tilde f(t,z) +\alpha \Delta \tilde f(t,z) \; &\text{ for }z\in\O,\\
\partial_{\vec{n}} \tilde f(t,z)=0 \quad &\text{ for  } z\in\partial\O,\\
\tilde f(0,z)=f_0(z)\ge 0 \quad &\text{ for   } z\in\O,
\end{dcases}
\end{equation}
where $t>0$. 
Hence, solving \eqref{eq-red:bis} provides a solution to  \eqref{eq-red} and thus a solution to \eqref{main}.
In the remaining of this section, our aim is to construct of positive solution to \eqref{eq-red:bis}.  
%This question is addressed by the following result. 
\subsection{Proof of Theorem \ref{th-existence}}
In this part, we show the two next theorems. Note that Theorem \ref{th-existence} about existence and uniqueness of a solution to \eqref{main} follows from Theorem \ref{thm:existence2} below (the substrate $s$ can indeed be deduced from $f$ via \eqref{totalBiomass}). 
\begin{theorem}
\label{thm:existence1}
Let $\O\subset \R^d$ be a bounded domain with a boundary of class $C^1$ and let $u\in L^{\infty}_{loc}(\R_+)$. Suppose also that $\alpha>0$ and that $\mu$ satisfies Hypotheses  \ref{hyp0}-\ref{hyp1}. Then, for all initial data  $f_0\ge 0$ such that $f_0\in L^1(\O)$, there exists a nonnegative function $\tilde f\in C(\R_+,L^1(\O))\cap C^1(\R_+^*,C^{2,\beta}(\O)\cap C^{0,\beta}(\bar \O))$ solution to \eqref{eq-red:bis}, for some $\beta\in(0,1)$.     
\end{theorem}
Assume temporarily that \Cref{thm:existence1} holds. By exploiting the  
relationship between $f$ and $\tilde f$, we deduce the following result (which implies \Cref{th-existence}). 
\begin{theorem}
\label{thm:existence2}
Let $\O\subset \R^d$ be a bounded domain with a boundary of class $C^1$ and let $u\in L^{\infty}_{loc}(\R_+)$. Suppose also that $\alpha>0$ and that $\mu$ satisfies Hypotheses  \ref{hyp0}-\ref{hyp1}. Then, for all initial data  $f_0\ge 0$ such that $f_0\in L^1(\O)$, there exists a unique nonnegative solution $f\in C(\R_+,L^1(\O))\cap C^{0,1}_{loc}(\R_+^*,C^{2,\beta}(\O)\cap C^{0,\beta}(\bar \O))$ solution to \eqref{eq-red} for some $\beta\in(0,1)$. Moreover, if $u\in C(\R)$, then $f\in C^1(\R_+^*,C^{2,\beta}(\O)\cap C^{0,\beta}(\bar \O))$. 
\end{theorem}
\begin{proof}
    The existence of a solution follows from Theorem~\ref{thm:existence1} setting 
    $$
    f(t,z):= e^{-\int_{0}^t u(\tau) \,d\tau} \tilde{f}(t,z), 
    \quad (t,z)\in \R_+^* \times \Omega.
    $$ 
    It remains to show that the solution is unique. 
    The regularity properties of the solution $f$ depends only on the regularity of the control $u$ and, since for $u\in L^{\infty}_{loc}(\R)$, the function $t\mapsto \ds e^{\int_{0}^t u(\tau) \,d\tau}$ 
    belongs to the space\footnote{It is the space of real-valued functions on $\R$ such that, on every bounded interval, the function and its first weak derivative belong to $L^p(\R)$.} $W^{1,p}_{loc}(\R)$ for all $p\ge 1$ and is, in particular, locally Lipchitz, we deduce that $f\in C^{0,1}_{loc}(\R_+^*,C^{2,\beta}(\O)\cap C^{0,\beta}(\bar \O))\cap C(\R_+,L^{1}(\O)) $. 
    Suppose there exist two distinct positive solutions $ f_1(t,z),f_2(t,z)$  to  \eqref{eq-red} with identical initial data. Then the corresponding functions
    $\tilde f_1,\tilde f_2$ solve \eqref{eq-red:bis} and inherit the same initial data. By integrating \eqref{eq-red:bis} with respect to $z$,  the function
    $\tilde m_i=\|\tilde f_i\|_{1}$  satisfies:  
\begin{equation*}%*{\label{tmp1}}
    \begin{dcases}
        \frac{d}{dt}\tilde m_i(t) = \int_{\O}\mu\left(M(t)+s_{in}-\frac{M(t)}{M(0)}\tilde m_i(t),z\right)\tilde f_i(t,z)\, dz,\\
        \tilde m_i(0)=m(0),
        \end{dcases}
    \end{equation*}
for $i=1,2$. By definition of $M$, we also have 
\begin{align}
&\frac{M(t)}{M(0)}=e^{-\int_{0}^t u(\tau)\,d\tau},\label{eq:esti-M/M0}\\
&M(t)+s_{in}=M(0)e^{-\int_{0}^t u(\tau)\,d\tau}+s_{in}= (m(0)+s_0)e^{-\int_{0}^t u(\tau)\,d\tau}+s_{in}\left(1-e^{-\int_{0}^tu(\tau)\,d\tau}\right), 
\end{align}
for all $t\ge 0$, and thus, for $i=1,2$:
\begin{equation}\label{eq:esti-M}
\forall t\geq 0, \; M(t)+s_{in}-\frac{M(t)}{M(0)}\tilde m_i<s_{in}+s_0+m(0).
\end{equation}
    Let us now define a function $h:\R_+\rightarrow \R$ as $${h(t):= \int_{\O}[\tilde f_1(t,z)-\tilde f_2(t,z)]\,dz}=\tilde m_1(t)-\tilde m_2(t), \quad t\geq 0.$$ We have $h(0)=0$ and, using \eqref{eq-red:bis}, we deduce that for all $t>0$, 
    \begin{multline*}
        \dot h(t) = \int_{\O}\mu\left(M(t)+s_{in}-\frac{M(t)}{M(0)}\tilde m_1(t),z\right)\tilde f_1(t,z)\,dz  \\ - \int_{\O}\mu\left(M(t)+s_{in}-\frac{M(t)}{M(0)}\tilde m_2(t),z\right)\tilde f_2(t,z)\, dz 
        =:\mathrm{rhs}(t)
    \end{multline*}
       Now since $\mu$ satisfies Hypotheses  \ref{hyp0}-\ref{hyp1}, there is $C\geq 0$ such that for all $t\geq 0$, 
     {\small{$$ \left|\mu\left(M(t)+s_{in}-\frac{M(t)}{M(0)}\tilde m_1(t),z\right)-\mu\left(M(t)+s_{in}-\frac{M(t)}{M(0)}\tilde m_2(t),z\right)\right|\le C\left|\frac{M(t)}{M(0)}\right||h(t)|,$$}}
\noindent and, therefore, by using \eqref{eq:esti-M/M0} and \eqref{eq:esti-M}, the expression
$\mathrm{rhs}(t)$  can be estimated as follows:
     \begin{align*}
      |\mathrm{rhs}(t)|&\le C |h(t)|\tilde m_1(t)  + \int_{\O}\mu\left(M(t)+s_{in}-\frac{M(t)}{M(0)}\tilde m_2(t),z\right)(\tilde f_1(t,z) -\tilde f_2(t,z))\, dz\\
      &\le C|h(t)|\tilde m_1(t)+  C_0 |h(t)|, \quad t\geq 0,
     \end{align*}
     where \begin{equation}{\label{tmpC0}}C_0:=\ds \sup_{(s,z)\in \Lambda} \; \mu(s,z) \quad \mathrm{with} \quad \Lambda:=[0,s_{in}+s_0+m(0)]\times \Omega.\end{equation}
    Thus, the function $h$ satisfies
    \[ \forall t > 0, \; |\dot h(t)| \le \left(C\tilde m_1(t)+ C_0\right) |h(t)|.\]
    From the above differential inequality and by using  
    %a standard argumentation based on 
    Gr\"onwall's inequality, we can see  that $h(t)= 0$ for all $t\ge0$. 
    As a consequence we have $\tilde m_1(t)= \tilde m_2(t)$ for all $t\geq 0$ and 
    thus also  $m_1(t)= m_2(t)$ for all $t\geq 0$. 
    Let us now set $w:=f_1 -f_2$. A quick calculation yields 
    \begin{equation}\label{eq-uniq-1}
    \begin{dcases}
    \partial_t w(t,z)=\left[\, \mu\left(M(t)+s_{in}-m_1(t),z\right) -u(t)\right] w(t,z) +\alpha \Delta_{z} w(t,z) & \text{ for } z\in\O,\\
    \partial_{\vec{n}} w(t,z)=0 \quad &\text{ for } z\in\partial\O,\\
    w(0,z)=0 \quad &\text{ for } z\in\O,
    \end{dcases}
    \end{equation}  
for all $t>0$. Hence, by applying the weak parabolic maximum principle we get $w(t,z)\le 0$ for all $(t,z)\in \R_+^* \times \Omega$. Interchanging the roles of  $f_1,f_2$, we also get $-w(t,z)\le 0$ for all $(t,z)\in \R_+^* \times \Omega$.  We conclude that $f_1(t,z)=f_2(t,z)$ for all $(t,z)\in \R_+^* \times \Omega$ as wanted. \end{proof}
We now turn to the proof of  \Cref{thm:existence1}.
\begin{proof}[Proof of \Cref{thm:existence1}] 
First, observe that, since $u\in L^{\infty}_{loc}(\R)$, the function 
$t\mapsto \ds{e^{\int_{0}^t u(\tau) \,d\tau}}$ belongs to the space 
$W^{1,p}_{loc}(\R)$ for all $p\ge 1$, and, in particular, it is H\"older continuous. 
To construct a solution to \eqref{eq-red:bis}, we will follow standard approximation techniques. 
The rest of the proof is divided into two steps. 
\smallskip

\noindent {\it{Step 1: Existence of a solution for a regularized initial condition.}}  
\smallskip
\\
First, we regularize $f_0$ by a smooth mollifier $(\rho_{\eps})_\eps$ with unit mass and we consider 
\eqref{eq-red:bis} together with the initial condition $f_{0,\eps}:=\rho_\eps\star f_0$ in place of $f_0$.
Next, for $\eps>0$ fixed, let us introduce 
a sequence of function $(\tilde f_{n,\eps})_{n\in \N}$ such that $\tilde f_{0,\eps}=f_{0,\eps}$ and where for $n \in \mathbb{N}^*$, the function $\tilde f_{n,\eps}$ is defined recursively by
\begin{equation}\label{eq-ex-1}
\begin{dcases}
\partial_t  \tilde f_{n,\eps}(t,z)=c_n(t,z)\tilde f_{n,\eps}(t,z) +\alpha \Delta_{z} \tilde f_{n,\eps}(t,z)&\text{ for }  t>0, z\in \O, \\
\partial_{\vec{n}} \tilde f_{n,\eps}(t,z)=0 &\text{ for } z\in\partial\O,\\
\tilde f_{n,\eps}(0,z)=f_{0,\eps}(z)\ge 0 & \text{ for } z\in\O,
\end{dcases}
\end{equation}
 where the function $c_n$ is defined by\footnote{Naturally, $\tilde m_{n,\eps}=\|\tilde f_{n,\eps}\|_{1}$, and in the following, $m_{\eps}=\| f_{\eps}\|_{1}$.}:
\begin{equation}{\label{coeff}}
c_n(t,z):=\mu\left(M(t)+s_{in}-\frac{M(t)}{M(0)}\tilde m_{n-1,\eps}(t),z\right) \; \text{ for }\; (t,z)\in \R_+\times \Omega. 
\end{equation}
Since by definition $f_{0,\eps} \in C^{\infty}(\O)$, $M\in C^{0,1}(\R)\cap L^{\infty}(\O)$,  and  $\mu$ is Lipschitz-continuous, from the standard parabolic theory 
(see, {\it{e.g.}}, \cite{Evans1998,Wu2006}), the sequence $(\tilde f_{n,\eps})_{n\in\N}$ is well defined, that is, for each $n\in \mathbb{N}^*$ there is a unique smooth solution $\tilde f_{n,\eps}$ to \eqref{eq-ex-1}  that belongs to the space $E_\beta$ for some $\beta\in (0,1)$ where $$E_\beta:=C(\R_+;C^{2,\beta}(\O)\cap C^{0,\beta}(\bar \O))\cap C^1(\R_+^*;C^{2,\beta}(\O)\cap C^{0,\beta}(\bar \O)).$$ 
Moreover since $f_{0,\eps}\ge 0$ and for each $n\in \mathbb{N}$, the null function is a sub-solution to \eqref{eq-ex-1}, a straightforward application of the parabolic comparison principle, yields  $\tilde f_{n,\eps}\ge 0$ for all $n\in \mathbb{N}$. 
From \Cref{hyp1} and using that $\tilde f_{n,\eps}\ge 0$ for all $n\in \mathbb{N}$, we deduce that
\begin{equation}\label{eq:esti-M-n} 
\forall n \in \mathbb{N}^*, \; \forall t \geq 0, \; M(t)+s_{in}-\frac{M(t)}{M(0)}\tilde m_{n-1,\eps}<s_{in}+s_0+m(0),
\end{equation}
as for proving \eqref{eq:esti-M}. 
Consequently, for all $n \in \mathbb{N}$, 
$\tilde f_{n,\eps}$  satisfies 
\begin{equation*}\label{eq-ex-2}
\begin{dcases}
\partial_t \tilde f_{n,\eps}(t,z)\le C_0\tilde f_{n,\eps}(t,z) +\alpha\Delta_{z} \tilde f_{n,\eps}(t,z) \quad &\text{ for  }  t>0, \; z\in \O, \\
\partial_{\vec{n}} \tilde f_{n,\eps}(t,z)=0 \quad &\text{ for   }t>0, \; z\in\partial\O,\\
\tilde f_{n,\eps}(0,z)=f_{0,\eps}(z)\ge 0 \quad &\text{ for  } z\in \O, 
\end{dcases}
\end{equation*}
where $C_0$ is given by \eqref{tmpC0}. %Now, 
 By applying again the parabolic comparison principle, 
 we get 
\begin{equation}{\label{tmp-mintime}}
\forall n\in \mathbb{N}^*, \; \forall (t,z)\in \R_+^* \times \Omega, \;   \tilde f_{n,\eps}(t,z)\leq e^{C_0t}v_\eps(t,z),
\end{equation} 
  where $v_\eps$ is the unique solution of the  diffusion equation 
  \begin{equation*}\label{eq-ex-3}
\begin{dcases}
\partial_t v_\eps(t,z)= \alpha \Delta_{z} v_\eps(t,z) \quad &\text{ for }  t>0, \; z\in \O, \\
\partial_{\vec{n}} v_\eps(t,z)=0 \quad &\text{ for } t>0, \; z\in\partial\O,\\
v_\eps(0,z)=f_{0,\eps}(z)\ge 0 \quad &\text{ for } z\in \O.
\end{dcases}
\end{equation*}
 Observe that from the comparison principle,  we also have  $\|v_\eps\|_{\infty}(t)\le \|f_{0,\eps}\|_{\infty}$ for all $t\ge0$. Therefore, we deduce that for all $T>0$ and all $n\in\N$, 
$$\forall (t,z)\in [0,T] \times\O, \; \;  f_{n,\eps}(t,z)\le e^{C_0T}\|f_{0,\eps}\|_{\infty}.$$
As a result, since $\mu$ satisfies \Cref{hyp0}, \eqref{eq:esti-M-n} implies that  there is $C'\geq 0$ such that 
\begin{equation}{\label{tmp1}}
\forall n \in \mathbb{N}, \; \forall (t,z)\in  [0,T]\times \Omega, \; F_n(t,z):=|c_n(t,z)\tilde f_n(t,z)|\leq C'e^{C_0T}\|f_{0,\eps}\|_{\infty}.
\end{equation}
Let us rewrite \eqref{eq-ex-1} as follows:
 \begin{equation*}
\begin{dcases}
\partial_t \tilde f_{n,\eps} -\alpha \Delta\tilde f_{n,\eps} = F_n(t,z) \quad &\text{ for }  t>0, \; z\in \O, \\
\partial_{\vec{n}} \tilde f_{n,\eps}(t,z)=0 \quad &\text{ for }t>0, \; z\in\partial\O,\\
\tilde f_\eps(n,z)=f_{0,\eps}(z)\ge 0 \quad &\text{ for }  z\in \O.
\end{dcases}
\end{equation*}
Since for all $T>0, F_n$ belongs to the space $L^{\infty}([0,T]\times \bar \O)$, by standard parabolic $L^p$ estimates (see \cite[Theorem 9.2.2]{Wu2006}), we deduce that for all $T>0$  and all $p\ge 1$, $\tilde f_{n,\eps}$ is now bounded independently of $n$ in the space  $W^{1,2,p}\big(Q_T \big)\cap W^{1,1,p}_{0}\big(Q_T \big)$, where $Q_T := (0,T)\times \O$ and $W^{1,2,p}\big(Q_T \big)$ and $ W^{1,1,p}_{0}\big(Q_T \big)$ denote the  $t$-anisotropic Sobolev spaces defined respectively by:
  \begin{align*}
  &W^{1,2,p}\big(Q_T \big):=\big\{u\in L^p(Q_T ) \; | \;  \partial_t u, \, \nabla u, \, \mathrm{and}\; \partial_{ij}u \in L^p(Q_T ), \; \mathrm{for} \; \mathrm{all} \; 1 \leq i,j\leq d \big\},\\
  &W^{1,1,p}_0\big(Q_T \big):=\big\{u\in L^p(Q_T) \; | \; \partial_{\vec{n}}u=0 \text{ on } \partial \O \; \mathrm{and} \; \partial_t u, \nabla u \in L^p(Q_T) \big\}.
\end{align*}

The embedding $W^{1,2,p}\big(Q_T \big)\hookrightarrow C^{0,\frac{\beta}{2}}\big([0,T],C^{0,\beta}(\bar \O) \big)$ (which is valid for $p>\frac{n+2}{2}$ and $0<\beta<2-\frac{n+2}{p}$, see, {\it{e.g.}}, \cite[Theorem 1.4.1]{Wu2006}) implies the existence of $\beta\in (0,1)$ such that the sequence $(\tilde f_{n,\eps})_{n\in\N}$ is uniformly bounded in the space $C^{0,\frac{\beta}{2}}\big([0,T],C^{0,\beta}\big(\bar \O \big)\big)$. 
As a result, from the definition of $F_n$, we see that $F_n$
is bounded in $C^{0,\frac{\beta}{2}}\big([0,T],C^{0,\beta}(\bar \O)\big)$ independently of $n$, therefore, by using Schauder estimates, the sequence $(\tilde f_{n,\eps})_{n\in \N}$ is uniformly bounded in $C^{1,\frac{\beta}{2}}((0,T),C^{2,\beta}(\O)\cap C^{0,\beta}(\bar \O) )$ for each $T>0$. 
By a diagonal extraction procedure, using the compact embedding $C^{1,\frac{\beta'}{2}}((0,T),C^{2,\beta'}(\O)\cap C^{0,\beta'}(\bar \O) )\hookrightarrow C^{1,\frac{\beta}{2}}((0,T),C^{2,\beta}(\O)\cap C^{0,\beta}(\bar \O) )$, for any $\beta'<\beta$ there exists a subsequence $(\tilde f_{n_k,\eps})_{k\in\N}$ which converges to a non-negative function $\tilde f_\eps\in C^{1,\frac{\beta'}{2}}(\R_+^*,C^{2,\beta'}(\O)\cap C^{0,\beta'}(\bar \O)) \cap C(\R_+,C^{2,\beta'}(\O)\cap C^{0,\beta'}(\bar \O))$ that is a solution to \eqref{eq-red:bis} considering $f_{0,\eps}$ as the initial condition. As a byproduct, we have $\tilde f_\eps\le e^{C_0t}v_\eps.$ % for all $(t,z)\in \R_+^* \times \Omega$.  
As the parameter $\eps>0$ was arbitrary in the above construction, we have obtained a solution $\tilde f_\eps$ to \eqref{eq-red:bis} with initial condition $f_{0,\eps}$ for any $\eps>0$. 
\smallskip
\\
\noindent  {\it{Step 2: Proof of the existence of a solution to \eqref{eq-red:bis} with the initial  condition $f_0$}}. 
\smallskip
\\
For this purpose, let $(\eps_n)$ be such that $\eps_n \downarrow 0$ and consider the sequence of functions $(\tilde f_{\eps_n})$. 
 We next show that this sequence is relatively compact in a suitable Banach space and that, up to a subsequence, it converges to a desired solution. 

First, we claim that for each $T>0$, $\tilde f_\eps$ is bounded in the space $C^{1,\frac{\beta}{2}}((0,T],C^{2,\beta}(\O)\cap C^{0,\beta}(\bar \O))$ independently of $\eps$, for some $0<\beta<1$. 
Going back to  \eqref{main}, we see that $s_\eps(t)$ satisfies
\[\dot s_\eps(t)=-\int_{\O}\mu(s_\eps(t),z)f_\eps(t,z)\,dz + u(t)(s_{in}-s_{\eps}(t)),\]
for $t\ge 0$. Note that for all $t\ge 0$, one has $s_\eps(t)\ge 0$ for all $\eps>0$.  This follows using that $s_\eps(0)=s_0>0$, $f_\eps\ge 0$, and Hypotheses  
\ref{hyp0} and \ref{hyp1}. 
% . %  we deduce that  .
As a consequence, we deduce that 
$$ \forall t\geq 0, \; 0\le m_\eps(t)\le  M(t)+s_{in} \le m(0)+s_0+s_{in}.$$
By integrating  \eqref{eq-red} over $\O$, we find that for all $t>0$, one has:
$$|\dot m_\eps(t)|=\left|\int_{\O}\mu( M(t)+s_{in} -m_\eps(t),z)f_\eps(t,z)\,dz- u(t)m_\eps(t)\right|\le (C_0+u(t))m_\eps(t). $$
Therefore, $m_\eps$ is uniformly Lipchitz-continuous over $(0,T)$. 
From this fact, combined to \Cref{hyp1} and the fact that the function $M$ is uniformly Lipchitz-continuous independently of $\eps$, we thus deduce that the following quantity $c_\eps$, defined by 
$$
c_\eps(t,z):=\mu\left(M(t)+s_{in}-\frac{M(t)}{M(0)}\tilde m_{\eps}(t),z\right) \; \text{ for }\; (t,z)\in \R_+\times \Omega,
$$
is uniformly bounded and Lipschitz-continous in time and space, independently of $\eps$. 
Now, $\Omega$ being a smooth domain, the Neumann heat kernel generates a strongly continuous semigroup of contractions on $L^1(\O)$ (see \cite{Brezis2010,Pazy1983}),  thus,  there exists $C>0$ such that:
$$
\forall (t,z)\in \R_+^*\times \Omega, \; \; |v_\eps(t,z)|\le Ct^{-\frac{d}{2}}\|f_{0,\eps}\|_1 \le Ct^{-\frac{d}{2}}\|f_{0}\|_1. 
$$
As a result,  for any $\delta>0$ we can find a constant $C(\delta)$ such that for all $t\ge \delta$ and all $z\in \Omega$, one has $v_\eps(t,z)\le C(\delta)$. 
The preceding property implies that for all $\delta>0$ and all $T>\delta$, $(\tilde f_\eps)$ is uniformly bounded independently of $\eps$ in the set $Q_T^\delta=(\delta,T)\times \O$ and that $c_\eps\tilde f_{\eps}$
is also uniformly bounded in $Q_T^\delta$ independently of $\eps$. By using standard parabolic $L^p$ estimates (\cite[9.2.2]{Wu2006}), it follows that for all $T>\delta$  and all $p\ge 1$, $(\tilde f_\eps)$ is  bounded independently of $\eps$ in   $W^{1,2,p}\big(Q_T^\delta \big)\cap W^{1,1,p}_{0}\big(Q_T^\delta \big)$. 
By using the embedding $W^{1,2,p}\big(Q_T^{\delta} \big)\hookrightarrow C^{0,\frac{\beta}{2}}\big( [\delta,T],C^{0,\beta}(\bar \O ) \big)$ for $p>\frac{n+2}{2}$ and $0<\beta<2-\frac{n+2}{p}$ (see \cite[Theorem 1.4.1]{Wu2006}) together with Schauder estimates, we conclude that $(\tilde f_\eps)$ is bounded in $C^{1,\frac{\beta}{2}}\big([\delta,T],C^{2,\beta}(\O)\cap C^{0,\beta}(\overline{\O}) \big)$ for some $\beta\in (0,1)$. This ends up the proof of our claim.  
 \smallskip
 
 Finally, let us consider a sequence of positive numbers  $(\delta_n)$ such that $\delta_n \downarrow 0$ as $n\rightarrow +\infty$. By using a diagonal extraction procedure and the  claim, there exist $\beta'\in (0,\beta)$ and a subsequence $(\tilde f_{\eps_{n_k}})_{k\in\N}$ which converges to a smooth nonnegative solution 
 $\tilde f$ of \eqref{eq-red:bis} over $Q_T$, with initial condition $f_0$ and such that 
 $$\tilde f \in C^{1,\frac{\beta'}{2}}\big((0,T), C^{2,\beta'}(\O)\cap C^{0,\beta'}(\bar \O) \big)\cap C([0,T], L^1(\O)).$$  
Since $T$ is arbitrary, we conclude that $\tilde f \in C^{1,\frac{\beta'}{2}}\big(\R_+^*, C^{2,\beta'}(\O)\cap C^{0,\beta'}(\bar \O) \big)\cap C(\R_+, L^1(\O))$.
\end{proof}
\section{Stabilization by auxostat-type controls}\label{sec2}
In this section, we address asymptotical properties of \eqref{main} 
when the control $u$ is an auxostat-type control given by:
\begin{align*}
   & \text{(i)}\quad u(t)= \frac{1}{s_{in} -s(t)}\int_{\O}\mu(s(t),z)f(t,z)\,dz, 
    &&\text{(ii)} \quad u(t)= \frac{1}{s_{in} -\sigma}\int_{\O}\mu(s(t),z)f(t,z)\,dz,\\
    &\text{(iii)}\quad u(t)= \frac{1}{s_{in} -s(t)}\int_{\O}\mu(\sigma,z)f(t,z)\,dz,  
    &&\text{(iv)}\quad u(t)= \frac{1}{s_{in} -\sigma}\int_{\O}\mu(\sigma,z)f(t,z)\,dz,
    \end{align*}
where $\sigma\in (0,s_{in})$ is given. Additionally, from now on, 
the initial data is taken in the set 
\begin{equation}{\label{def-D}}
\D:=(0,s_{in})\times \left\{f_0 \in L^{1}(\O)\; | \; f_0\ge 0 \; \mathrm{and} \;  \int_{\O}f_0(z)\,dz >0\right\}.
\end{equation}
\subsection{Control of the form (i)}\label{sec2-1}
In this part, we prove \Cref{th:stabi-i}. To this end, observe that the control defined in $(i)$ gives $\dot{s}=0$ and therefore,  $s(t)= s_0$ for all time $t \geq 0$. It follows that  
$$\forall t\geq 0, \; \; u(t)= \frac{1}{s_{in} -s_0}\int_{\O}\mu(s_0,z)f(t,z)\,dz,$$
and, consequently, $f$ satisfies :
\begin{equation*}
\begin{dcases}
\partial_t f(t,z)=\left(\mu(s_0,z)- \frac{1}{s_{in}-s_0}\int_{\O}\mu(s_0,z')f(t,z')\,dz'\right)f(t,z) +\alpha \Delta_{z} f(t,z) & \text{for } t>0, z\in\O,\\
\partial_{\vec{n}} f(t,z)=0 & \text{for }t>0, x\in\partial\O,\\
f(0,z)=f_0(z)\ge 0 & \text{for } z\in\O.
\end{dcases}
\end{equation*}
The function $m(t):=\int_{\Omega} f(t,z) \, dz$ then satisfies: 
$$
\dot m(t)= \left(1-\frac{m(t)}{s_{in}-s_0}\right)\int_{\O}\mu(s_0,z)f(t,z)\,dz
$$
over $\R_+^*$. 
For initial conditions $f_0\in\D$, a quick analysis  shows that  $\ds{m(t) \to s_{in} -s_0}$ as $t\to +\infty$.
On the other hand, thanks to \cite{Coville2012a,Leman2015,Mirrahimi2012a}, we know that  for any $1 \leq p<+\infty$, $f(t,x)\to (s_{in}-s_0) \frac{\varphi_{s_0}}{\|\varphi_{s_0}\|_1}$ in $L^p(\O)$ as $t\rightarrow +\infty$, 
where $\varphi_{s_0}$ is a positive eigenfunction  associated with the first eigenvalue
$\lambda_1(\alpha \Delta + \mu(s_0,z))$ and Neumann boundary condition. 
More precisely, $\varphi_{s_0}$ satisfies 
\begin{equation*}
\begin{dcases}
    \alpha \Delta \varphi_{s_0} +\mu(s_0,z)\varphi_{s_0} + \lambda_1 \varphi_{s_0} =0 \quad &\text{ for all  } z\in\O, \\
\partial_{\vec{n}} \varphi(z)=0 \quad &\text{ for all  } z\in\partial\O.
\end{dcases}
\end{equation*}
This ends up the proof of \Cref{th:stabi-i}.  
\begin{remark}
In case (i), the solution $f$ can be made explicit: $$f(t,z)=\frac{\Psi(t,z) w(t,z)}{1+\int_{0}^t \int_{\O}\frac{\mu(s_0,z)}{s_{in}-s_0} \Psi(s,z)w(s,z)\,dzds},
$$
where $\Psi:=e^{-\lambda_1 t}\varphi_{s_0}$ and $w(t,z)=e^{t\l }\left[\frac{f_0(z)}{\varphi_{s_0}}(z)\right]$ for $(t,z)\in\R_+\times \Omega$. 
Here $e^{t\l}$ stands for the Neumann diffusion semi-group generated by the elliptic operator $\l[w]:= \frac{\alpha}{\varphi_{s_0}^2}  \nabla \cdot \left(\varphi_{s_0}^2\nabla w\right)$.
\end{remark}
\subsection{Controls of the form (ii), (iii) and (iv)} \label{sec2-2}
The goal now is to prove Theorem \ref{th:stabi-other}. 
We start by proving the following lemma.
\begin{lemma}\label{lem:autoxat-s-et-m}
 Let $\sigma\in(0,s_{in})$ and $\mu$ satisfying Hypotheses \ref{hyp0}-\ref{hyp1}.  Assume further that at least one of the following assumption holds:
 \begin{itemize}
     \item $ u $ is given by (ii) and for all $s>0$, there is $c_0(s)>0$ such that
     $$\inf_{z\in\O}\mu(s,z)\ge c_0(s)>0,$$ 
     \item $ u $ is given by (iii) and for all $s\geq 0$, there is $c_1(s)>0$ such that 
     $$\inf_{z\in \O}\partial_s \mu(s,z)\ge c_1(s)>0,$$ 
     \item $ u $ is given by (iv) and for all $z\in\O,$ $\mu(\cdot,z) \in C^{2}(\R_+)$ and is strictly concave in $s$, {\it{i.e.}},
     $$\partial_{ss}\mu(s,z)<0 \quad \mathrm{for}\ \mathrm{all} \quad (s,z)\in [0,s_{in}]\times \O.$$
 \end{itemize}
 Let $(s_0,f_0)\in\D$. Then,  the unique solution $(s,f)$ of \eqref{main} with the corresponding control satisfies
 $$\lim_{t\rightarrow +\infty}s(t)= \sigma \; \; \mathrm{and} \; \;  \lim_{t\rightarrow +\infty}m(t)= s_{in}-\sigma.$$
 \end{lemma}

%\begin{remark}
%\end{remark} 
% \begin{remark}
%   Thanks to the proof of this result, we will see that the condition  $(s_0,f_0)\in \D$ can be relaxed in case $(ii)$ or $(iv)$: indeed,  the above lemma then remains true if $f_0\ge0, f_0\in L^1(\O)$.
% \end{remark}

\begin{remark} 
The preceding conditions are satisfied if $\mu$ is of Monod type. Additionally, thanks to the behavior of $s$ and the uniform bounds on the $L^1$-norm of $f$, one can  show that the above auxostat-type controls are admissible if $u_{max}$ is chosen large enough.
\end{remark}
 
\begin{proof}
We treat the three cases separately and, for convenience,  we start by (ii). 
\smallskip

\paragraph{\bf Case where $u$ is given by $(ii)$}  
In that case,  for all $(t,z)\in \R_+^*\times \Omega$, \eqref{main} rewrites   
\begin{equation*}\label{eq:autoxat-ii}
    \begin{dcases}
        \partial_t f(t,z)=\left(\mu(s(t),z) -\frac{1}{s_{in}-\sigma}\int_{\O}\mu(s(t),z')f(t,z')\;dz'\right)f(t,z) +\alpha\Delta f(t,z)& t>0, \ z\in\O,\\
        \dot s(t)=-\int_{\O}\mu(s(t),z)f(t,z)\,dz + \frac{s_{in}-s(t)}{s_{in} -\sigma} \int_{\O}\mu(s(t),z)f(t,z)\,dz & t>0,\\
        \partial_{\vec{n}} f(t,z)=0 &t>0, \ z\in\partial\O,\\
s(0)= s_{0}>0,\quad f(0,z)=f_0(z)\ge 0 \quad & z\in\O.
    \end{dcases}
\end{equation*}
By rearranging the terms of the second expression and integrating the first one over $\O$, we get:
% and rearranging the terms in the second we have 
   \[\begin{dcases}
       \dot s(t)= \left( -1+ \frac{s_{in}-s(t)}{s_{in} -\sigma}\right)\int_{\O} \mu(s(t),z)f(t,z)\,dz & t>0,\\
   \dot m(t)= \int_{\O}\mu(s(t),z)f(t,z)\,dz\left(1-\frac{m(t)}{s_{in}-\sigma}\right) & t>0.
   \end{dcases}
   \]
   
   Since $(s_0,f_0)\in\D$, the quantities $\mu(s(t),z)$ and $f(t,z)$ remain   positive. 
   Using the ODE\footnote{ordinary differential equation} satisfied by $s$, we have
   $\dot s(t)>0$ when  $s(t)< \sigma$ and $\dot s(t)<0$ when $s(t)> \sigma$. Moreover, if there is $t_0\ge0$ such that $s(t_0)=\sigma$, then, by Cauchy-Lipschitz's Theorem, we have $s(t)=\sigma$ for all $t\ge t_0$.
   Similarly, we find that $\dot m(t)>0$ when $m(t)<s_{in}-\sigma$  and that $\dot{m}(t)<0$ when  $m(t)>s_{in}-\sigma$. From this, we deduce the following inequalities:
    $$\max(\sigma,s_0)\ge s(t)\ge \min(\sigma,s_0)\quad \text{and} \quad \max(s_{in}-\sigma,m(0))\ge m(t)\ge \min(s_{in}-\sigma,s_0)$$
    for all $t\geq 0$. 
   By continuity of $\inf_{\O}\mu(s,\cdot)$ with respect to $s$ and from the previous estimates on $s$, we deduce that there exist constants $C_2\geq 0$ and $C_3\geq 0$ such that  
   $$\forall t \geq 0, \; \; \forall z\in \Omega, \; \; 0\leq C_2\le \mu(s(t),z)< C_3.$$ 
   %and  thus $\ds{\int_{\O}\mu(s(t),z)f(t,z)\,dz\sim m(t)}$ which stays bounded above and below and stay away from zero.  
   Moreover, we have $C_2>0$ since $s_0>0$.  
Going back to the ODEs satisfied respectively by $s$ and $m$, we conclude that $s(t)\to \sigma$ and  $m(t)\to s_{in}-\sigma$ as $t\rightarrow +\infty$. 
\medskip

\paragraph{\bf Case where $u$ is given by $(iii)$.}
In that case,  for all $t>0$ and all $z\in\O$, \eqref{main} rewrites 
\begin{equation*}\label{eq:autoxat-iii}
    \begin{dcases}
        \partial_t f(t,z)=\left(\mu(s(t),z) -\frac{1}{s_{in}-s(t)}\int_{\O}\mu(\sigma,z')f(t,z')\;dz'\right)f(t,z) +\alpha\Delta f(t,z)& t>0, z\in\O,\\
        \dot s(t)=-\int_{\O}\mu(s(t),z)f(t,z)\,dz + \frac{s_{in}-s(t)}{s_{in} -s(t)} \int_{\O}\mu(\sigma,z)f(t,z)\,dz& t>0,\\
        \partial_{\vec{n}} f(t,z)=0 &t>0, \; z\in\partial\O,\\
        s(0)= s_{0}>0,\quad f(0,z)=f_0(z)\ge 0& z\in\O.
    \end{dcases}
\end{equation*}
By rearranging the terms of the second expression, we find that
$$
\dot s(t)= \int_{\O}[\mu(\sigma,z)-\mu(s(t),z)]f(t,z)\,dz =(\sigma -s(t))\iint_{\tilde K} \partial_s\mu(\tau \sigma + (1-\tau)s(t),z)f(t,z)\, d\tau dz,
$$
where $\tilde K:=\Omega \times [0,1]$. 
Since $\partial_s\mu(\tau ,z)>0$ and $f>0$, we deduce as previously that 
$\dot{s}(t)>0$ when  $s(t)< \sigma$ and that $\dot{s}(t)<0$ when $s(t)>\sigma$. Moreover, if there is $t_0\ge0$ such that $s(t_0)=\sigma$, and thus we have 
$s(t)=\sigma$ for all $t\ge t_0$.
As a consequence, for all $t\ge 0,$ $$ \max(\sigma,s_0)\ge s(t)\ge \min(\sigma,s_0).$$
To conclude, it remains to establish upper and lower bounds for the above double integral  
away from zero. Integrating over $\O$ the equation satisfied by $f$ yields
\begin{equation}\label{lem:autoxat-iii}
 \dot m(t)= \int_{\O}\mu(s(t),z)f(t,z)\,dz-\frac{m(t)}{s_{in}-s(t)}\int_{\O}\mu(\sigma,f(t,z))\,dz, \quad t\geq 0. 
\end{equation}
Now assume that $s_0>\sigma$.  %$s_0\ge s(t)>\sigma$ for all $t\ge 0$, 
It follows that $s$ is decreasing and since $\mu$ is increasing in $s$, we  obtain the following differential inequalities:
\[
\begin{dcases}
    \dot m(t)\ge \int_{\O}\mu(\sigma,z)f(t,z)\,dz\left(1  -\frac{m(t)}{s_{in}-s_0}\right),\\
    \dot m(t)\le m(t)\left( \sup_{z\in\O}\mu(s_0,z)  - \frac{\inf_{z\in \O}\mu(\sigma,z)}{s_{in}-\sigma} m(t)\right),
\end{dcases}
\]from which we deduce that for all $t\ge 0,$ 
$$ \max\left( m(0), \frac{\sup_{z\in\O}\mu(s_0,z)}{\inf_{z\in\O}\mu(\sigma,z)}(s_{in}-\sigma)\right)\ge  m(t)\ge \min\left(m(0),(s_{in}-s_0)\right).$$ 
Likewise, when $s_0<\sigma $,  
$s$ is increasing and from \eqref{lem:autoxat-iii}, we obtain for all $t\geq 0$ 
\[
\begin{dcases}
    \dot m(t)\ge \int_{\O}\mu(\sigma,z)f(t,z)\,dz\left(\inf_{z\in\O}\mu(s_0,z) -\sup_{z\in\O}\mu(\sigma,z)\frac{m(t)}{s_{in}-\sigma}\right),\\
    \dot m(t)\le \int_{\O}\mu(\sigma,z)f(t,z)\,dz \left( 1   - \frac{m(t)}{s_{in}-s_0} \right),
\end{dcases}
\]from which we deduce that for all $t\ge 0$ 
$$ \max\left( m(0), s_{in}-s_0\right)\ge m(t)\ge \min\left(m(0),\frac{\inf_{z\in\O}\mu(s_0,z)}{\sup_{z\in\O}\mu(\sigma,z)}(s_{in}-\sigma)\right).
$$
From the lower and upper bounds on $s(t)$ and since for all $\tau \in [0,1]$, 
$\partial_s \mu(\tau,z)>C_1(\tau)$ (uniformly in $z$),
there exist positive constants $C_4$ and $C_5$ such that 
%for all $\tau\in[0,1]$ and $t\ge0$ 
$$\forall \tau \in [0,1], \; \forall t \geq 0, \; \; \forall z\in \Omega, \; \; C_4\le \partial_s\mu(\tau\sigma+(1-\tau)s(t),z)\le C_5.$$ 
Combining with the lower and upper bounds on $m$ yields

\[\iint_{\tilde K} \partial_s\mu(\tau \sigma + (1-\tau)s(t),z)f(t,z)\, d\tau dz \le C_5\max\left(m(0), (s_{in}-s_0), \frac{\sup_{z\in\O}\mu(s_0,z)}{\inf_{z\in\O}\mu(\sigma,z)}(s_{in}-\sigma)\right)\]
     and
     \[   \iint_{\tilde K} \partial_s\mu(\tau \sigma + (1-\tau)s(t),z)f(t,z)\, d\tau dz   \ge C_4\min\left\{m(0), s_{in}-s_0, \frac{\inf_{z\in\O}\mu(s_0,z)}{\sup_{z\in\O}\mu(\sigma,z)}(s_{in}-\sigma)\right\}.\]
We can then conclude as in the previous case  that $s(t)\to\sigma$ and $m(t)\to s_{in}-\sigma$ as $t\rightarrow +\infty$. 
\medskip
 
\paragraph{\bf Case where $u$ is given by $(iv)$.}
In that case,  for all $t>0$ and all $z\in\O$, \eqref{main} rewrites

\begin{equation}\label{eq:autoxat-iv}
\begin{dcases}
\partial_t f(t,z)=\left(\mu(s(t),z)- \frac{1}{s_{in} -\sigma}\int_{\O}\mu(\sigma,z')f(t,z')\,dz'\right)f(t,z) +\alpha \Delta_{z} f(t,z)&t>0, z\in\O, \\
\dot s(t)=-\int_{\O}\mu(s(t),z')f(t,z')\,dz' +\frac{1}{s_{in} -\sigma}\int_{\O}\mu(\sigma,z)f(t,z)\,dz(s_{in}-s(t))& t>0,\\
\partial_{\vec{n}} f(t,z)=0 & t>0, z\in\partial\O,\\
s(0)= s_{0}>0,\quad f(0,z)=f_0(z)\ge 0 & z\in\O.
\end{dcases}
\end{equation} 
Recall that $M(t)=s(t)-s_{in}+m(t)$ fulfills $ M'(t)=-u(t)M(t)$ so that $$M(t)=M(0)e^{-\int_0^t u(\tau)d\tau}, \quad t\geq 0.$$ 
\begin{claim}
One has $M(t)\to 0$ as $t\rightarrow +\infty$. 
\end{claim}
\begin{proof}
    It is enough to show that $\int_0^t u(\tau)d\tau\to +\infty$ as $t\rightarrow +\infty$. Suppose by contradiction that this is not the case. 
    As a result, the map
 $t\mapsto \int_{\O}\mu(\sigma,z)f(t,z)\,dz$ belongs to $L^1(\R_+)$, 
therefore  $\int_{\O}\mu(\sigma,z)f(t,z)\,dz\to 0$ as $t\to +\infty$. 
Recall that $$\sup_{z\in \O}\mu(\sigma,z)\ge \mu(\sigma,z)\ge \min_{z\in \O}\mu(\sigma,z)>0,$$  
which gives $m(t)\rightarrow 0$ as $t\rightarrow +\infty$. 
Now, observe that 
$$
\forall t \geq 0, \; \; s(t)= s_{in} -m(t) +(s(0)-s_{in}+m(0))e^{-\int_{0}^{t}u(\tau)\,d\tau},
$$
hence, $s(t)$ admits a limit $s^*>0$ as $t\rightarrow +\infty$ and we have 
$$
s^*=s_{in}\left(1 -e^{-\int_{0}^{+\infty}u(\tau)\,d\tau}\right) +(s(0)+m(0))e^{-\int_{0}^{+\infty}u(\tau)\,d\tau}. 
$$
Obviously, there is $t_0\geq 0$ such that for all $t\ge t_0$, one has 
$s(t)\ge s^*/2$ and thus $\mu(s(t),z)\ge \mu(s^*/2,z)$. % for all $t\ge t_0$. 
By integrating over $\O$ the equation satisfied by $f$ in \eqref{eq:autoxat-iv}, we obtain

\begin{align*}
\dot{m}(t)&= \int_{\O}\mu(s(t),z)f(t,z)\,dz -u(t)m(t)\\
&\ge \int_{\O}\mu(s^*/2,z)f(t,z)\,dz -\frac{1}{s_{in}-\sigma}\left(\sup_{z\in \O}\mu(\sigma,z)\right)m^2(t),\\
&\ge (\inf_{z\in \O}\mu(s^*/2,z)) m(t) -\frac{1}{s_{in}-\sigma}\left(\sup_{z\in \O}\mu(\sigma,z)\right)m^2(t) 
\end{align*}
for all $t\geq t_0$. 
$$m(t)\ge C:= \min\left\{m(t_0),\frac{s_{in}-\sigma}{2}\frac{\inf_{z\in \O}\mu(s^*/2,z)}{\sup_{z\in \O}\mu(\sigma,z)}\right\}>0 
$$
%\quad \text{ for all} \quad t\ge t_0,$$ 
for all $t\geq t_0$. This contradicts the fact that $m(t)\to 0$. 
\end{proof}
From now on we follow the argumentation used in \cite{Bayen2022}, that is, we rewrite the ODE satisfied by $s$ in the following way:
\begin{align*}
\dot{s}(t)&= -\int_{\O}\mu(s(t),z')f(t,z')\,dz' +\frac{1}{s_{in} -\sigma}\int_{\O}\mu(\sigma,z')f(t,z')\,dz(s_{in}-s(t))\\
& =(s_{in}-s(t))\left(\int_{\O} \left[\frac{\mu(\sigma,z)}{s_{in} -\sigma}-\frac{\mu(s(t),z')}{(s_{in}-s(t))}\right]f(t,z')\,dz'\right)\\
& =(s_{in}-s(t))(\sigma -s(t))\left(\int_{\O} \left[\int_{0}^1 \partial_s \rho(s(t)+\xi(\sigma -s(t)),z )\,d\xi\right]f(t,z')\,dz'\right)
\end{align*}
where $\ds{\rho(s,z):=\frac{\mu(s,z)}{s_{in}-s}}$.
Note that $\partial_s \rho(s,z)= \frac{\partial_s \mu(s,z)}{s_{in}-s}+ \frac{\mu(s,z)}{(s_{in}-s)^2}> 0$ for all $(s,z)\in  (0,s_{in})\times \O$ if $\mu(\cdot,z)$ is strictly concave and smooth. Indeed, under this assumption, %since $\mu$ is concave 
a Taylor expansion gives $$\mu(s_{in},z)=\mu(s,z) +\partial_s\mu(s,z)(s_{in} -s)+ \frac{1}{2}\int_{s}^{s_{in}}(s_{in}-\tau)^2\partial_{ss}\mu(\tau,z) \,d\tau$$
and therefore $$\mu(s_{in},z) -  \frac{1}{2}\int_{s}^{s_{in}}(s_{in}-\tau)^2\partial_{ss}\mu(\tau,z) \,d\tau=(s_{in}-s)^2\partial_s\rho(s,z)>0.$$
{\it{First case}}. If there is $t_0\ge0$ such that  $s(t_0)=\sigma$ then, $s(t)=\sigma$ for all $t\ge t_0$ and, obviously, $s(t)\to \sigma$ as $t\to +\infty$. 
Since $M(t)\to0,$ we deduce that $ m(t)\to s_{in}-\sigma$ as $t\to +\infty$ in this case. 
\smallskip
\\
{\it{Second case}}. Suppose  that $s_{in}>s_0>\sigma$. From \eqref{eq:autoxat-iv},  $s$ is decreasing and since $s(t)\ge  \sigma$ for all $t\geq 0$, $s$ converges to some value $\bar s \in [\sigma,s_{in})$ as $t\rightarrow +\infty$. Using that $M(t)\to0$, we have $m(t)\to s_{in}-\bar s>0$ as $t\rightarrow +\infty$. Therefore  there exists  $t_1$ such that for all $t\ge t_1$ , $2(s_{in}-\bar s) \ge m(t)\ge \frac{s_{in}-\bar s}{2}$. As a consequence, there are positive constants $C_6,C_7$ and $t_3\geq 0$ such that for all $t\ge t_3$, one has
$$0\le C_6(s_{in}-s(t))(\sigma -s(t))\le \dot s(t) \le C_7 (s_{in}-s(t)).(\sigma -s(t)). $$
These lower and upper bounds on $s$ show that $s(t)\to \sigma$ as $t\to +\infty$ and thus $m(t)\to s_{in}-\sigma$.
\smallskip
\\
{\it{Third case}}. The last case where $s_{in}>s_0>\sigma$ in analogous to the second one. 
\end{proof}

We now derive the stationary equation satisfied by the limiting population.
For this purpose, let $\varphi_\sigma$ be a positive eigenfunction associated to the principal eigenvalue $\lambda_1$ of the spectral problem 
 \begin{equation}\label{eq:autoxat-spec}
\begin{dcases}
\alpha \Delta_{z} \varphi(z)  +\mu(\sigma,z)\varphi =-\lambda\varphi\quad &\text{ for }  z\in\O,\\
\partial_{\vec{n}} \varphi(z)=0 \quad  &\text{ for  } x\in\partial\O.
\end{dcases}
\end{equation} 
Having constructed this function, we now show the following existence result. 
\begin{lemma}\label{lem:autoxat-equi}
Let $\sigma\in(0,s_{in})$ and $\mu$ satisfying Hypotheses \ref{hyp0}-\ref{hyp1}.  Assume further that at least one of the following assumption holds:
 \begin{itemize}
     \item $ u $ is given by (ii) and for all $s>0$, there is $c_0(s)>0$ such that
     $$\inf_{z\in\O}\mu(s,z)\ge c_0(s),$$ 
     \item $ u $ is given by (iii) and for all $s \geq 0$, there is $c_1(s)>0$ such that 
     $$\inf_{z\in \O}\partial_s \mu(s,z)\ge c_1(s),$$ 
     \item $ u $ is given by (iv) and for all $z\in\O,$ $\mu(\cdot,z) \in C^{2}(\R_+)$ and is strictly concave in $s$, {\it{i.e.}}, 
     $$\partial_{ss}\mu(s,z)<0 \quad \mathrm{for \ all} \quad (s,z)\in [0,s_{in}]\times \O.$$
 \end{itemize}
Then, there exists a unique bounded positive stationary solution $(\bar f, \bar s)$ to  \eqref{main} associated with the corresponding control.  
Moreover, one has $\bar s=\sigma$,  and $\bar f =(s_{in}-\sigma)\frac{\varphi_\sigma}{\|\varphi_\sigma\|_{1}}$.
\end{lemma} 
 \begin{proof}
 
 We only prove the result in $(iv)$, the other cases being similar. 
 Let $(\bar s,\bar f)$ be a stationary solution to \eqref{eq:autoxat-iv} with $\bar f\in L^{\infty}$ and $\bar f> 0$, that is,  
 \begin{equation}\label{eq:autoxat-iv-sta}
\begin{dcases}
\alpha \Delta_{z} \bar f(z)+\left(\mu(\bar s,z)- \frac{1}{s_{in} -\sigma}\int_{\O}\mu(\sigma,z')\bar f(z')\,dz'\right)\bar f(z) =0 &  \text{ for } z\in\O,\\
-\int_{\O}\mu(\bar s,z')\bar f(z')\,dz' +\frac{1}{s_{in} -\sigma}\int_{\O}\mu(\sigma,z')\bar f(z')\,dz(s_{in}-\bar s)=0 \\
\partial_{\vec{n}} \bar f(z)=0 & \text{ for } z\in\partial\O.
\end{dcases}
\end{equation} 
Let us consider the couple  $(\sigma, (s_{in}-\sigma)\frac{\varphi_\sigma}{\|\varphi_{\sigma}\|_{1}} )$, then a simple computation shows that such couple  is a  solution to \eqref{eq:autoxat-iv-sta}. 
Let us now prove that it is the only one. Assume that $(\bar s,\bar f )$ is another positive stationary solution.  
Since $\mu$ satisfies \Cref{hyp0},  from standard elliptic regularity we deduce that $\bar f \in C^{2,\alpha}(\O)$. 
The second equation in \eqref{eq:autoxat-iv-sta} now implies that  
 $$
 \int_{\O}\bar f(z)\left[\frac{\mu(\bar s, z)}{s_{in} -\bar s} - \frac{\mu(\sigma, z)}{s_{in} -\sigma} \right]\,dz =0. 
 $$
Since $\bar f>0$ in $\O$ and $\mu(\cdot,z)$ is strictly concave for all $z$, the quantity $\left[\frac{\mu(\bar s, z)}{s_{in} -\bar s} - \frac{\mu(\sigma, z)}{s_{in} -\sigma} \right]$ has a constant sign. From the above equation we deduce that for every $z\in \Omega$,
 $$
\left[\frac{\mu(\bar s, z)}{s_{in} -\bar s} - \frac{\mu(\sigma, z)}{s_{in} -\sigma} \right]=0, 
 $$
 hence, $\bar s =\sigma$. 
 Carrying this information over  the equation for $\bar f$ shows that $\bar f$ is a positive  stationary solution
 to 
 \begin{equation*}
\begin{dcases}
\alpha \Delta_{z} \bar f(z)+\left(\mu(\sigma,z)- \frac{1}{s_{in} -\sigma}\int_{\O}\mu(\sigma,z')\bar f(z')\,dz'\right)\bar f(z) =0 \quad &\text{ for  }  z\in\O,\\
\partial_{\vec{n}} \bar f(z)=0 \quad &\text{ for  } z\in\partial\O.
\end{dcases}
\end{equation*} 
 This means exactly that $\bar f$ is a positive eigenfunction to \eqref{eq:autoxat-spec}. 
 From the properties of the eigenpair $(\lambda_1,\varphi_\sigma)$ (see \cite{Courant2024}), 
we deduce that there is $\theta>0$ such that $\bar f=\theta \varphi_\sigma$ with
 $\lambda_1:=-\frac{1}{s_{in}-\sigma}\int_{\O}\mu(\sigma,z)\bar f(z)\,dz$.
This gives $$\theta=\frac{-\lambda_1 (s_{in}-\sigma)}{\int_{\O}\mu(\sigma,z)\varphi_\sigma(z)\,dz}.$$
To conclude the proof, we integrate \eqref{eq:autoxat-spec} over $\O$ to obtain  
$$-\lambda_1= \frac{\int_{\O}\mu(\sigma,z)\varphi_\sigma(z)\,dz}{\int_{\O}\varphi_\sigma } \quad \mathrm{and} \quad \bar f =\frac{(s_{in}-\sigma)}{\|\varphi_{\sigma}\|_1}\varphi_\sigma .$$
 \end{proof}
 
To end up the proof of \Cref{th:stabi-other}, we need to prove convergence of solutions (associated with an auxostat-type control) to the unique solution of the corresponding stationary equation. 
\begin{lemma}{\label{lem34}}
    Let $\sigma\in(0,s_{in})$ and $\mu$ satisfying Hypotheses  \ref{hyp0}-\ref{hyp1}.  Assume further that one the following assumption holds:
 \begin{itemize}
     \item $ u $ is given by (ii) and for all $s>0$, there is $c_0(s)>0$ such that
     $$\inf_{z\in\O}\mu(s,z)\ge c_0(s),$$ 
     \item $ u $ is given by (iii) and for all $s \geq 0$, there is $c_1(s)>0$ such that 
     $$\inf_{z\in \O}\partial_s \mu(s,z)\ge c_1(s),$$ 
     \item $ u $ is given by (iv) and for all $z\in\O,$ $\mu(\cdot,z) \in C^{2}(\R_+)$ and is concave in $s$, \textit{i.e.},
     $$\partial_{ss}\mu(s,z)<0 \quad \text{for all} \quad (s,z)\in [0,s_{in}]\times \O.$$
 \end{itemize}
Then, for every $(s_0,f_0)\in \D$, the unique solution  $(s,f)$ to \eqref{main},  associated with the corresponding control, converges in $L^2$ to  $(\sigma,(s_{in}-\sigma)\frac{\varphi_\sigma}{\|\varphi_\sigma\|_{1}})$ as $t\to +\infty$. 
\end{lemma}
\begin{proof}
Again, we only prove the result in case $(iv)$, the other 
cases being similar. 
Recall that  $\varphi_\sigma$ is a solution to \eqref{eq:autoxat-spec}. 
Without loss of generality, we suppose that $\varphi_\sigma$ is 
normalized, {\it{i.e.}}, $\|\varphi_\sigma\|_2=1$.
Our aim is to prove that 
\begin{equation*}
    \lim_{t\to+\infty}\|f(t,\cdot)-\bar f(\cdot)\|_{2}=0.
\end{equation*}
Thanks to \Cref{lem:autoxat-equi}, $\bar f =\theta \varphi_\sigma$ with $\theta=\frac{s_{in}-\sigma}{\|\varphi_\sigma\|_1}$. To prove the lemma, it is enough to show that:
\begin{equation}\label{eq:goal-stabi}
   \lim_{t\to +\infty} \|f(t,\cdot)-\theta \varphi_\sigma(\cdot)\|_{2}=0. 
\end{equation}
To prove \eqref{eq:goal-stabi}, we decompose $L^2(\Omega)$ 
using the standard orthogonal decomposition associated with the closed subspace spanned by $\varphi_\sigma$ (see \cite{Rudin1987}): 
$$
L^2(\O)=\text{span}(\varphi_\sigma) \oplus \varphi_\sigma^{\perp},
$$
where $\varphi_\sigma^{\perp}$ is the orthogonal complement of 
$\text{span}(\varphi_\sigma)$ with respect to the usual $L^2$ product.  
From \Cref{th:existence-closeloop},  $f\in C^{1}((0,+\infty),C^{2,\beta}(\O)\cap C^{0,\beta}(\bar \O))\cap C^{0}([0,+\infty),L^1(\O))$,  hence, for all $t>0$, $f(t,\cdot)\in L^2(\O)$. Therefore,  for all $t>0$, there are $\lambda(t)\in \R$ and $h(t,\cdot)\in L^2(\Omega)$ such that 
\begin{equation}\label{eq:decomposition}
\forall z\in \Omega, \; \; f(t,z)=\lambda(t)\varphi_\sigma(z)+h(t,z),
\end{equation}
together with the orthogonality constraint $\int_\Omega \varphi_\sigma(z) h(t,z) \, dz=0$.
Multiplying \eqref{eq:decomposition} by $\varphi_\sigma$ yields %\eqref{eq:decomposition}: %, we have
\begin{equation}\label{eq:lambda}
\lambda(t) = \int_\O \varphi_\sigma(z)f(t,z)\,dz, \quad t\geq 0. 
\end{equation}
From the orthogonal decomposition, we also have
\begin{equation}\label{eq:fl2}
\int_{\O}f^2(t,z)\,dz = 
\lambda^2(t) + \int_{\O}h^2(t,z)\,dz, \quad t\geq 0. 
\end{equation}
Now, recall that $f$ is of class $C^1$ w.r.t.~$t$. Hence, 
by using the Lebesgue derivation theorem in \eqref{eq:lambda}, 
we deduce that $\lambda$ is of class $C^1$ as well.  Therefore, by using the decomposition \eqref{eq:decomposition}, the function $h$ is $C^1$ in time.  Similarly, $\varphi_\sigma$ being in $C^2(\O)$,   \eqref{eq:decomposition} implies  that $h(t,\cdot)$ is with values in  $C^2(\O)$. 
%\textcolor{red}{arrete ici}
Substituting the previous decomposition into \eqref{eq:autoxat-iv} yields 
\begin{equation} \label{eq:dot-decomp}
\dot{\lambda}(t)\varphi_\sigma+\partial_th(t,z)= \left(\mu(s(t),z)- \frac{1}{s_{in} -\sigma}\int_{\O}\mu(\sigma,z')f(t,z')\,dz'\right)f(t,z)\\ +\alpha \Delta_{z}f(t,z), 
\end{equation}
for all $(t,z)\in \R_+^* \times \Omega$. 
Using that 
$$\int_{\O} \partial_t h(t,z)\varphi_\sigma(z)\,dz=\partial_t \int_{\O}h(t,z)\varphi_\sigma(z)\,dz=0,$$ 
and multiplying \eqref{eq:dot-decomp} by $\varphi_\sigma$ and integrating over $\O$, we deduce that for $t>0$, 
$$
\dot{\lambda}(t)=\int_{\O}\left[\left(\mu(s(t),z)- \frac{1}{s_{in} -\sigma}\int_{\O}\mu(\sigma,z')\big[\lambda(t)\varphi_\sigma+ h(t,z')\big]\,dz'\right)f(t,z) +\alpha \Delta_{z}f(t,z)\right] \varphi_\sigma(z) \,dz.
$$
On the other hand, as $\alpha\Delta_z\varphi_\sigma(z)+\mu(\sigma,z)\varphi_\sigma(z)+\lambda_1\varphi_\sigma(z)=0$ over $\Omega$, we also have for $t>0$,
$$
\int_\O \alpha\Delta_z f(t,z)\varphi_\sigma(z)\,dz = \int_\O f(t,z)\alpha \Delta_z\varphi_\sigma(z)\,dz = -\int_\O f(t,z)\big(\mu(\sigma,z)\varphi_\sigma(z)+\lambda_1\varphi_\sigma(z)\big)\,dz.
$$
As a consequence, we find that for $t>0$, 
\begin{align*}
    &\dot\lambda(t)  = \int_\O \mu(s(t),z)\varphi_\sigma(z)f(t,z)\, dz-\frac{1}{s_{in}-\sigma}\int_\O \mu(\sigma,z)\big( \lambda(t)\varphi_\sigma(z)+h(t,z)\big)\,dz\int_\O f(t,z)\varphi_\sigma(z)\,dz \\
    &\hspace{8cm}-\int_\O f(t,z)\big(\mu(\sigma,z)\varphi_\sigma(z)+\lambda_1\varphi_\sigma(z)\big)\,dz \\
    & = \int_\O \big(\mu(s(t),z)-\mu(\sigma,z)\big)\varphi_\sigma(z)f(t,z)\, dz-\frac{\lambda(t)}{s_{in}-\sigma}\int_\O \mu(\sigma,z)\big( \lambda(t)\varphi_\sigma(z)+h(t,z)\big)\,dz -\lambda_1\lambda(t).
\end{align*}
From the proof of \Cref{lem:autoxat-equi}, we have $\lambda_1=\frac{-\theta}{s_{in}-\sigma}\int_{\O}\mu(\sigma,z) \varphi_\sigma(z)\,dz$ with $\theta>0$. For $t>0$, let $$a(t):=\int_\O \big(\mu(s(t),z)-\mu(\sigma,z)\big)\varphi_\sigma(z)f(t,z)\, dz,$$
 so that for $t>0$, $\dot{\lambda}(t)$ becomes:
    \begin{align*}
  \dot \lambda(t)  &= a(t)+\lambda(t)\left(-\lambda_1- \frac{1}{s_{in}-\sigma}\int_\O \mu(\sigma,z)\big( \lambda(t)\varphi_\sigma(z)+h(t,z)\big)\,dz\right)
     \\
    &= a(t)+\lambda(t)\left(\frac{\theta}{s_{in}-\sigma}\int_\O \mu(\sigma,z)\varphi_\sigma(z)\,dz-\frac{1}{s_{in}-\sigma}\int_\O \mu(\sigma,z)\big( \lambda(t)\varphi_\sigma(z)+h(t,z)\big)\,dz\right)
    \\
    &= a(t)+ \left(\frac{1}{s_{in}-\sigma}\int_\O \mu(\sigma,z)\varphi_\sigma(z)\,dz\right) \big[\lambda(t)(\theta-\lambda(t))\big]\, -\lambda(t)\int_{\O}\mu(\sigma,z)h(t,z)\,dz.
\end{align*}
For $t>0$, let us set $b(t):= \frac{1}{\gamma_0}\int_{\O}\mu(\sigma,z)h(t,z)\,dz$ where 
$\gamma_0:=\frac{1}{s_{in}-\sigma}\int_\O \mu(\sigma,z)\varphi_\sigma(z)\,dz$.  One has:
\begin{equation}\label{eq:dot-lambda}
    \dot \lambda= a(t)+ \gamma_0\lambda(t)\left(\theta - \lambda(t)-b(t)\right).
\end{equation}
Because $\varphi_\sigma>0$ solves \eqref{eq:autoxat-spec}, the strong maximum principle and Hopf’s lemma guarantee  existence of constants $0<c_\sigma<C_\sigma$ such that $c_\sigma\le \varphi_{\sigma}\le C_\sigma$ 
(see \cite{Wu2006}).
Therefore, \eqref{eq:lambda} yields
$$\forall t \geq 0, \; \; c_\sigma m(t)\le  \lambda(t) \le C_\sigma m(t). $$%=\int_{\O}f(t,z)\varphi_\sigma(z)\,dz 
Thanks to  \Cref{lem:autoxat-s-et-m}, $m(t)\to s_{in}-\sigma$ as $t\rightarrow +\infty$, therefore, we deduce that for all $t\geq 0$
\begin{equation}\label{eq:lambda-bound}
K_0\le \lambda(t)\le K_1,
\end{equation}
for some constant $K_1>K_0>0$. 
On the other hand, observe that for every $t\geq 0$, one has 
$$|a(t)|\le \|\mu(s(t),\cdot)-\mu(\sigma,\cdot)\|_{\infty} \int_{\O}\varphi_\sigma(z) f(t,z)\,dz.$$ 
Since $\lim_{t\rightarrow +\infty} s(t)= \sigma$, and in view of \eqref{eq:lambda-bound} together with Hypotheses \ref{hyp0}-\ref{hyp1} 
on $\mu$, 
  one has
$$\lim_{t\to +\infty}a(t)= 0.$$
In order to show that $\lim_{t\to +\infty}\lambda(t)= \theta$ (based on \eqref{eq:dot-lambda}), let us prove that $b(t)\to 0$ as $t\to +\infty$. 
For this purpose, we shall prove that the $L^2$-norm of $h$ in the trait variable tends to zero. 
\begin{claim}
    $\lim_{t\to +\infty}\|h\|_{2}(t)=0.$
\end{claim}
Suppose an instant that the claim is true. By definition of $b(t)$ and by using  the Cauchy-Schwartz inequality, we get 
$$|b(t)|\lesssim \left|\int_{\O}\mu(\sigma,z)h(t,z)\,dz\right| \lesssim \|h\|_{2}(t)\to 0 \; \; \mathrm{as}\; \; t\rightarrow +\infty.$$
By~\eqref{eq:dot-lambda}, $\lim_{t\rightarrow +\infty} 
\lambda(t)= \theta$, thus \eqref{eq:goal-stabi} holds true which concludes the proof of Lemma~\ref{lem34}. 
\end{proof}
Theorem \ref{th:stabi-other} then follows directly from Lemma \ref{lem:autoxat-equi} and Lemma \ref{lem34}. Note that convergence provided in the Lemma \ref{lem34} is in $L^2$. But, this is enough to obtain convergence in any $L^p$ space, $p\geq 1$, and complete the proof of \Cref{th:stabi-other}. Indeed:
\smallskip
\\
$\bullet$ if $p\in [1,2]$, H\"older's inequality  implies $\|f(t,\cdot)-\bar f(\cdot)\|_p\le C_p\|f(t,\cdot)-\bar f(\cdot)\|_2$ where $C_p\geq 0$. 
\\
$\bullet$ If $p>2$,  note that $\bar f \in L^{\infty}$, that
$f\in C^1(\R^*_+,C^{2,\beta}(\O)\cap C^{0,\beta}(\bar \O))$ satisfies \eqref{main} and that 
$t\mapsto \|f(t,\cdot)\|_1$ is uniformly bounded over $\R_+$. Thanks to Sobolev's embeddings, we get that $t\mapsto\|f(t,\cdot)\|_{\infty}$ is uniformly bounded over $\R_+$.  The convergence in $L^p$ then follows from the inequality $\|f(t,\cdot)-\bar f(\cdot)\|_p\le \|f(t,\cdot)-\bar f(\cdot) \|_{\infty}^{p-2}\|f(t,\cdot)-\bar f(\cdot)\|_2$.
\smallskip

To end-up the proof of Lemma \ref{lem34}, it remains to show the claim. 
\begin{proof}[Proof of the Claim] 
Let us first observe that $t\mapsto \|h\|_1(t)$ is uniformly bounded. 
Indeed, from the decomposition of $f$ given by \eqref{eq:decomposition}, 
we have  
$$\forall (t,z)\in \R_+\times \Omega, \; \; |h(t,z)|\le f(t,z) +\lambda(t)\varphi_\sigma(z).$$ 
 Thanks to the Cauchy-Schwartz inequality, we also have  
 $\int_{\O}\varphi_\sigma(z)\,dz \lesssim \|\varphi_\sigma\|_{2}=1$, 
 therefore,  
 $$
\forall t\geq 0, \; \;   \|h\|_1(t) = \int_{\O} |h(t,z)|\,dz \lesssim m(t) +\lambda(t)<+\infty.
 $$
Our aim now is to find a differential inequality satisfied by $\|h\|_2$ to ensure its convergence to $0$. For this purpose, 
recall that $\varphi_\sigma>0$. Using \eqref{eq:autoxat-spec}, we deduce that:
\begin{equation}{\label{tmp-ter1}} \mu(\sigma,z)= -\alpha \frac{\Delta \varphi_\sigma(z)}{\varphi_\sigma(z)} +\frac{\theta}{s_{in}-\sigma}\int_{\O}\mu(\sigma,z')\varphi_\sigma(z')\,dz'\quad \text{ for all } \quad z\in\O. \end{equation}
Therefore, \eqref{eq:autoxat-iv} can be rewritten
\[
\partial_tf(t,z)=[\mu(s(t),z) -\mu(\sigma,z)]f(t,z)+ \left[\mu(\sigma,z) -\frac{1}{s_{in}-\sigma}\int_{\O}\mu(\sigma,z')f(t,z')\,dz'\right]f(t,z) +\alpha \Delta f(t,z),\]
where $(t,z)\in \R_+^* \times \Omega$. 
For the proof, define the auxiliary function
$$ \Gamma(t):=\frac{1}{s_{in} -\sigma}\int_{\O}\mu(\sigma,z')\left[f(t,z') - \theta\varphi_\sigma(z')\right] \,dz',$$
for $t\geq 0$. The equation satisfied by $f$ over $\R_+^* \times \Omega$ becomes (recall \eqref{tmp-ter1}):
\begin{equation}\label{eq:reformulate}
  \partial_tf(t,z)  = [\mu(s(t),z) -\mu(\sigma,z)]f(t,z) - \Gamma(t)f(t,z) +\alpha \Delta f(t,z) -\alpha\frac{f(t,z)}{\varphi_\sigma(z)}\Delta \varphi_\sigma(z).
\end{equation}
Multiplying the previous equality by $f(t,z)$ and integrating over $\O$ yields
\begin{multline*}
\frac{1}{2}\frac{d}{dt}\int_{\O}f^2(t,z)\,dz=  -\Gamma(t)\int_{\O}f^2(t,z)\,dz +\int_{\O}\alpha f(t,z)\left( \Delta f(t,z) -\frac{f(t,z)}{\varphi_\sigma(z)}\Delta \varphi_\sigma(z)\right)\,dz \\
+\int_{\O}(\mu(s(t),z) -\mu(\sigma,z))f^2(t,z)\,dz.
\end{multline*}
Ingration by parts gives for all $t\geq 0$ (time is omitted):
{\small{\[\mathcal{E}(f):=\int_{\O}f(t,z)\left( \Delta f(t,z) -\frac{f(t,z)}{\varphi_\sigma(z)}\Delta \varphi_\sigma(z)\right)\,dz =-\int_\O  \left| \nabla f(t,z)\right|^2\,dz +\int_\O \nabla\left(\frac{f^2(t,z)}{\varphi_\sigma(z)}\right)\cdot\nabla\varphi_\sigma(z)\,dz.\]}}
A direct calculation shows that for $(t,z)\in \R_+^* \times \Omega$, one has:
\begin{align*}
    \left|\nabla f(t,z) \right|^2 &=  \left|\nabla \left(\frac{f(t,z)}{\varphi_\sigma(z)}\,\varphi_\sigma(z) \right)\right|^2\\
    &=\varphi^2_{\sigma}(z) \left| \nabla\left(\frac{f(t,z)}{\varphi_\sigma(z)}\right)\right|^2+2f(t,z)\,\nabla\varphi_\sigma(z)\cdot\nabla\left(\frac{f(t,z)}{\varphi_\sigma(z)}\right)+\frac{f^2(t,z)}{\varphi_\sigma^2(z)} \left| \nabla\varphi_\sigma(z)\right|^2 \\
    &=\varphi^2_\sigma(z) \left| \nabla\left(\frac{f(t,z)}{\varphi_\sigma(z)}\right)\right|^2+2\,\frac{f(t,z)}{\varphi_\sigma(z)}\,\nabla\varphi_\sigma(z)\cdot\nabla f(t,z) -\frac{f^2(t,z)}{\varphi_\sigma^2(z)} \left| \nabla\varphi_\sigma(z)\right|^2.
\end{align*}
So, by substituting the previous expression into $\|\nabla f\|_{2}^2$, we find 
    \begin{align*}
   \mathcal{E}(f) &= \int_\O \left(2\,\frac{f(t,z)}{\varphi_\sigma(z)}\,\nabla\varphi_\sigma(z)\cdot\nabla f(t,z) -\frac{f^2(t,z)}{\varphi_\sigma^2(z)} \left| \nabla\varphi_\sigma(z)\right|^2\right)\,dz  \\
   &\hspace{2cm}-\int_\O\left(\varphi^2_\sigma(z) \left| \nabla\left(\frac{f(t,z)}{\varphi_\sigma(z)}\right)\right|^2+2\,\frac{f(t,z)}{\varphi_\sigma(z)}\,\nabla\varphi_\sigma(z)\cdot\nabla f(t,z) -\frac{f^2(t,z)}{\varphi^2_\sigma(z)} \left| \nabla\varphi_\sigma(z)\right|^2\right)\,dz,\\
    &= -\int_\O\varphi_\sigma^2(z) \left| \nabla\left(\frac{f(t,z)}{\varphi_\sigma(z)}\right)\right|^2\,dz,
    \end{align*}
for $t\geq 0$ (time is omitted in $\mathcal{E}$). As a result, we obtain
\begin{multline}\label{eq:dtf2}
    \frac{1}{2}\frac{d}{dt}\int_{\O}f^2(t,z)\,dz=  -\Gamma(t)\int_{\O}f^2(t,z)\,dz -\alpha \int_\O\varphi_\sigma^2(z) \left| \nabla\left(\frac{f(t,z)}{\varphi_\sigma(z)}\right)\right|^2\,dz\\ +\int_{\O}(\mu(s(t),z) -\mu(\sigma,z))f^2(t,z)\,dz.
\end{multline}
On the other hand, by multiplying \eqref{eq:reformulate} by $\varphi_\sigma$ and integrating over $\O$, we get
\[\frac{d}{dt}\left(\int_{\O}f(t,z)\varphi_\sigma(z)\,dz\right)\nonumber= \int_{\O}(\mu(s(t),z) -\mu(\sigma,z))f(t,z)\varphi_\sigma(z)\,dz  -\Gamma(t)\int_{\O}f(t,z)\varphi_\sigma(z)\,dz,\]
which, upon applying \eqref{eq:dot-decomp}, yields 
\begin{equation}\label{eq:dot-lambda-rw}
    \dot\lambda(t)= \int_{\O}(\mu(s(t),z) -\mu(\sigma,z))f(t,z)\varphi_\sigma(z)\,dz  -\Gamma(t)\int_{\O}f(t,z)\varphi_\sigma(z)\,dz
\end{equation}
for $t>0$. Starting from  \eqref{eq:fl2}, 
differentiating with respect to time yields
$$
\frac{1}{2}\frac{d}{dt}\int_{\O}f^2(t,z)\,dz=\lambda(t)\dot{\lambda}(t) + \frac{1}{2}\frac{d}{dt}\int_{\O}h^2(t,z)\;dz.
$$
By combining the above relation with  \eqref{eq:lambda}, \eqref{eq:fl2}, \eqref{eq:dtf2},\eqref{eq:dot-lambda-rw} and the representation of $f$ given by \eqref{eq:decomposition}, 
we obtain that $h$ satisfies the following ODE over $\R_+^*$:
\begin{multline*}
\frac{1}{2}\frac{d}{dt}\int_{\O}h^2(t,z)\,dz=  - \Gamma(t)\int_{\O}h^2(t,z)\,dz -\alpha\int_{\O}\varphi_\sigma^2(z) \left| \nabla\left(\frac{h(t,z)}{\varphi_\sigma(z)}\right) \right|^2\,dz \\
+\lambda(t)\int_{\O}(\mu(s(t),z) -\mu(\sigma,z))f(t,z)h(t,z)\,dz.
\end{multline*}
From \eqref{eq:dot-lambda-rw}, we check that $\ds{-\Gamma(t)= \frac{\dot{\lambda}(t)}{\lambda(t)} -\frac{a(t)}{\lambda(t)}}$ for $t>0$, hence:
\begin{multline*}%\label{eq-h2-bis}
\frac{1}{2}\frac{d}{dt}\int_{\O}h^2(t,z)\,dz= \left(\frac{\dot{\lambda}(t)}{\lambda(t)} -\frac{a(t)}{\lambda}\right) \int_{\O}h^2(t,z)\,dz -\alpha\int_{\O}\varphi_\sigma^2(z) \left| \nabla\left(\frac{h(t,z)}{\varphi_\sigma(z)}\right) \right|^2\,dz\\+ \lambda(t)\int_{\O}(\mu(s(t),z) -\mu(\sigma,z))h(t,z)f(t,z)\,dz .
\end{multline*}
The elliptic operator $\l=\alpha\nabla\cdot \left(\varphi_{\sigma}^2\nabla\left(\frac{\cdot}{\varphi_\sigma}\right)\right)$ having a compact resolvent in $L^2(\O)$, its spectrum is discrete. Moreover, its principal eigenvalue is $0$ and associated to the eigenfunction $\varphi_\sigma$ and $\lambda_2(\l)$, its second eigenvalue, is strictly positive.
By using the variational characterization of $\lambda_2(\l)$ (see \cite{Courant2024}), since $h \in \varphi_\sigma^{\perp}$, we can find a constant $L_0>0$ such that
$$
\forall t>0, \; \; -\alpha\int_{\O}\varphi^2_\sigma(z) \left| \nabla\left(\frac{h(t,z)}{\varphi_\sigma(z)}\right) \right|^2\,dz\le -L_0\int_{\O}h^2(t,z)\,dz.
$$
Therefore, for all $t>0$, we get that
\begin{multline}\label{eq-h2-ter}
\frac{1}{2}\frac{d}{dt}\int_{\O}h^2(t,z)\,dz\le \left(-L_0+\frac{\dot{\lambda}(t)}{\lambda(t)} -\frac{a(t)}{\lambda}\right) \int_{\O}h^2(t,z)\,dz \\+ \lambda(t)\int_{\O}(\mu(s(t),z) -\mu(\sigma,z))h(t,z)f(t,z)\,dz .
\end{multline}
Let us bound the last term in the right-hand side. 
By using again \eqref{eq:decomposition}, we obtain that %the decomposition of $f$,
\begin{multline*}
 \int_{\O}(\mu(s(t),z) -\mu(\sigma,z))f(t,z)h(t,z)\,dz= \lambda(t)\int_{\O}(\mu(s(t),z) -\mu(\sigma,z))\varphi_\sigma(z)h(t,z)\,dz\\ + \int_{\O}(\mu(s(t),z) -\mu(\sigma,z))h^2(t,z)\,dz
\end{multline*}
for all $t\geq 0$ and thus 
\[
\int_{\O}(\mu(s(t),z) -\mu(\sigma,z))f(t,z)h(t,z)\,dz
\le \eps(t)\lambda(t)C_\sigma\int_{\O} |h(t,z)|\,dz + \eps(t)\int_{\O}h^2(t,z)\,dz
\]
where we have set  $\eps(t):=\|\mu(s(t),\cdot) -\mu(\sigma,\cdot)\|_{\infty}$ for $t\geq 0$.
Using this in \eqref{eq-h2-ter}, we then get 
\begin{multline}\label{eq-h2-2}
\frac{1}{2}\frac{d}{dt}\int_{\O}h^2(t,z)\,dz\le  \left(-L_0+\frac{\dot{\lambda}(t)}{\lambda(t)} -\frac{a(t)}{\lambda(t)} +\lambda(t)\eps(t)\right) \int_{\O}h^2(t,z)\,dz \\ +\eps(t)\lambda^2(t)C_\sigma \int_{\O} |h(t,z)|\,dz.
\end{multline}
for all $t\geq 0$. Recall that $\lim_{t\rightarrow +\infty} a(t)=\lim_{t\rightarrow +\infty}\eps(t)= 0$ and that from \eqref{eq:lambda-bound}, $\lambda$ is uniformly bounded from above and below. Consequently,  there exists $t_0\ge 0$ such that one has 
\[
\forall t\geq t_0, \; \; 
\frac{a(t)}{\lambda(t)} +\lambda(t)\eps(t) \le \frac{L_0}{2}.\]
Since   $t\mapsto \|h\|_{1}(t)$ is bounded with respect to $t$, we obtain  
from \eqref{eq-h2-2}:
\begin{equation*}\label{eq-h2-3}
\forall t\geq t_0, \; \; \frac{d}{dt}\int_{\O}h^2(t,z)\,dz<  C_1\eps(t) +  2\left(-\frac{L_0}{2}+\frac{\dot{\lambda}(t)}{\lambda(t)}\right) \int_{\O}h^2(t,z)\,dz, 
\end{equation*}
for some constant $C_1>0$. Multiplying the above equation by $e^{\int_{t_0}^t \left(L_0-2\frac{\dot \lambda (\tau)}{\lambda(\tau)}\right)\,d\tau}$,  integrating over $[t_0,t]$ and collecting all terms containing $\int_{\O}h^2(t,z)\,dz$ on the left-hand side yields 
\begin{align*}
\int_{\O}h^2(t,z)\,dz&\le  e^{\int_{t_0}^t \left(-L_0+2\frac{\dot{\lambda}(s)}{\lambda(s)}\right) \,ds} \int_{\O}h^2(t_0,z)\,dz+ C_1\int_{t_0}^t\eps(\tau)e^{\int_{\tau}^t  \left(-L_0+2\frac{\dot{\lambda}(s)}{\lambda(s)}\,ds\right)}\,d\tau,\\
&=  e^{-L_0(t-t_0)} \left(\frac{\lambda(t)}{\lambda(t_0)}\right)^2 \int_{\O}h^2(t_0,z)\,dz+ C_1\int_{t_0}^t\eps(\tau)\left(\frac{\lambda(t)}{\lambda(\tau)}\right)^2 e^{-L_0(t-\tau)}\,d\tau,
\end{align*}
the last equality resulting from a direct integration.
From \eqref{eq:lambda-bound}, one has $K_0 \le \lambda(t)\le K_1 $ for all $t\ge 0$ (where $K_0>0$), and therefore
\[\int_{\O}h^2(t,z)\,dz\le e^{-L_0(t-t_0)} \left(\frac{K_1}{K_0}\right)^2 \int_{\O}h^2(t_0,z)\,dz+ C_1 \left(\frac{K_1}{K_0}\right)^2\int_{t_0}^t\eps(\tau) e^{-L_0(t-\tau)}\,d\tau, \]
for all $t\geq 0$, and thus  
\[\limsup_{t\to+\infty}\int_{\O}h^2(t,z)\,dz\le C_1 \left(\frac{K_1}{K_0}\right)^2\int_{t_0}^{\infty}\eps(\tau) e^{-L_0(t-\tau)}\,d\tau \le \frac{C_1}{L_0} \left(\frac{K_1}{K_0}\right)^2\sup_{s\in [t_0,+\infty)} \eps(s). \]
To conclude, note that $\ds\lim_{t_0\rightarrow +\infty} \sup_{s\in [t_0,+\infty)} \eps(s)=0$ since $\eps(t)\rightarrow 0$ as $t\rightarrow +\infty$. 
Finally, we obtain
\[\lim_{t\to+\infty}\int_{\O}h^2(t,z)\,dz\le 0.\]
This ends up the proof of the claim.  
\end{proof}
\smallskip
\section{Target controllability with Monod's kinetics}
{\label{secTarget}}
\subsection{Reachability of the target set via auxostat-type controls}\label{secTarget-1}
This section  analyzes reachability properties of the optimal control problem \ref{OCP0} stated in Section~\ref{sec-intro} (with the goal of proving Proposition~\ref{prop:controlability}). For completeness, we recall its formulation. 
Let $\mu$ be of Monod type, {\it{i.e.}}, $\mu(s,z)=\frac{\bar \mu  s}{r(z)+s}$ 
where $r\in C(\bar\O)$ is positive and $\bar \mu\in \R_+^*$. Consider \eqref{main} with this kinetics and $u\in \mathcal{U}$. 
The problem under investigation reads as follows: 
\begin{equation}{\label{mt2}}\inf_{u\in \U} T_u \quad \mathrm{s.t.} \; \; \quad f(t,\cdot)\in \T_0 \quad\text{ for all }\, t\ge T_u,\end{equation}
where $f$ is the positive solution to \eqref{main} associated with the control $u\in \mathcal{U}$ and the initial condition $(s_0,f_0)$ is in the subset of $\mathcal{D}$ (recall \eqref{def-D}) given by
$$
\F= (0,s_{in})\times \left\{f\in L^1(\Omega) \; | \; 
\exists (\kappa,\eta)\in \R_+^*\times \R_+^*, \; \forall z\in \Omega, \; 
f(z)\geq \kappa\mathds{1}_{B_{\eta}(\bar x)\cap\O}(z)\right\},
$$
where $\bar x \in
\mathrm{arg\,min}_{z\in \bar \Omega} (r(z))$. 
From Theorem~\ref{th-existence}, we know that the unique solution to \eqref{main} associated with a given control $u$ and starting at the initial condition $(s_0,f_0)$ at time $t=0$ is positive for all time $t>0$. 
This leads us to the following observation, in line with~\cite{Bayen2017} in the ODE setting. For future reference, let $\mathcal{H}$ be defined as:
$$
\mathcal{H}:=[0,s_{in}]\times \left\{f\in \times L^1(\O) \; | \;  f\ge 0\quad \text{and}\quad s+\int_{\O}f(y)\,dy=s_{in}\right\}. 
$$
\begin{propri}{\label{inva}} The set $\h$ is forward invariant by \eqref{main} and attractive provided that 
$u \not \in L^1(\R_+)$. 
\end{propri}
\begin{proof}
Thanks to \eqref{eq:edo-M}, for every $(s_0,f_0)\in \h$, one has 
$M(0)=s(0)-s_{in}+m(0)=0$ and thus we have $M(t)=0$ and $s(t)+m(t)=s_{in}$ for all $t\geq 0$ showing that $\h$ is invariant. Since  $M(t)=M(0)e^{-\int_{0}^t u(\tau)\,d\tau}\to 0$ as $t\rightarrow+\infty$, the result follows. 
%the attraction's property follows. attractivity  
\end{proof}
The monotonicity properties of $\mu$ directly yields
$$
\forall s\in [0,s_{in}], \; \; \bar x\in 
\mathrm{arg\,max}_{z\in \bar \Omega}(\mu(s,z)).
$$ 
Now, in light of \eqref{target-property}, it is easy to see that \eqref{mt2} is well-posed \footnote{If $k_0>r_1$, every admissible control is optimal whereas if $k_0<r_0$, then no control can steer the system to $\mathcal{T}_0$.} 
only if $k_0\in (r_0,r_1)$. Next, our goal  is to prove that the target set is both reachable and positively invariant under \eqref{main} (Proposition \ref{prop:controlability}). 
We first derive useful properties for the principal eigenpair $(\lambda_1,\varphi_{\sigma})$ of the  spectral problem:
\begin{equation}\label{eq:varphi}
\begin{dcases}
   \alpha \Delta \varphi +\mu(\sigma,z)\varphi + \lambda \varphi =0  & \quad \text{ for  } z\in\O, \\
   \partial_{\vec{n}} \varphi(z)=0 &  \quad \text{ for   } z\in\partial\O,
\end{dcases} 
\end{equation}
where $\sigma>0$. Doing so, for every 
$(\alpha,\sigma)\in \R_+^*\times (0,s_{in})$, 
let us denote by $\varphi_{\sigma,\alpha}$ the principal eigenfunction associated with \eqref{eq:varphi}  such that $\|\varphi_{\sigma,\alpha}\|_{1}=1$. 
\begin{lemma}\label{lem:varphi}
    The principal eigenfunction $\varphi_{\sigma,\alpha}$ satisfies  $\displaystyle\lim_{\alpha \to 0}\opk{\varphi_{\sigma,\alpha}}= \min_{z\in \O}r(z)$. 
\end{lemma}
\begin{proof} The proof is divided into three steps. 
\subsubsection*{Step 1:}
     For each $ \alpha>0$, let $\lambda_1(\alpha)$ be the principal eigenvalue of \eqref{eq:varphi} (the dependence of $\lambda_1(\alpha)$ w.r.t.~$\sigma$ being omitted). We begin by showing that $\lambda_1(\alpha) \to -\max_{z\in \O} \mu(\sigma,z)$ as $\alpha \downarrow 0$. 
     Recall that $\lambda_1(\alpha)$ is characterized by the variational principle 
     (see, {\it{e.g.}}, \cite{Courant2024}):
    \begin{equation}
    \label{princ_ev}
        \lambda_1(\alpha) = \min \left\{ \alpha\int_\O |\nabla \varphi|^2(z)\, dz  - \int_\O\mu(\sigma,z) \varphi^2(z)\, dz \ | \ \varphi\in H^1(\Omega)  \; \mathrm{and} \; \|\varphi\|_{2}=1\right\}.
    \end{equation}
    For all $\varphi\in H^1(\Omega)$, one has $\int_\O |\nabla \varphi|^2(z)\, dz \geq 0$. Hence,  
    the map $\alpha\mapsto\lambda_1(\alpha)$ is  increasing, thus, 
    \begin{align*}
      \alpha\int_\O |\nabla \varphi|^2(z)\, dz -\int_\Omega \mu(\sigma,z) \varphi^2(z) \, dz \ \geq \ -\max_{z\in\Omega}\mu (\sigma,z),
    \end{align*}
for all $\varphi\in H^1(\Omega)$ such that $\|\varphi\|_{2}=1$ and all $\alpha>0$.  This gives
          \begin{equation}
         \label{low_bound}
         \forall \alpha>0, \; \; 
         \lambda_1(\alpha) \geq  -\max_{z\in\Omega}\mu (\sigma,z).
     \end{equation}
     Additionally, $\alpha\mapsto \lambda_1(\alpha)$ being increasing and bounded from below, there exists $\lambda_1^0\in\R$ such that $\lambda_1(\alpha)\longrightarrow\lambda_1^0
     $
     as $\alpha\downarrow 0$. Next, our goal is to prove that 
     \begin{equation}{\label{res-interm1}}\ds{\lambda_1^0 = -\max_{z\in\Omega}\mu (\sigma,z)}.\end{equation}
     For this purpose, we introduce suitable test functions.
     Take $x_0\in\mathrm{arg\,max}\big(\mu(\sigma,\cdot)\big)$, and a sequence $(z_k)\in \text{int}(\O)^\mathbb{N}$ such that $z_k\to x_0$  as $k\rightarrow +\infty$.
     Such a sequence is introduced for technical reasons in the case where $x_0\in \partial \Omega$. 
    Now, for each $k\in\N$, let $\eta_k:\O\to\R$ be defined as
    $$
    \eta_k(z):= \left\{
     \begin{array}{ccl}
         C \exp \left( -\frac{1}{\delta_k^2-|z-z_k|^2}\right) & \mathrm{if} &|z-z_k|<\delta_k \smallskip \\
         0 & \mathrm{if} & |z-z_k| \geq \delta_k,
     \end{array}
     \right.
     $$    
where $\delta_k>0$ is such that $B_{\delta_k}(z_k)\subset \Omega$ and $C>0$ is such that $\|\eta_k\|_{2}=1$. Additionally, let us also define $\eta_{k,\eps}:\O\to\R$ as 
$$\eta_{k,\eps}(z):=\frac{1}{\eps^{n/2}} \eta_k\left(\frac{z}{\eps}\right),$$ 
where $z\in \Omega$. 
Observe that for all $(\eps,k)\in \R_+^* \times \mathbb{N}$, the function $\eta_{k,\eps}$ is of class $ C^\infty(\O)$ and that 
$$
\|\eta_{k,\eps}\|_{2}=1 \; \; \mathrm{and} \; \; \int_\O |\nabla \eta_{k,\eps}|^2(z) \,dz= \frac{1}{\eps}\int_\O |\nabla \eta_k|^2(z) \,dz.
$$
Thanks to this mollifier approach, for every function $\rho\in C(\O)$ and for all $k \in \mathbb{N}$, one has
     \begin{equation}
     \label{dirac}
         \lim_{\eps\to 0} \int_\O \rho(z)\eta_{k,\eps}^2(z) \, dz = \rho(z_k).
     \end{equation}
The proof of \eqref{dirac} being standard, it is omitted. 
 \begin{comment}
    To prove this, note that as $\|\eta_{k,\eps}\|_{2} = 1$ we can deduce that
    $$
    \int_\O \rho(z)\eta^2_{k,\eps}(z) \, dz - \rho(zs_k) = \int_\O \eta^2_{k,\eps}(z) \big(\rho(z)-\rho(z_k) \big) \, dz.
    $$
    Now take $\gamma>0$; since $\rho$ is continuous at $z_k$, there is $\eps_0>0$ such that for all $|z-z_k|<\eps_0$ we will have $|\rho(z)-\rho(z_k)|<\gamma$. By definition of $\eta_{k,\eps}$, we have $supp(\eta_{k,\eps})\subset B_{\eps\delta_k}(z_k)$, so for each $k$ there exists $\eps_k$ such that for all $\eps\le \eps_k$ we have $\eps\delta_k< d(z_k,\partial \O)$. Therefore by choosing $\eps$ small enough we achieve  $\supp (\eta_{k,\eps})\subset B_{\eps_k\delta_k}(z_k)\subset \O$ and we obtain 
    $$
    \left|\int_\O \eta^2_{k,\eps}(z) \big(\rho(z)-\rho(z_k) \big) \, dx\right| \leq \int_\O \eta_{k,\eps}^2(z) \big|\rho(z)-\rho(z_k) \big| \, dx \leq \gamma \int_\O \eta_{k,\eps}^2(z)\, dz = \gamma.
    $$
    As $\gamma$ was taken arbitrary, this proves \eqref{dirac}. 
   \end{comment}
Coming back to \eqref{princ_ev}, we obtain that
$$
     \lambda_1(\alpha) \leq   \alpha\int_\O |\nabla \eta_{k,\eps}(z)|^2\, dz-\int_\Omega \mu(\sigma,z) \eta_{k,\eps}^2(z) \, dz  
        = \frac{\alpha}{\eps}\int_\O |\nabla \eta_k(z)|^2 \,dz-\int_\Omega \mu(\sigma,z) \eta_{k,\eps}^2(z) \, dz,
$$
whence, letting $\alpha \downarrow 0$, gives
$$
\lambda_1^0  \leq -\int_\Omega \mu(\sigma,z) \eta_{k,\eps}^{2}(z) \, dz. 
$$
By construction, the above integral converges to $- \mu(\sigma, z_k)$ as $\eps \downarrow 0$, hence (recall \eqref{low_bound}):
$$
\forall k \in \mathbb{N}, \; \; -\max_{z\in\Omega}\mu (\sigma,z) \leq \lambda_1^0 \leq - \mu(\sigma,z_k). 
$$
Letting $k\rightarrow +\infty$ leads to the desired conclusion \eqref{res-interm1}, thereby completing the first step.
\subsubsection*{Step 2:} Recall that for each $\alpha>0$, one has $\|\varphi_{\sigma,\alpha}\|_{1}=1$ and consider a sequence $(\alpha_n)$ such that $\alpha_n \downarrow 0$ as $n\rightarrow +\infty$. From the Banach–Alaoglu Theorem, we may assume that, up to a subsequence, $(\varphi_{\sigma,\alpha_n})$ converges weakly-$*$ 
(in the sense of measures) to a probability measure $\nu$ as $n\rightarrow +\infty$. Now, our aim is to show that $\supp(\nu)\subset \mathrm{arg\,max}_{z\in \Omega}  \mu(\sigma , z) $. 
The weak formulation of \eqref{eq:varphi} gives us,   for all $C^2$ function $\psi$ with compact support in $\Omega$,
\begin{align*}
    0 &= -\alpha \int_\Omega \nabla \varphi_{\sigma,\alpha}(z)\cdot\nabla \psi(z) \, dz + \int_\Omega \big(\mu(\sigma,z)+\lambda_1(\alpha) \big) \varphi_{\sigma,\alpha}(z) \psi(z)\, dz \\
    &= \alpha \int_\Omega \Delta \psi(z) \varphi_{\sigma,\alpha}(z)\, dz + \int_\Omega \big(\mu(\sigma,z)+\lambda_1(\alpha) \big)\psi(z) \varphi_{\sigma,\alpha}(z)\, dz,
\end{align*}
and for all $\alpha>0$. By taking $\alpha=\alpha_n$ and letting $n\rightarrow +\infty$, we get
\begin{align*} 
    \int_\Omega \left(\mu(\sigma,z)- \max_{z\in\Omega}\mu (\sigma,z) \right)  \psi(z)\, d\nu(z)=0.
\end{align*}
Since $\psi$ is arbitrary, we deduce that for $\nu-$almost every $z\in \supp \nu$,
$$ 
\mu(\sigma,z)= \max_{z'\in\Omega}\mu (\sigma,z').
$$
As a conclusion of this step, there exists a subset $A\in \O$ such that $\nu(A)=0$ and satisfying $\supp \nu \setminus A \subset \mathrm{arg\,max}_{z\in \Omega}\mu(\sigma,z)$.
\subsubsection*{Step 3:} Our last step is to show that $\lim_{\alpha \to 0}\opk{\varphi_{\sigma,\alpha}}= \min_{z\in\O}r(z)$. 
Doing so, we first prove that $\lim_{n \to +\infty}\opk{\varphi_{\sigma,\alpha_n}}= \min_{z\in\O}r(z)$. For this purpose, we start by a simple observation. Take $z_0\in\mathrm{arg\,max}_{z\in \Omega}  \mu(\sigma,z)$ and observe that
$$
\mu(\sigma,z_0) = \max_{z\in\O} \frac{\bar\mu \sigma}{r(z)+\sigma} =   \frac{\bar \mu\sigma}{\min_{z\in\O} r(z)+\sigma} =   \frac{\bar\mu \sigma}{r(z_0)+\sigma},  
$$
which gives $z_0\in\mathrm{arg\,min}_{z\in \Omega} r(z)$. 
Second, since $\supp(\nu)\setminus A\subset \mathrm{arg\, max}_{z\in \Omega} \mu(\sigma ,z)\subset \mathrm{arg\,min}_{z\in \Omega} r(z)$ we obtain, thanks to the weak-$*$ convergence of  $(\varphi_{\sigma,\alpha_{n}})$ to $\nu$, 
$$
\lim_{n\rightarrow +\infty}\opk{\varphi_{\sigma,\alpha_n}} 
=  \frac{\int_\O r(z)\, d\nu(z)}{\int_\O \, d\nu(z)} = \min_{z\in\O} r(z),  
$$
To end up the proof of step 3, suppose by contradiction that the claim is false. 
Then, there would exist $\eps>0$ and a sequence $(\alpha_n)$ such that $\alpha_n \downarrow 0$ and satisfying $|\opk{\varphi_{\sigma,\alpha_n}}-\min_{z\in\O}r(z)|>\eps$. 
But, by using the previous steps,
we can extract a converging subsequence $(\varphi_{\sigma,\alpha_{n_k}})$ such that 
$\opk{\varphi_{\sigma,\alpha_{n_k}}} \to \min_{z\in\O} r(z)$ as $k\rightarrow +\infty$. 
This contradicts $|\opk{\varphi_{\sigma,\alpha_{n_k}}}-\min_{z\in\O}r(z)|>\eps.$ This ends up the proof of step 3 and the lemma. 
\end{proof}
We are now in a position to prove Proposition \ref{prop:controlability}.
\begin{proof}[Proof of Proposition \ref{prop:controlability}] {\it{Step 1 : $\alpha>0$}}. We first prove the result for $\alpha\not=0$.
Recall that our goal is to show the existence of $\alpha_0\in \R_+^*$ such that for all $\alpha\in [0, \alpha_0)$ and for all $\;(s_0,f_0)\in \F$ there is $u\in\U$ and $T_{u}$ such that $f(t,\cdot)\in \T_0$ for all $t\geq T_u$. 
Accordingly, let us consider a
control function of the form
\begin{equation}{\label{control3}}
u(t)=
\left\{
\begin{array}{cll}
\upsilon & \mathrm{if} & t\in [0,T_0),\smallskip\\
u_\sigma(t)& \mathrm{if} & t\in [T_0,+\infty),
\end{array}
\right.
\end{equation}
where $\upsilon$ is given by \eqref{def-cst}, $\sigma\in (0,s_{in}/2)$, $u_\sigma$ is the auxostat-type control given by (iv), and $T_0$ will be fixed hereafter.  
Let  $(s_0,f_0)\in\F$. Thanks to \eqref{eq:edo-M}, there is $T_0\geq 0$ such that for all $t\geq T_0$, one has 
$s(t)+m(t)\le 2s_{in}$.
 We deduce that 
$$
\forall t \geq T_0, \; \; u_\sigma(t)=\frac{1}{s_{in}-\sigma}\int_{\O}\mu(\sigma,z)f(t,z)\,dz \le \frac{\upsilon m(t)}{s_{in}-\sigma}\le \frac{2\upsilon s_{in}}{s_{in}-\sigma}\le 4 \upsilon s_{in}\le u_{max}.
$$
With $\upsilon\leq u_{max}$, the control in \eqref{control3} is thus admissible. 
Let us denote by $(s_\alpha,f_\alpha)$ the associated solution to \eqref{main} starting at $(s_0,f_0)$ at $t=0$.  
By \Cref{th:stabi-other}, $(s_\alpha,f_\alpha)$ converges to $(\sigma,\theta\varphi_{\sigma,\alpha})$ in $L^2(\O)$ as $t\to+\infty$, where 
$\theta:=\frac{s_{in}-\sigma}{\|\vphi_\sigma\|_{1}}$. Let us given $\delta\in (0,\delta_0)$. 
By application of Lemma~\ref{lem:varphi}, since $r_0<k_0$ (recall \eqref{target-property}), there is $\alpha_0>0$, such that
\begin{equation}{\label{limit0}}
\forall \alpha \in (0,\alpha_0], \; \; 0 \leq \opk{\varphi_{\sigma,\alpha}} \leq k_0-\delta. 
\end{equation}
To complete the proof, we now analyze the convergence of $\opk{f_\alpha(t,\cdot)}$ as $t\rightarrow +\infty$. Using that 
$f_\alpha(t,\cdot)\to \bar{f}_\alpha$ in $L^2(\O)$ as $t\to\infty$, we deduce that $\opk{f_\alpha(t,\cdot)}\rightarrow \opk{\theta\varphi_{\sigma,\alpha}}$ as $t\rightarrow +\infty$. From the definition of $\mathcal{K}$, note that 
$\opk{\theta\varphi_{\sigma,\alpha}}=\opk{\varphi_{\sigma,\alpha}}$. Together with \eqref{limit0}, this implies that $\mathcal{T}_0$ is reachable for all $\alpha\in (0,\alpha_0]$. Moreover, by 
\eqref{limit0}, the solution stays in the target set for all sufficiently large times. 
\medskip

\noindent 
{\it{Step 2: $\alpha=0$}}. To conclude the proof, we have to show that $\T_0$ is reachable if $\alpha=0$ (no mutation). 
From \cite{Desvillettes2008,Mirrahimi2012a,Perthame2007}, the unique solution $(s,f)$ to \eqref{main} associated with the constant control $u(t)=\upsilon/2$ starting at $(s_0,f_0) \in \F$ at $t=0$, converges weakly-$*$ to a positive measure $\nu$ with support in the set $\ds{\mathrm{arg\,max}_{z\in \supp(f_0)}(\mu(\bar s, z )-\upsilon/2)}$ where $\bar s\in (0,s_{in})$. Note that
$$
\max_{z\in\O} \mu(\bar s,z) = \max_{z\in\O} \ \frac{\bar\mu\ \bar s}{r(z)+\bar s}  =  \frac{\bar\mu\ \bar s}{\min_{z\in\O} r(z)+\bar s},
$$
and thus $\mathrm{arg\,min}_{z\in\Omega}\, r(z) = \mathrm{arg\,max}_{z\in\Omega}\, \mu(\bar{s},z)$. Now, recalling the definition of the initial condition set \eqref{def-D}, we know that for some $\bar x \in \mathrm{arg\,min}_{z\in\Omega}\,r(z)$, one has $\bar{x}\in \supp(f_0) $, thus 
$$
\ds{\mathrm{arg\,max}_{z\in \supp(f_0)}\,\mu(\bar s, z )} = \ds{\mathrm{arg\,max}_{z\in \O}\,\mu(\bar s, z ) \cap \supp(f_0)}. 
$$
Finally we obtain
$$\supp(\nu) \subset \ds{\mathrm{arg\,max}_{z\in \O}(\mu(\bar s, z )-\upsilon/2) \cap \supp(f_0)} \subset \mathrm{arg\,min}_{z\in\Omega}\, r(z), 
$$
and thus $\opk{\nu}=r_0$. By using the weak-$*$ convergence $f(t,\cdot)\stackrel{*}{\rightharpoonup} \nu$  as $t\rightarrow +\infty$, 
we conclude as well that  there is $T_0>0$ such that for all $t\geq 0$, one has $\opk{f(t,\cdot)} <  k_0$, which ends  the proof. 
\end{proof}

\subsection{Numerical simulation of auxostat-type controllers}{\label{sec-num-auxostat}}

Numerical simulations of system~\eqref{main} were performed using the finite element software \texttt{FreeFem++}. Piecewise linear (P1) elements were employed on a one-dimensional mesh consisting of 5000 nodes, and time integration was carried out using a semi-implicit Euler scheme with a time step $\Delta t = 0.01$. To illustrate graphically the result stated in~\Cref{prop:controlability}, we compute numerically solutions to \eqref{main} where $u$ is the auxostat-type control (iv). This allows to obtain $\opk{f_\alpha(t,\cdot)}$ for various values of $\alpha \in [0,0.01]$. As initial condition, we take $f_0 \equiv 5$, $s_0=5$ and we set $\Omega := [1,3]$, $\sigma := 9$ and $s_{in} := 35$. Moreover, 
the growth function is defined as $\mu(s,z) := \frac{s}{z + s}$ and $u(t)= \frac{1}{s_{in} -\sigma}\int_{\O}\mu(\sigma,z)f(t,z)\,dz$ (case (iv)). Numerical results are depicted on Fig.~\ref{fig:K_fa}. 

\begin{figure}[ht]%K_fa.pdf = fig1.pdf
    \centering
\includegraphics[width=0.65\linewidth]{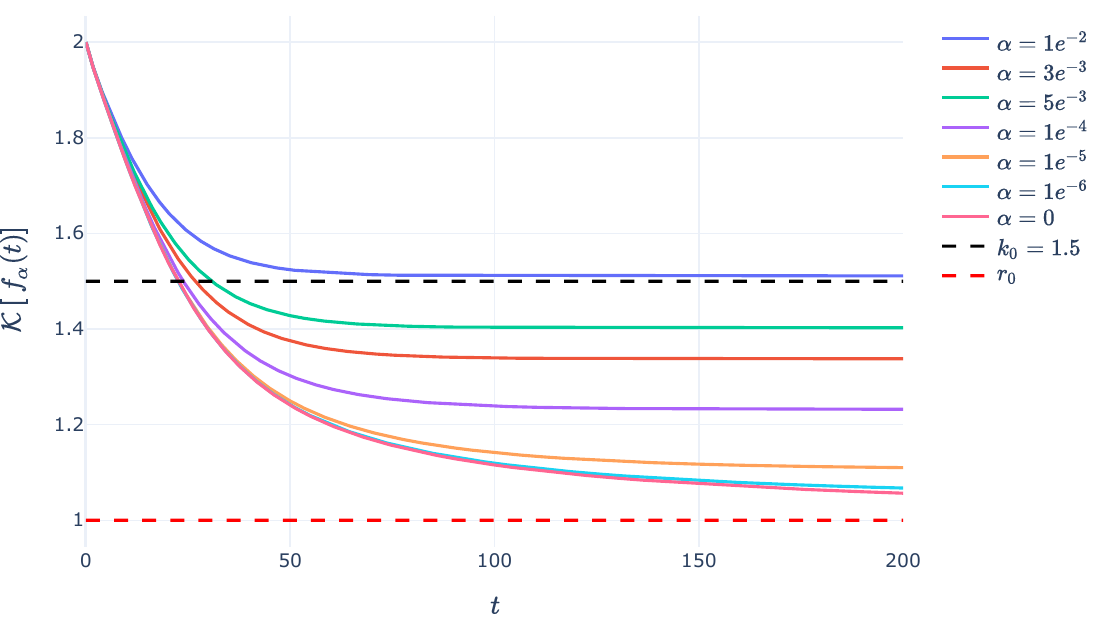}
    \caption{Behavior of the functional $\opk{f_\alpha(t,\cdot)}$ w.r.t.~time for 
    different values of the parameter $\alpha$, where $f_\alpha$ is the solution of \eqref{main} with $u(\cdot)$ given by (iv).}
    \label{fig:K_fa}
\end{figure}

As we can observe on Fig.~\ref{fig:K_fa}, for $k_0=1.5$, the target set is reached  at some time less than $35$ for every $\alpha<0.01$, confirming Proposition \ref{prop:controlability}. We also observe numerically that for $\alpha = 0.01$, the trajectory reaches an equilibrium point approximately at time $t=70$ such that $\opk{\bar f_{0.01}} \approx 1.511$. 
Consequently, for this specific value of $\alpha$, the target set is not reached using this auxostat-type control if $k_0 < 1.511$. On the other hand, when $\alpha < 0.01$, we see that the value of the functional $\opk{f_\alpha(t,\cdot)}$ decreases over time and seems to converge to a constant value as the solution approaches an equilibrium point. This behavior shows numerically that, under the control (iv), solutions reach the target set in finite time and remain in it hereafter. 
\smallskip
\section{Existence of an optimal control and numerical approximation}{\label{sec-existence}}
In this section, we prove Theorem \ref{thm-bigBoss} about the existence of a solution to the minimum time control problem  \eqref{OCP0}. We also perform numerical optimization of the time-to-target within two control subclasses: first for auxostat-type controls, and then for constant controls.
\subsection{Proof of Theorem \ref{thm-bigBoss}}\label{sec-existence-1}
Let $f_0\in \F$.  
It follows from Proposition~\ref{prop:controlability} that the state can be steered to $\mathcal{T}_0$ in finite time and, once there, remains in $\mathcal{T}_0$ for all future time. We can then consider a minimizing sequence $(T_n,u_n)$, {\it{i.e.}}, such that 
$$
\ds{T_n \to T^*:=\inf_{u\in \U}T_{u}} \quad \mathrm{as} \; \; n\rightarrow +\infty,
$$ 
and where for all $n\in \mathbb{N}$, 
$u_n$ is an admissible control and 
$T_n \geq 0$ is the first entry-time of the associated solution into the target set.
Without any loss of generality, we may assume that $(T_n)$ is non-increasing and such that $T_n \le T_0$ for all $n\in \mathbb{N}$. Our goal is to prove that $(u_n)$ converges, up to a subsequence, to an admissible control $u^*$ and that this limit $u^*$ steers the system to the target in the minimal time $T^*$. By definition of $\U$, for all $T>0$ and for all $n\in \mathbb{N}$, $u_n(t)\in [0,u_{max}]$ for a.e.~$t\in [0,T]$, therefore $(u_n)$ is uniformly bounded in $L^2([0,T])$. 
As a consequence, there is 
$u^*\in L^2([0,T])$ such that, up to a subsequence, 
$(u_n)$ weakly converges to $u^*$ in $L^2([0,T])$. By a diagonal extraction, there exists $u^*\in\U$, such that $(u_n)$ converges weakly in $L^2_{loc}([0,+\infty))$. 
We wish now to prove that the optimal time $T^*$ is achieved for the control $u^*$. 
This relies on the following result, stated as a Lemma.
\begin{lemma}{\label{lem:f*-existence}}
    There exists a solution ($f^*,s^*)$ to \eqref{main} such that  
    $$ \forall t \geq T^*, \; \; \opk{f^*(t,\cdot)}\le k_0.$$
\end{lemma}

\begin{proof}
Since $u^* \in \mathcal{U}$, \Cref{th-existence,th:existence-alpha0} guarantee the existence of a mild solution $(s^*, f^*)$ to \eqref{main} for every $\alpha \in [0, \alpha_0)$. 
To complete the proof, it remains to show that  $f^*$ satisfies 
$$ \forall t \geq T^*, \; \; \opk{f^*(t,\cdot)}\le k_0.$$
Doing so, let $t_0> T^*$ and $T:=\max(T_0, t_0+1)$.
From the weak convergence of $u_n$ to $u^*$, we have
$$
\forall t\in [0,T], \;\;  M_n(t)\to M^*(t)=e^{-\int_0^t u^*(\tau)\,d\tau}, 
$$ 
as $n \rightarrow +\infty$. In addition, by using \eqref{eq:edo-M}, the sequence $(M_n)$ is uniformly bounded  in $C^{0,1}(\R_+)$. So, up to a subsequence, we may assume that $(M_n)$ converges to $M^*$  uniformly over $[0,T]$. 
We now treat the case $\alpha>0$ and $\alpha=0$ separately. 
\smallskip

\noindent {\it{First case : $\alpha>0$}}. 
By \Cref{thm:existence1,thm:existence2}, for each $n\in \mathbb{N}$, there exist $f_n$, $\tilde f_n$ solving respectively \eqref{eq-red} and \eqref{eq-red:bis}. Moreover, for all $n\in \mathbb{N}$, one has: 
$$f_n=e^{-\int_{0}^t u_n(\tau)\,d\tau}\tilde f_n \; \; \mathrm{and} 
\; \; \tilde f_n \in C^1(\R^*_+,C^{2,\beta}(\O)\cap C^{0,\beta}(\bar \O))\cap C(\R_+,L^1(\O)).$$ 
%In addition, since $\mu$ satisfies 
Under Hypotheses \ref{hyp0}-\ref{hyp1}, by integrating \eqref{eq-red} and using that $(m_n), (u_n)$ are uniformly bounded over $[0,T]$, we obtain that the sequence $(m_n)$ is uniformly bounded in the space $C^{0,1}(\R_+)$.
So, there exists $\hat m\in C([0,T])$ such that, up to a subsequence, $(m_n)$ uniformly converges  to $\hat m$   over $[0,T]$.
From \Cref{sec1} (recall \eqref{tmp-mintime}), for all $n\in \mathbb{N}$, we have
$$\tilde f_n\le e^{C_0t}\ops{f_0(\cdot)}{t},$$ 
where $\s_t$ denotes  the Neumann heat kernel semi-group.
As a result, for all $n\in \mathbb{N}$ and for all $\delta\in(0,T)$ and all
$t\in [\delta,T]$, we have  $\|\tilde f_n\|_{\infty}(t)\le e^{C_0T} C\delta^{-\frac{d}{2}}\|f_{0}\|_1$. 
Therefore, by using that  $(m_n)$ and $(M_n)$ are bounded and Lipschitz-continuous independently of $n$, we deduce from Schauder estimates, that  the sequence $(\tilde f_n)$ is  uniformly bounded in the space $C^{1,\frac{\beta}{2}}\big ([\delta,T],C^{2,\beta}(\O)\cap C^{0,\beta}(\bar\O)\big)$ for some $\beta\in(0,1)$.
As a consequence, by taking a sequence  $(\delta_n), \delta_n \downarrow 0$ and by using a classical  diagonal extraction argument, $(\tilde f_n)$ strongly  converges  to  some function $\tilde g$ in $C([0,T],L^1(\O))\cap C^{1,\frac{\beta'}{2}}\big((0,T],C^{2,\beta'}(\O)\cap C^{0,\beta'}(\bar\O)\big)$ for some $\beta'\in (0,\beta)$. % \textcolor{red}{Terence: should we add (using classical convergence arguents letting $\delta \downarrow 0$)}\JC{OK}. 
By using that $(M_n)$ uniformly converges to $M^*$ over $[0,T]$ and that for all $n\in \mathbb{N}$, $f_n=\frac{M_n}{s_0+m(0)-s_{in}}\tilde f_n$, we also get that $f_n\to g$ locally uniformly in $(0,T]$.
Moreover, by using the mild representation of $f_n$, one has for all $t\in (0,T]$ and all $n\in \mathbb{N}$ 
$$ f_n(t,z)=\ops{f_0}{t}(z) + \int_{0}^t\ops{\mu(M_n(\tau)+s_{in} - m_n(\tau),\cdot)f_n(\tau,\cdot) - u_n(\tau) f_n(\tau,\cdot) }{t-\tau}(z)\,d\tau.$$
Hence, in the limit $n\rightarrow +\infty$, 
the weak $L^2$-convergence of $(u_n)$ and the point-wise convergence of $(f_n)$ imply that $g$ satisfies
\[ g(t,z)=\ops{f_0}{t}(z) + \int_{0}^t\ops{\mu(M^*(\tau)+s_{in} - \hat m(\tau),\cdot)g(\tau,\cdot) - u^*(\tau) g(\tau,\cdot) }{t-\tau}(z)\,d\tau\]
for all $(t,z)\in (0,T]\times \Omega$. 
Thus for all $t\in (0,T]$,  $g$ is a mild solution to \eqref{main} associated with $u^*$. By uniqueness of the mild solution (\Cref{th-existence}), we conclude that 
$g\equiv f^*_{|_{[0,T]\times \O}}$. 
Now observe that for all $t\in (T^*,T)$, there exists  $n_t\in \mathbb{N}$ such that for all $n\ge n_t,$ 
$\opk{f_n(t,\cdot)}\le k_0$. 
Therefore, by using the weak convergence of $(f_n)$  %for all $t\in (T^*,T)$ 
to $g$ over $(T^*,T)$ together 
with the inequality $m_n(t)\ge m(0)e^{-u_{max}T_0}$ we deduce that 
$\opk{f^*(t,\cdot)}=\opk{g(t,\cdot)}\le k_0$ for all $t\in (T^*,T)$.
In particular, we have $\opk{f^*(t_0,\cdot)}\le k_0$ since $t_0\in (T^*,T)$. 
The above argument being true for any $t_0>T^*$, we  deduce that 
$$\forall t \geq T^*, \; \; \opk{f^*(t,\cdot)}\le k_0.$$
Hence, $f^*$ enters permanently into the target set $\T_0$ in time $T^*$.
 \medskip
 
\noindent {\it{Second case : $\alpha=0$}}. In this case, 
 the methodology is essentially the same as when $\alpha>0$. However since the problem \eqref{main} does not exhibit regularizing properties, the argument above must be revised to handle substantially less regular spaces and weaker topologies.  
 By \Cref{th:existence-alpha0}, for all $n\in \mathbb{N}$,  $f_n\in C^1(\R_+,L^1)$ and from the definition of  $\tilde f_n$, we also have 
 $$ \int_{\O}\tilde f_n(t,z)\,dz \le e^{t\,u_{max}}(m_0+s_{in}+s_0)$$
 for all $t\ge 0$ and all $n \in \mathbb{N}$. 
 So, there is $K_0>0$ such that 
 \begin{equation}{\label{estim01}} 
 \forall n \in \mathbb{N}, \; \forall t\in [0,T], \; \;\int_{\O}\tilde f_n(t,z)\,dz \le K_0.\end{equation}
  Let $\omega$ be an arbitrary subset of $\Omega$.  
 Integrating  \eqref{eq-red:bis} over $\omega$ and using the fact that $\alpha=0$ yields
 \begin{equation}{\label{estim02}} 
 \forall n \in \mathbb{N}, \; \forall t\in (0,T], \; \;
 \left| \int_{\o}\partial_t \tilde f_{n}(t,z)\,dz \right|\le C_0\int_{\o}\tilde f_n(t,z)\,dz\le C_0 K_0, 
 \end{equation}
and  \begin{equation}{\label{estim03}}
\forall n \in \mathbb{N}, \; \forall t\in [0,T], \; \;
  \int_{\o} \tilde f_{n}(t,z)\,dz \le e^{C_0t}\int_{\o}\tilde f_0(z)\,dz. 
 \end{equation} 
For a function satisfying estimates \eqref{estim01}-\eqref{estim02}-\eqref{estim03}, we can apply the following lemma. Its application will enable us to extract a convergent subsequence, thereby completing the proof of the theorem for the case $\alpha=0$. Let $C_b(\Omega)$ denote the space bounded functions over $\Omega$. 
\begin{lemma}\label{lem:radon-compact}
     There exist a family of bounded Borel measure $(\nu_t)_{t\in[0,T]}$  and a subsequence  
      $(\tilde f_{n_k})_{k\in\N}$ of $(\tilde f_n)$ such that,  for all $t\in[0,T]$,  $\tilde f_{n_k}\stackrel{\ast}{\rightharpoonup} \nu_t$,  that is, for  any $\varphi \in C_b(\O)$, 
 $$\lim_{k \rightarrow +\infty}\, \int_{\O}\varphi(z)\tilde f_{n_k}(t,z)\, dz =\int_{\O}\varphi(z)\, d\nu_t(z).$$
 \end{lemma}

 Since $\O\subset \R^d$ is bounded, the space $(\M(\O),\mathrm{weak-}*)$ is a complete separable metric space  and the lemma follows from a generalization of the Aubin-Lions type Lemma due to Simon \cite{Simon1986}. For sake of completeness, a proof of this lemma is given in the appendix.
Let us now identify the equation satisfied by the collection of measures $(\nu_t)_{t\in[0,T]}$.
Recall that $\tilde f_n$ satisfies the following representation: 
$$
\forall t\in [0,T], \; \forall z\in \Omega, \; \tilde f_n(t,z)= f_0(z) + \int_{0}^t\mu(M_n(\tau)+s_{in} - m_n(\tau),z)\tilde f_n(\tau,z)\,d\tau.
$$ 
So, for any $\varphi\in C_b(\O)$, we have  for all $t\in[0,T]$,
\begin{equation}\label{eq:*-weak-formulation}
\int_{\O}\varphi(z)\tilde f_n(t,z)\, dz = \int_{\O}\varphi(z) f_0(z)\,dz + \iint_{\tilde K_t}
\mu(M_n(\tau)+s_{in} - m_n(\tau),z)\varphi(z)\tilde f_n(\tau,z) \,dz d\tau,
\end{equation}
where $\tilde K_t:=[0,t]\times \Omega$. 
Now, since $(m_n)$ and $(M_n)$ are uniformly bounded in $C^{0,1}([0,T])$, by using  the Arzelà-Ascoli Theorem, we may assume that, up to a subsequence,  $(M_n)$ converges pointwise to 
$M^*$ as $n\rightarrow +\infty$. We deduce that for all $t\in[0,T]$, $$m_n(t)=\frac{M_n(t)}{M_n(0)}\int_{\O}\tilde f_n(t,z)\,dz \to \frac{M^*(t)}{M^*(0)}\int_{\O}d\nu_t(z)$$ as $n\rightarrow +\infty$, 
thanks to weak-$*$ convergence of $\tilde f_n$. 
Thus, passing to the limit as $n\rightarrow +\infty$ in \eqref{eq:*-weak-formulation}, we obtain that 
for all $t\in [0,T]$ and  all $\varphi \in C_b(\O)$,
\begin{equation}\label{eq:mild-nut}
\int_{\O}\varphi(z)\, d\nu_t(z) = \int_{\O}\varphi(z) f_0(z)\,dz +  \iint_{\tilde K_t}\mu\left(M^*(\tau)+s_{in} - \frac{M^*(\tau)}{M^*(0)}\int_{\O}d\nu_\tau(z),z\right)\varphi(z) \,d\nu_{\tau}(z)\,d\tau.
\end{equation}
Fix $\varphi\in C_b(\O)$. From the above equation, since $\|\mu\|_{\infty}\le \upsilon $ %C_0$ 
and $t\mapsto \nu_t(\O)$ is uniformly bounded over $[0,T]$,  the quantity $\int_{\O}\varphi(z)\, d\nu_t(z)$ is Lipschitz-continuous w.r.t.~$t$. Consequently, it is differentiable w.r.t.~$t$ and we 
get for almost every $t\in (0,T)$,
$$
\frac{d}{dt}\left(\int_{\O}\varphi(z)\,d\nu_t(z)\right)= \int_{\O}\mu\left(M^*(t)+s_{in} - \frac{M^*(t)}{M^*(0)}\int_{\O} d\nu_t(z),z\right)\varphi(z) \,d\nu_{t}(z).
$$
Multiplying the previous equality by $e^{-\int_{0}^{t}u^*(\tau) d\tau}$ and integrating over $(0,t)$ yields
$$
\int_{0}^t e^{-\int_{0}^{\tau }u^*(s)\, ds}\frac{d}{dt}\left(\int_{\O}\varphi(z)\,d\nu_\tau(z)\right)d\tau= 
\iint_{\tilde K_t}\mu\left(M^*(\tau)+s_{in} - \int_{\O}d\zeta_\tau(z),z\right)\varphi(z) \,d\zeta_{\tau}(z)d\tau
$$
where $(\zeta_t(z))_{t\in[0,T]}$ is the  family of measure defined for $t\in[0,T]$ by  $\zeta_{t}(z):=e^{-\int_{0}^tu^*(\tau)\,d\tau}\nu_t(z)$.
Using integration by parts, the left hand side becomes
\begin{multline*}
    \int_{0}^t e^{-\int_{0}^{\tau }u^*(s) \,ds}\frac{d}{dt}\left(\int_{\O}\varphi(z)\,d\nu_\tau(z)\right)d\tau=\int_{\O}e^{-\int_{0}^{t }u^*(s)\, ds}\varphi(z)\,d\nu_t(z) - \int_{\O}\varphi(z)\,d\nu_0(z) \\ + \int_{0}^t u^*(\tau) \int_{\O}e^{-\int_{0}^{\tau} u^*(s) \,ds} \varphi(z)\,d\nu_t(z)d\tau,
\end{multline*} 
for $t\in [0,T]$, which gives % and thus we get 
$$
\int_{\O}\varphi(z)\,d\zeta_t(z)= \int_{\O}\varphi(z)f_0(z)\,dz + 
\iint_{\tilde K_t}\left[\mu\left(M^*(\tau)+s_{in} - \int_{\O}\,d\zeta_\tau(z),z\right) -u^*(\tau)\right]\varphi(z) \,d\zeta_{\tau}(z)d\tau
$$
for $t\in [0,T]$. 
The above argument being valid for any $\varphi \in C_b(\O)$. This means that $\zeta_{t}$ is a mild  solution of \eqref{main} in  $L^{\infty}([0,T],(\M^+(\O), \mathrm{weak-}*))$.
Our next aim is to show that the measures $\nu_t$ and $\zeta_t$ are absolutely continuous 
w.r.t.~the Lebesgue measure in $\O$.
For this purpose, recall that $\tilde f^*\in C^1(\R_+^*, L^1(\O))\cap C^0(\R_+,L^1(\O))$  is a solution to the system
\[\begin{dcases}
\partial_t \tilde f(t,z) = \mu\left(M^*(\tau)+s_{in} - \frac{M^*(\tau)}{M^*(0)}\int_{\O}\tilde f(t,z)\,dz,z\right)\tilde f(t,z) &\text{ for }\; t>0, z\in\O,\\
\tilde f(0,z)=f_0(z) &\text{ for  }\; z\in\O.
\end{dcases}
\]
Now, we can take $\varphi=1$ in \eqref{eq:mild-nut}. This gives for $t\in [0,T]$: % we have 
$$ \int_{\O}d\nu_t(z) = \int_{\O} f_0(z)\,dz + 
\iint_{\tilde K_t}\mu\left(M^*(\tau)+s_{in} - \frac{M^*(\tau)}{M^*(0)}\int_{\O}d\nu_\tau(z),z\right) \,d\nu_{\tau}(z)\,d\tau.$$
As well, since $\tilde f^*$ is also a  mild solution, we have for $t\in [0,T_0]$:
$$ \int_{\O}\tilde f^*(t,z)\,dz = \int_{\O} f_0(z)\,dz + 
\iint_{\tilde K_t}\mu\left(M^*(\tau)+s_{in} - \frac{M^*(\tau)}{M^*(0)}\int_{\O}\tilde f^*(t,z)dz,z\right)\tilde f^*(t,z)\,dz\,d\tau.
$$
We now argue as in the uniqueness proof of \Cref{thm:existence2}.
For $t\in [0,T]$, let us set $h(t):= \int_{\O}d\nu_t(z) -\int_{\O}\tilde f^*(t,z)\,dz $. We then obtain 
$$
\forall t\in [0,T], \; \; |h(t)|\le \upsilon\int_{0}^t|h(\tau)|\,d\tau,
$$
and by using Gr\"onwall's inequality, we deduce that for all $t\in [0,T]$, one has
$h(t)=0$. By using again the uniqueness proof of \Cref{thm:existence2}, we can check that for all $\varphi \in C_b(\O)$, one has
$$\forall t\in [0,T], \; \; \ds{\int_{\O}\varphi(z)\,d\nu_t(z)=\int_{\O}\varphi(z)f^*(t,z)\,dz}.$$
Hence, $\nu_t$ is absolutely continuous w.r.t.~the Lebesgue measure and, by the Radon-Nykodim Theorem, we deduce that $\frac{d\nu_t}{dz}=\tilde f^*$ for all $t\in[0,T]$. 
By definition of $\zeta_t$, we conclude as well that $\zeta_t$ is absolutely continuous w.r.t.~the Lebesgue measure and that we have  
$$\forall t\in [0,T], \; \; \frac{d\zeta_t}{dz}=e^{-\int_0^tu^*(\tau)\,d\tau}\tilde f^*.$$
We can now conclude that $T^*$ is the optimal time. 
Doing so, recall that  the operator $\k$ is invariant by scaling, so,  by using the weak-* convergence, we have that 
$$
\opk{f_n(t,\cdot)}=\opk{\tilde f_n(t,\cdot)} \to \opk{\nu_t}=\opk{\zeta_{t}}=\opk{f^*(t,\cdot)}.
$$
As a consequence, as in the case $\alpha>0$, we can check that for all $t\in(T^*,T)$, one has $\opk{f^*(t,\cdot)}\le k_0$ implying that $\opk{ f^*(t_0,\cdot)}\le k_0$.  
Again, the above argument is true for any arbitrary $t_0>T^*$. As a consequence, we have
$$\forall t\ge T^*, \; \; \opk{f^*(t,\cdot)}\le k_0.$$
Hence, $f^*$ enters permanently into the target set $\T_0$ in the minimal time $T^*$.

\end{proof}
As a final conclusion, regardless of whether $\alpha = 0$ or $\alpha > 0$, we have shown that there exists an admissible control $u^*$ leading to a  solution $f^*$ of~\eqref{main} that reaches and stays in the target set in minimal time $T^*$, thereby completing the proof of Theorem~\ref{thm-bigBoss}.\\

%\CA{Clau: I think the proof is correct now, I see no problem with the times! \\ Just one question, can we say anything about the smoothness of the optimal control $u^*$? I don't see clearly if we can say something and depending on this we can state that the solution will be mild or strong, but I don't know if this is relevant from an optimal control point of view, is just an observation. } \JC{ That is a good question, for the moment I don't know. Maybe finding  optimality condition would help to describe the regularity of the control.}

Next, synthesizing an optimal control  requires to establish necessary optimality conditions for \eqref{OCP0}. Although system \eqref{main} presents similarities with parabolic equations, it contains a non-local term. For this reason, deriving such conditions in our setting is beyond the scope of this paper, instead, we address this issue through numerical simulations.
\subsection{Numerical simulations of sub-optimal controls}\label{sec-existence-2}
For both theoretical and numerical reasons, the search for optimal controls is performed in two distinct classes. First, we consider the class of auxostat-type controls as defined in (iv); then, we examine the simpler class of constant controls.

\subsubsection{Time minimization in the class of auxostat-type controls (iv)}

Following the numerical scheme described in Section~\ref{sec-num-auxostat}, we performed an extensive search over 340 values of $\sigma \in (0,s_{in})$ to determine the first entry time into the target set for solutions to~\eqref{main} under the auxostat-type control (iv). This procedure allows us to identify, in a straightforward manner, an optimal pair $(\sigma^*, T(\sigma^*))$ minimizing the entry time.
\begin{figure}[htbp!]%R_sigma.pdf=fig2.pdf
\centering
\begin{subfigure}{0.48\textwidth}
  \includegraphics[width=\linewidth]{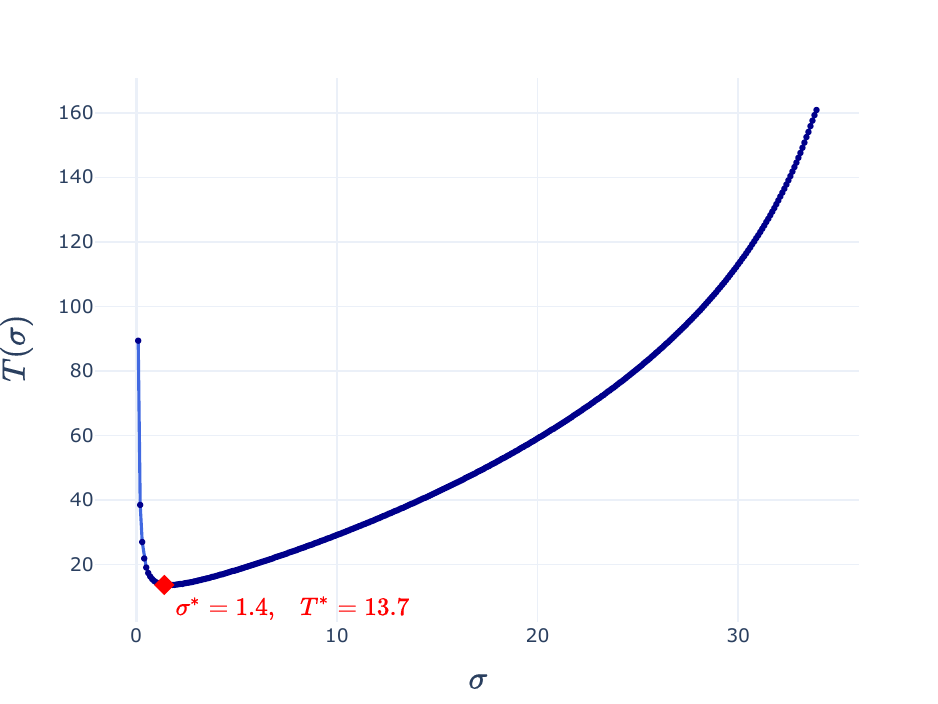}
  \caption{Entry time $T(\sigma)$ to the target set for $\sigma\in(0,s_{in})$ and solutions of \eqref{main} with autoxtat control. }
  \label{Fig:2-a}
\end{subfigure}\hfill
\begin{subfigure}{0.48\textwidth}%Kt_sigma.pdf=fig3.pdf
  \includegraphics[width=\linewidth]{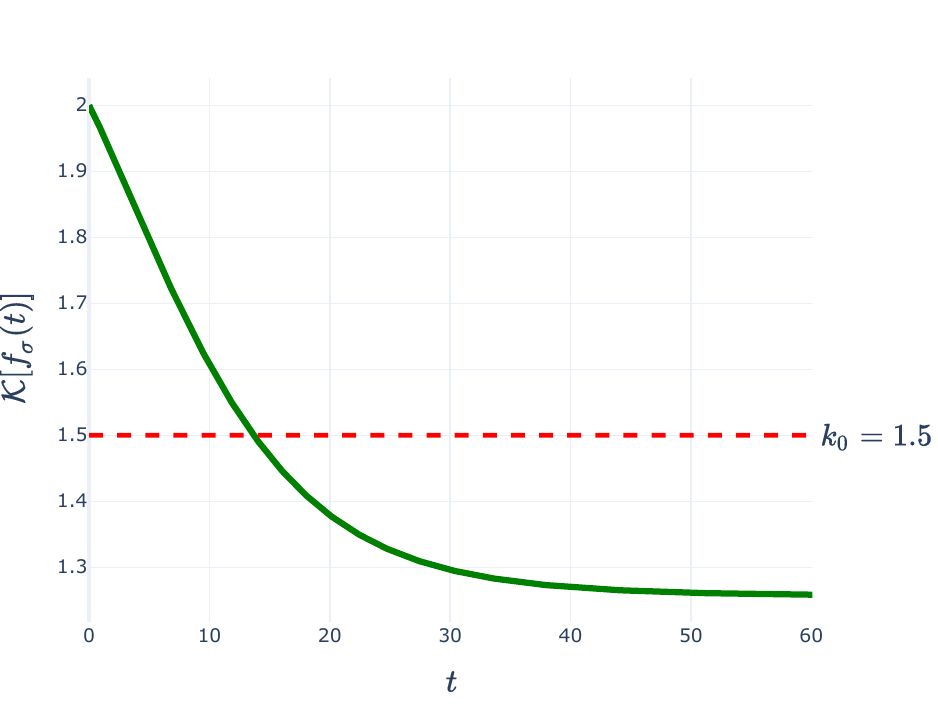}
  \caption{Function $\opk{f(t,\cdot)}$ evaluated on the solution of \eqref{main} at the optimal value $\sigma^*=1.4$, for each time step.}
  \label{Fig:2-b}
\end{subfigure}

\medskip

\begin{subfigure}{0.48\textwidth}%fT_sigma.pdf=fig4.pdf
  \includegraphics[width=\linewidth]{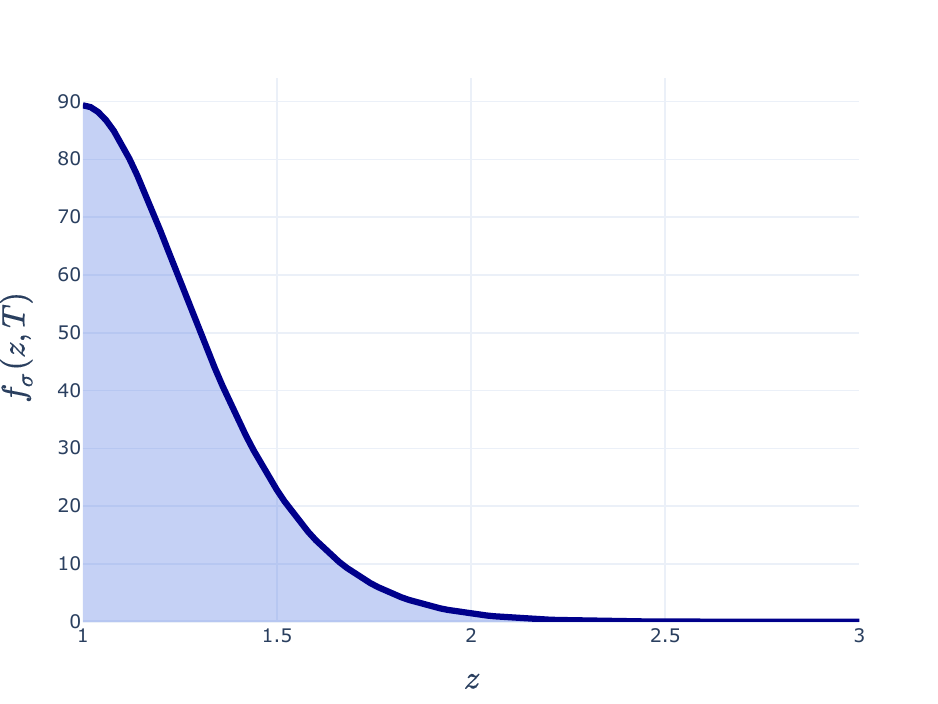}
  \caption{Final distribution of the solution of \eqref{main} for auxostat-type control using $\sigma=1.4$. }
  \label{fig:2-c}
\end{subfigure}\hfill
\begin{subfigure}{0.48\textwidth}%u-s-R_sigma.pdf=fig5.pdf
  \includegraphics[width=\linewidth]{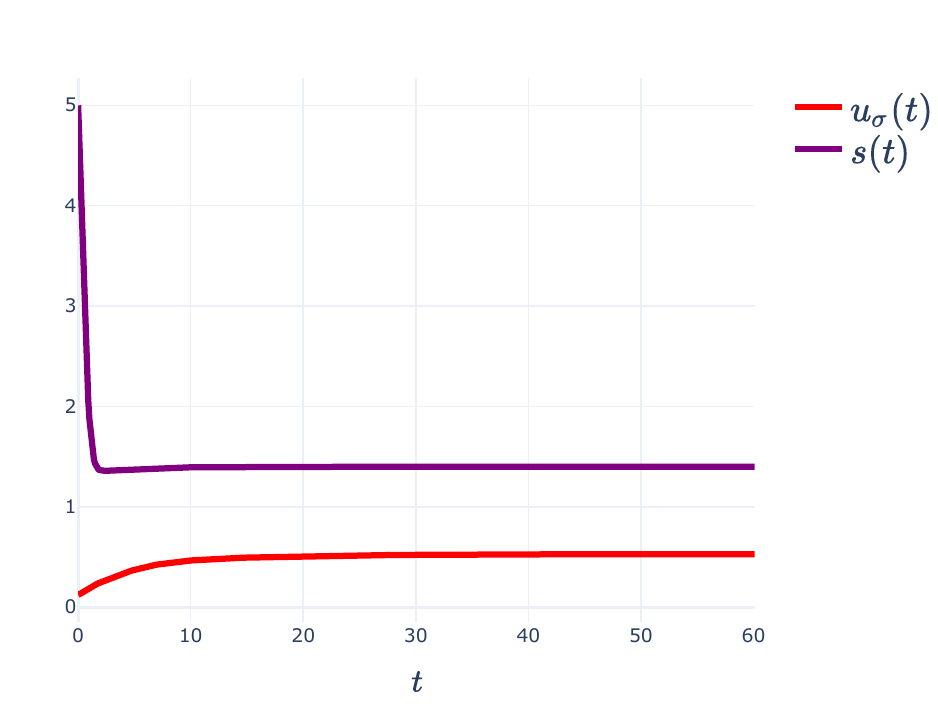}
  \caption{Auxostat-type control $u_{\sigma^*}$ for the optimal $\sigma$ value and the produced $s(t)$ values. }
  \label{fig:2-d}
\end{subfigure}

\caption{}
\end{figure}
Using the same parameter values as in the aforementioned section and a final time horizon of $T = 60$, we find that the optimum is attained for $\sigma^* = 1.4$, with $T(\sigma^*) = 13.7$, as shown in Figure~\ref{Fig:2-a}. For this optimal value, Figure~\ref{Fig:2-b} displays the mean half-saturation functional at each time, given by $\opk{f(t,\cdot)}$, where $f$ denotes the solution to~\eqref{main} corresponding to $\sigma^* = 1.4$. We observe that, for the threshold $k_0 = 1.5$, the solution reaches the target set and remains within it until the final time, indicating that the entry time can be considered unique. 
Figure~\ref{fig:2-c} shows the final biomass distribution associated with this optimal solution. The distribution closely resembles the expected equilibrium profile, suggesting convergence toward a stable state. Recall from Lemma~\ref{lem:autoxat-s-et-m} that, under auxostat-type control, the substrate concentration $s(t)$ converges to $\sigma$. This behavior is illustrated in Figure~\ref{fig:2-d}, where the control rapidly drives the substrate trajectory toward $\sigma^*$.

\subsubsection{Time minimization among constant controls}
Similarly as before, we conducted an extensive search over 800 values of $u \in (0,8)$ to determine the first entry time into the target set for solutions of system~\eqref{main} under constant control. This approach also allows us to identify an optimal pair $(u^*, T(u^*))$ that minimizes the entry time. The primary motivation for this analysis is to compare the auxostat-type control strategy with a simpler control formulation. Indeed, constant controls are easier to implement in practical experiments, making it relevant to assess their performance relative to more sophisticated strategies such as the auxostat-type. 
Constant controls are also relevant, as they are connected to the eventual turnpike property of the system \cite{trelat2015turnpike}. 
\begin{figure}[htbp!]
\centering
\begin{subfigure}{0.48\textwidth}%R_u.pdf=fig6.pdf
  \includegraphics[width=\linewidth]{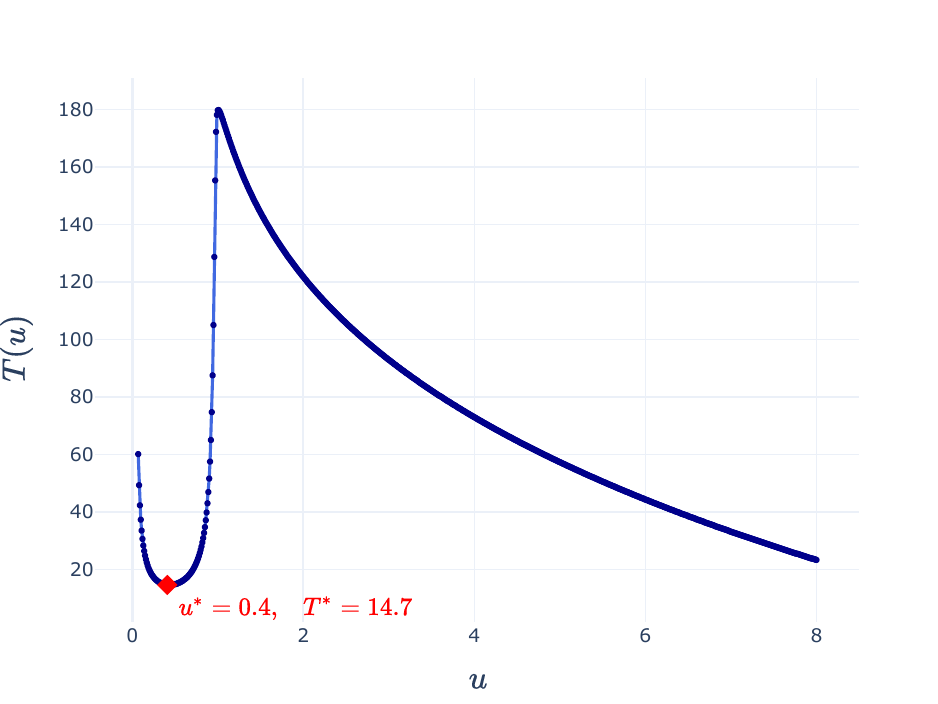}
  \caption{Entry time $T(u)$ to the target set for $u\in(0,8)$ and solutions of \eqref{main} with constant control $u$. }
  \label{Fig:3-a}
\end{subfigure}\hfill
\begin{subfigure}{0.48\textwidth}%Kt_u.pdf=fig7.pdf
  \includegraphics[width=\linewidth]{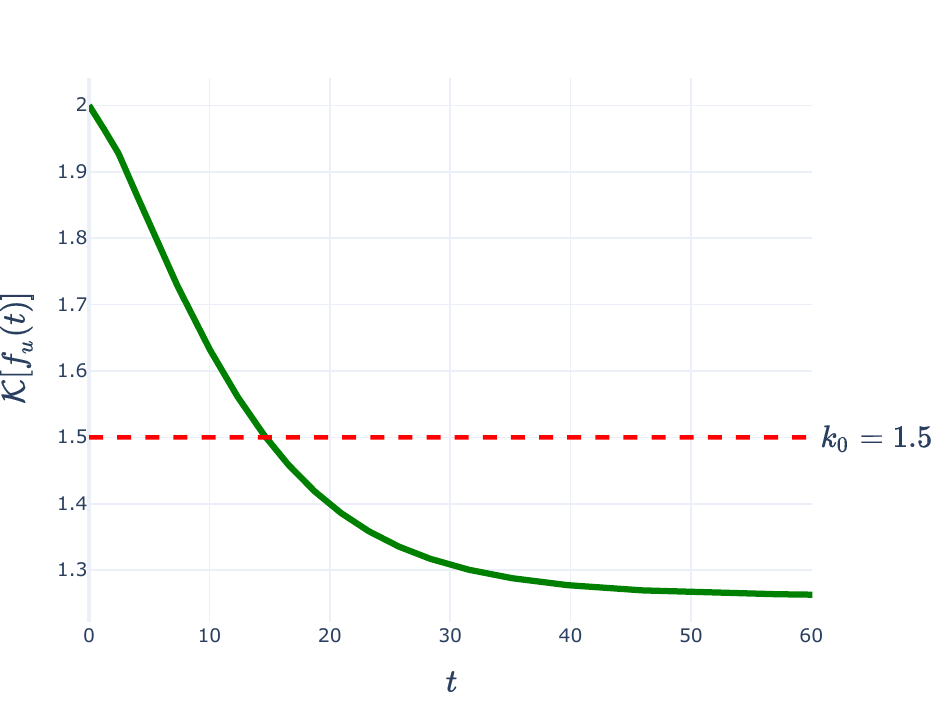}
  \caption{Function $\opk{f(t,\cdot)}$ evaluated on the solution of \eqref{main} at the optimal value $u^*=0.4$, for each time.}
  \label{Fig:3-b}
\end{subfigure}

\medskip

\begin{subfigure}{0.48\textwidth}%fT_u.pdf=fig8.pdf
  \includegraphics[width=\linewidth]{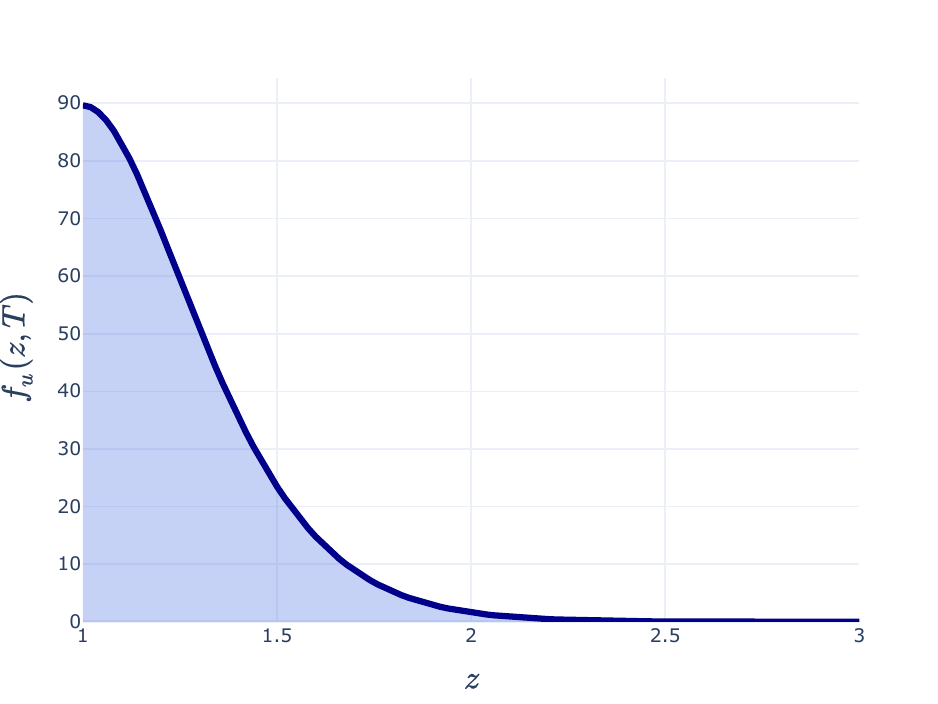}
  \caption{Final distribution of the solution of \eqref{main} for constant control using $u=0.4$. }
  \label{fig:3-c}
\end{subfigure}\hfill
\begin{subfigure}{0.48\textwidth}%u-s-R_u.pdf=fig9.pdf
  \includegraphics[width=\linewidth]{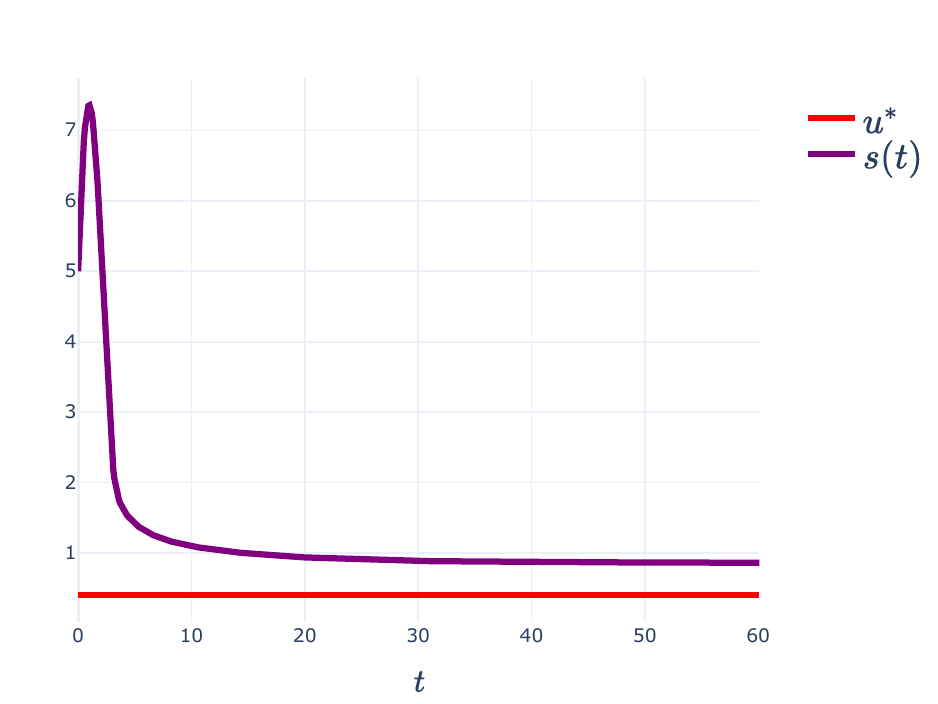}
  \caption{Constant control $u=u^*$ for the optimal value and the produced $s(t)$ values. }
  \label{fig:3-d}
\end{subfigure}

\caption{}
\end{figure}
Using the same parameter values as in the previous section and a final time horizon of $T = 60$, we find that the optimum is attained for $u^* = 0.4$, with an entry time of $T(u^*) = 14.7$, as shown\footnote{Even though another minimum may exist for $u>8$, it was not considered since it leads to washout.} in Figure~\ref{Fig:3-a}. For this optimal value, Figure~\ref{Fig:3-b} displays the evolution of the mean half-saturation functional, given by $\opk{f(t,\cdot)}$, where $f$ denotes the solution of~\eqref{main} corresponding to $u^* = 0.4$. We observe that, for the threshold $k_0 = 1.5$, the solution reaches the target set and remains within it until the final time, indicating that the entry time can be considered unique. 
Although convergence to an equilibrium was not analytically addressed for constant controls, Figures~\ref{fig:3-c} and~\ref{fig:3-d} suggest that the final biomass distribution closely resembles the expected equilibrium profile and that the substrate concentration $s(t)$ converges to a value of approximately $0.86$, indicating convergence toward a stable state (although convergence of solutions to \eqref{main} in this regime is, at present, unclear).
We conclude that  the minimal entry times obtained for constant controls are comparable to those achieved with the auxostat-type formulation.

\smallskip
\section{Concluding remarks and perspectives}{\label{conclusion}}

In this paper, we have investigated qualitative properties of the solutions to a chemostat system with trait structure (in the spirit of \cite{Perthame2007}) incorporating an interspecies exchange term (see \cite{alvarez2026global,bayen2023stability}). These properties have been studied from the perspective of optimal control. In particular, we established the following results:
\begin{itemize}
\item[$\bullet$] Well-posedness of the input–output map for a given measurable control (Theorem \ref{th-existence});
\item[$\bullet$] Stabilization of the system around a desired stationary solution via an auxostat-type control (Theorem \ref{th:stabi-other});
\item[$\bullet$] Analysis of a minimal-time problem for species selection: attainability of the target and existence of an optimal control 
(Proposition \ref{prop:controlability} and Theorem \ref{thm-bigBoss}). 
\end{itemize}
As concerns the attainability result, we obtained a property stronger than simple reachability of the target: namely, we have shown that the target is not only attainable, but also that the underlying system can be kept within it thereafter.
The originality of our work lies in introducing control questions within this framework. Indeed, considering a phenotypic trait makes these problems significantly more challenging to address than in the classical case of a finite number of species.
A natural extension of our stabilization results is about the asymptotic behavior of solutions to \eqref{main} in the case of a constant dilution rate.  
This question of significant interest could be addressed whenever the mutation rate goes to zero (in the spirit of \cite{alvarez2026global,bayen2023stability} and as an extension of the competitive exclusion principle in \cite{Perthame2007}). Since it is quite involved, it is left for further investigation. As a continuation of the minimum-time control problem, a second major question concerns the derivation of necessary optimality conditions for optimal control problems governed by \eqref{main}, considering for instance a general objective function.  
Because the population equation is non-local, system \eqref{main} does not fit precisely into the standard theory of parabolic control systems.
Hence, this natural question is more involved (particularly because of the particular structure of \eqref{main}) and lies beyond the scope of the present work. Finally, applying direct optimization methods to the underlying high-dimensional system (comprising several hundred species) is a major challenge and a complex problem that merits more in-depth future study.
\begin{comment}
\end{comment}
\color{black}
\section*{Acknowledgements}
The authors thank Francis Mairet and Olivier Cots for fruitful discussion on the subject. J\'er\^ome Coville is supported by ANR ReaCH-23-CE40-0023. J\'er\^ome Coville thanks Filippo Santambrogio, Thomas Lepoutre and L\'eon Matar Tine for fruitful discussions.

\appendix
\section{An Aubin-Lions type lemma in the space of finite Radon measures}
In this appendix, we provide a detailed proof of  \Cref{lem:radon-compact}. 
First, observe that for $t=0$ one has $\tilde f_n(0,\cdot)=f_0$ for all $n\in \mathbb{N}$, hence, $f_n(0,\cdot)\stackrel{\ast}{\rightharpoonup} \nu_0=f_0(z)\,dz$ as $n\to +\infty$.
Recall also  the following uniform bounds for all $\o\subset\O$:
\begin{align}
&\forall t\in(0,T], \; \; \forall n \in \mathbb{N}^*, \; \; \int_{\o}\tilde f_n(t,z)\,dz\le e^{C_0T}\int_{\o}f_0(z)\,dz \le K_1,\label{eq:radon-esti1}\\
&\forall t\in(0,T], \; \; \forall n \in \mathbb{N}^*, \; \; \left| \int_{\o}\partial_t \tilde f_n(t,z)\,dz\right|\le C_0K_1.\label{eq:radon-esti2}
\end{align}
Via the canonical identification of $L^1$ with a subset of 
$\M(\O)$, we denote, for each $t\in [0,T]$, by 
$\nu_n(t,z)$, the positive Radon measure corresponding to 
$\tilde f_n(t,z)dz$. 
The estimates then give
$$
\forall n \in \mathbb{N}, \; \; \forall t\in [0,T], \; \; \nu_n(t,\O)<K_1. 
$$
Next, for each fixed $t\in[0,T]$, the Banach-Alaoglu Theorem (see \cite{Brezis2010}) implies that a subsequence of  $(\nu_n(t,\cdot))$  converges weak-$*$ to some bounded Radon measure 
$\nu_t\in \M(\O)$. So by using a diagonal extraction procedure, we can find a subsequence  $(\nu_{n_k})_{k\in \N}$ such that 
 $\nu_{n_k}(t_q,\cdot) \stackrel{\ast}{\rightharpoonup} \nu_{t_q}$ for any $t_q\in[0,T]\cap \Q$. 
The rest of the proof is divided into several steps. 
\smallskip

\noindent {\it{First step : weak-$*$ convergence of $(\nu_{n_k}(t_q,\cdot))_{k\in\N}$ on indicator functions of Borel sets}}. 
Our goal is to prove that for every Borel set $\o\subset \Omega$, one has:
$$\lim_{k \rightarrow +\infty} \, \int_{\O}\mathds{1}_{\o}(z)\,d\nu_{n_k}(t_q,z) = \int_{\O}\mathds{1}_{\o}(z)\,d\nu_{t_q}(z). $$
Let $\eps>0$. Since $\nu_{t_q}$ is a Radon measure, for all $\delta>0$, there exist a compact set $F_\delta\subset \R^d$ and an open set $G_\delta\subset \R^d$ such that 
$F_\delta \subset \o \subset G_\delta$ and satisfying
$$|G_\delta\setminus F_\delta|\le \delta \quad\text{ ; } \quad \nu_{t_q}(G_\delta\setminus F_\delta)\le \delta.$$
Observe that $\O\setminus G_\delta$ is a closed subset of $\O$ and that $F_\delta \cap (\O\setminus G_\delta)=\emptyset$. Since $\O$ is separable, by the Uryshon Lemma, there exists 
a continuous function $\varphi_{\delta}:\Omega \rightarrow [0,1]$
such that 
$$
\varphi_\delta(x)=
\left\{
\begin{array}{rll}
1 & \mathrm{if} & x\in F_\delta,\\
0 & \mathrm{if} & x\in \Omega\setminus G_\delta.
\end{array}
\right.
$$
Now by using that  $f_n\ge 0$ and that $ \varphi_{\delta}\mathds{1}_{\o}\ge \varphi_\delta$, we deduce that 
\begin{align*}
    \int_{\O}\mathds{1}_{\o}(z)\,d\nu_{n_k}(t_q,z)&=\int_{\O}\mathds{1}_{\o}(z)\varphi_{\delta}(z)\,d\nu_{n_k}(t_q,z) +   \int_{\O}\mathds{1}_{\o}(z)(1-\varphi_{\delta}(z))\,d\nu_{n_k}(t_q,z),\\
    &\ge \int_{\O}\mathds{1}_{\o}(z)\varphi_{\delta}(z)\,d\nu_{n_k}(t_q,z),\\
    &\ge \int_{\O}\varphi_{\delta}(z)\,d\nu_{n_k}(t_q,z).
\end{align*} 
Therefore, one has
\begin{equation}\label{eq:radon-liminf}
   \liminf_{k\to +\infty}  \int_{\O}\mathds{1}_{\o}(z)d\nu_{n_k}(t_q,z) \ge \int_{\O}\varphi_{\delta}(z)\,d\nu_{t_q}(z)\ge \nu_{t_q}(F_\delta)\ge \nu_{t_q}(\o) -\delta.
\end{equation}
On the other hand, by using \eqref{eq:radon-esti1}, we have 
\begin{align*}
\int_{\O}\mathds{1}_{\o}(z)\,d\nu_{n_k}(t_q,z)&=
    \int_{\O}\mathds{1}_{F_\delta}(z)\,d\nu_{n_k}(t_q,z) +   \int_{\O}\mathds{1}_{\o}(z)(1-\mathds{1}_{F_\delta}(z))\,d\nu_{n_k}(t_q,z),\\
    &\le \int_{\O}\mathds{1}_{F_\delta}(z)\,d\nu_{n_k}(t_q,z) + \int_{\o\setminus F_\delta} \,d\nu_{n_k}(t_q,z)\\
    &\le \int_{\O}\mathds{1}_{F_\delta}(z)\,d\nu_{n_k}(t_q,z) + e^{C_0T}\int_{\o\setminus F_\delta} f_{0}\,dz.
    \end{align*}
    Consequently, since $\varphi_\delta\mathds{1}_{F_\delta}\le \varphi_\delta$, we find that 
\[
    \int_{\O}\mathds{1}_{\o}(z)\,d\nu_{n_k}(t_q,z)\le \int_{\O}\varphi_\delta(z)\,d\nu_{n_k}(t_q,z) + e^{C_0T}\int_{\o\setminus F_\delta} f_{0}\,dz.\]
This yields
\[  \limsup_{k\to +\infty}  \int_{\O}\mathds{1}_{\o}(z)d\nu_{n_k}(t_q,z) \le \int_{\O}\varphi_{\delta}(z)\,d\nu_{t_q}(z)  + e^{C_0T}\int_{\o\setminus F_\delta} f_{0}\,dz.\]
   By using that $\varphi_\delta \equiv 0$ over $\O\setminus G_\delta$, we deduce that $\ds{\int_{\O}\varphi_{\delta}(z)\,d\nu_{t_q}(z)\le \nu_{t_q}(G_\delta)}$ and therefore
   \begin{equation}\label{eq:radon-limsup}
   \limsup_{k\to +\infty}  \int_{\O}\mathds{1}_{\o}(z)d\nu_{n_k}(t_q,z) \le \nu_{t_q}(\o) +\delta + e^{C_0T}\int_{\o\setminus F_\delta} f_{0}\,dz.
\end{equation}
   Since the function $f_0$ is uniformly integrable (it is in $L^1$),  there is 
   $\delta_0>0$ such that for all set $A$,  
   $$|A| \leq \delta_0 \quad \Rightarrow \quad \ds{\int_{A}f_0(z)dz\le \frac{\eps}{2 e^{C_0T}}}.$$ 
  Take $\delta \le \min(\frac{\eps}{2},\delta_0)$. From \eqref{eq:radon-liminf}-\eqref{eq:radon-limsup}, we get that  
   $$ \nu_{t_q}(\o) -\eps\le  \liminf_{k\to +\infty}  \int_{\O}\mathds{1}_{\o}(z)d\nu_{n_k}(t_q,z) \le \limsup_{k\to +\infty}  \int_{\O}\mathds{1}_{\o}(z)d\nu_{n_k}(t_q,z)\le \nu_{t_q}(\o) + \eps. $$
  Since $\eps>0$ is arbitrary, we conclude that 
   $$ \nu_{t_q}(\o) \le  \liminf_{k\to +\infty}  \int_{\O}\mathds{1}_{\o}(z)d\nu_{n_k}(t_q,z) \le \limsup_{k\to +\infty}  \int_{\O}\mathds{1}_{\o}(z)d\nu_{n_k}(t_q,z)\le \nu_{t_q}(\o).$$ 
   This concludes the first step. 

\smallskip

\noindent {\it{Second step: Regularity of the measure}}. From \eqref{eq:radon-esti2} and from the previous step, for any $t_q,t_{q'} \in [0,T]\cap\Q$,  we have for any $\o\subset \O$:
\begin{align*}
|\nu_{t_q}(\o)-\nu_{t_{q'}}(\o)|&\le \left|\lim_{k\to +\infty}\left(\int_{\o} \tilde f_{n_k}(t_{q},z)-\tilde f_{n_k}(t_{q'},z) \,dz\right)\right| \\ 
&\le|t_{q}-t_{q'}| \left|\lim_{k\to +\infty}\left(\int_{\o}\int_{0}^1 \partial_t\tilde f_{n_k}(t_{q'}+s(t_{q}-t_{q'}),z) \,dzds\right) \right|\\
&\le   K_1C_0|t_{q}-t_{q'}|. 
\end{align*}
 As a result, for any $t\in[0,T]$, for any sequence $(t_{q_n}) \in ([0,T]\cap \Q)^\mathbb{N}$ such that $t_{q_n}\to t$ and for any $\o\subset\O$, $(\nu_{t_{q_n}}(\o))_n$ is Cauchy.  
Therefore, it converges to a unique limit.   
Moreover, we can define a measure on the Borel $\sigma$-algebra of $\O$ as follows: 
 $$ \nu_t =\sup_{t_q\le t}\nu_{t_q}. $$
 Since $(\nu_{t_q})_q$ are positive bounded Radon measured and the set $\Q\cap [0,T]$ is dense in $[0,T]$, we obtain that  $\nu_t$ is a positive bounded Borel measure over $\O$. 
\smallskip

\noindent {\it{Third step: Weak$-\ast$ convergence to $\nu_t$}}. 
To conclude, we need to prove that for all $t\in[0,T]$,  $\nu_{n_k}(t,\cdot)\stackrel{\ast}{\rightharpoonup} \nu_t$ as $k\to +\infty$.
 Since $(\M(\O),\mathrm{weak-}*)$ is a complete separable metric space,  we start by showing  that for any $t\in [0,T]$, the sequence $(\nu_{n_k}(t,\cdot))_{k\in \N}$ is Cauchy in $(\M(\O),\mathrm{weak-}*)$.   Doing so,  given $\varphi \in C(\O)$, we need to prove that the sequence $(\int_{\O}\varphi(z)\,d\nu_{n_k}(t,z))_{k\in\N}$
is Cauchy. Let $\eps>0$. First, since $\bar \O$ is a compact subset of $\R^d$, by the Stone-Weierstrass theorem \cite{Rudin1987}, there exist finitely many $\o_i\subset \O$ and $\gamma_i\in\R$ such that 
$\|\varphi -\sum_{i}\gamma_i\mathds{1}_{\o_i}(\cdot)\|_{\infty}<\frac{\eps}{4K_1}$.
Set $\varphi_\eps:= \sum_{i}\gamma_i\mathds{1}_{\o_i}(\cdot)$. Then, for any $n_k, n_{k'}$, by using the definition of the measure $\nu_{n_k}(t,z)$, we have  
\begin{multline*}
\left|\int_{\O}\varphi(z)d\nu_{n_k}(t,z) -\int_{\O}\varphi(z)\,d\nu_{n_{k'}}(t,z)\right| \le \|\varphi -\varphi_\eps\|_{\infty}\left(\int_{\O}d\nu_{n_k}(t,z)+\int_{\O}d\nu_{n_{k'}}(t,z)\right)\\ + \left|\int_{\O}\varphi_\eps(z)[\tilde f_{n_k}(t,z) -\tilde f_{n_{k'}}(t,z)]\,dz\right|.
\end{multline*}
Using \eqref{eq:radon-esti1}, we obtain
\begin{equation}\label{eq:radon-esti-phi}
\left|\int_{\O}\varphi(z)d\nu_{n_k}(t,z) -\int_{\O}\varphi(z)\,d\nu_{n_{k'}}(t,z)\right| \le
\frac{\eps}{2}+ \left|\int_{\O}\varphi_\eps(z)[\tilde f_{n_k}(t,z) -\tilde f_{n_{k'}}(t,z)]\,dz\right|.
\end{equation}
Using the  density of  $\Q\cap[0,T]$  in $[0,T]$, let $t_{q}\in(0,T]\cap\Q$ be such that $|t-t_{q}|<\frac{\eps}{6C_0K_1\sum_{i}|\gamma_i|}$.
Thanks to the triangular inequality, we deduce that for all $n_k,n_{k'}$,   
\begin{multline*}
\left| \int_{\O}\varphi_{\eps}(z)(\tilde f_{n_k}(t,z)-\tilde f_{n_{k'}}(t,z))dz\right|\le \left| \int_{\O}\varphi_\eps(z)(\tilde f_{n_k}(t,z)-\tilde f_{n_{k}}(t_q,z))dz\right| \\ \qquad \qquad \qquad \qquad \qquad+\left| \int_{\O}\varphi_\eps(z)(\tilde f_{n_k}(t_q,z)-\tilde f_{n_{k'}}(t_q,z))dz\right|\\+\left| \int_{\O}\varphi_{\eps}(z)(\tilde f_{n_{k'}}(t_q,z)-\tilde f_{n_{k'}}(t,z))dz\right|.
\end{multline*}
From the definition of $\varphi_\eps$, \eqref{eq:radon-esti2}, and the choice of $t_q$, 
the first integral in the  above right hand side satisfies
\begin{align} 
\left| \int_{\O}\varphi_\eps(z)(\tilde f_{n_k}(t,z)-\tilde f_{n_{k}}(t_q,z))dz\right|&=\left|\sum_{i} \gamma_i \int_{\o_i}(\tilde f_{n_k}(t,z)-\tilde f_{n_{k}}(t_q,z))dz\right|\nonumber\\  
&\le \sum_{i} |\gamma_i|  \left|\int_{\o_i}(\tilde f_{n_k}(t,z)-\tilde f_{n_{k}}(t_q,z))dz\right|\nonumber\\
&\le \sum_{i} |\gamma_i| |t-t_q|  \left|\int_{0}^1\int_{\o_i}\partial_t\tilde f_{n_k}(t_q+s(t-t_q),z)dzds\right| \nonumber\\
&\le \sum_{i} |\gamma_i| |t-t_q|  C_0K_1=\frac{\eps}{6}. \label{eq:radon-esti-int1}
\end{align}
Similarly, we get 
\begin{equation}\label{eq:radon-esti-int2}
\left| \int_{\O}\varphi_\eps(z)(\tilde f_{n_{k'}}(t_q,z)-\tilde f_{n_{k'}}(t,z))dz\right|\le \frac{\eps}{6}.
\end{equation}
Finally, recall that $\nu_{n_k}(t_q,\cdot) \stackrel{\ast}{\rightharpoonup} \nu_{t_q}$, for indicator functions (first step). As a result, 
the sequence $\left(\int_{\O}\varphi_\eps(z)\tilde f_{n_k}(t_q,z)dz\right)_{k\in\N}$ is  Cauchy , and there exists $N>0$ such that for all $n_{k},n_{k'}>N$:%, we have 
\begin{equation}\label{eq:radon-esti-int3}
    \left| \int_{\O}\varphi_\eps(z)(\tilde f_{n_k}(t_q,z)-\tilde f_{n_{k'}}(t_q,z))dz\right|\le \frac{\eps}{6}. 
\end{equation}
Combining \eqref{eq:radon-esti-phi}-\eqref{eq:radon-esti-int1}-\eqref{eq:radon-esti-int2}-\eqref{eq:radon-esti-int3}, we deduce that for all $n_k, n_{k'}> N$,
\[\left|\int_{\O}\varphi(z)d\nu_{n_k}(t,z) -\int_{\O}\varphi(z)\,d\nu_{n_{k'}}(t,z)\right|\le
\frac{\eps}{2}+ \frac{\eps}{6}+ \frac{\eps}{6}+ \frac{\eps}{6}=\eps.\]
At this step, the sequence $(\int_{\O}\varphi d\nu_{n_k}(t,z))_{k\in\N}$ converges, but, we need  to identify its limit. From the Stone-Weierstrass Theorem \cite{Rudin1987}  and the above argumentation, we only need to identify this limit when $\varphi$ is an indicator function of a Borel set. This amounts to prove that for any Borel set $\o$, $\nu_{n_k}(t,\o)\to \nu_{t}(\o)$ as $k\rightarrow +\infty$. 
Again, let $\o$ be a Borel set and let $\eps>0$.  By definition of $\nu_t$ (second step), there is a sequence  $(t_{q_i})_{i\in\N}$ such that $t_{q_i}\to t$ and  $ \nu_{t_{q_i}}(\o)\le \nu_t(\o)\le \nu_{t_{q_i}}(\o)+\frac{\eps}{3}$ for all $i\ge0$. Since $t_{q_i}\to t$ as $i\to+\infty$, by using \eqref{eq:radon-esti2}, we can find $i_0$, such that $ \left| \int_{\o}d\nu_{n_k}(t,z)-\int_{\o}d\nu_{n_k}(t_{q_{i_0}},z)\right|\le \frac{\eps}{3}$  for all $n_k\in \mathbb{N}$. The triangular inequality then implies
\[
\left| \int_{\o}d\nu_{n_k}(t,z)-\int_{\o}d\nu_t(z)\right|\le \frac{2\eps}{3} + \left| \int_{\o}d\nu_{n_k}(t_{q_{i_0}},z)- \int_{\o}d\nu_{t_{q_{i_0}}}(z)\right|.
\]
We  conclude by using the first step and the fact that $\nu_{n_k}(t_{i_0},\o)\to \nu_{t_{i_0}}(\o)$ as $k\to+\infty$.

\section*{}
\bibliographystyle{amsplain}
\bibliography{bibliography.bib}

\end{document}